\documentclass[12pt]{article}
\usepackage[vmargin=1.18in,hmargin=1.1in]{geometry}				
\usepackage{amssymb,amsfonts,amsthm,mathtools,bm,bbm,mathrsfs} 	
\usepackage{pifont} 											
\usepackage[dvipsnames]{xcolor} 								
\usepackage[colorlinks=true, citecolor=blue, urlcolor=blue]{hyperref}  
\usepackage{graphicx} 											
\usepackage{caption} 											
\usepackage{verbatim}
\usepackage[misc]{ifsym} 										
\usepackage[ruled, linesnumbered]{algorithm2e} 					
\usepackage{layouts} 
\usepackage[uniquename=false, 
            uniquelist=false, 
            url=false, 
            doi=false, 
            isbn=false, 
            natbib=true, 
            backend=bibtex, 
            style=authoryear, 
            maxcitenames=1,
            maxbibnames=5]{biblatex}                            
\usepackage{booktabs} 
\usepackage{multirow} 
\usepackage[noend]{algpseudocode}
\usepackage{subcaption}
\usepackage{multirow}
\usepackage{authblk}
\usepackage{chngcntr}
\usepackage{tocloft}

\makeatother

\newcommand*\tightdisplay{%
  \setlength{\abovedisplayskip}{2pt}
  \setlength{\belowdisplayskip}{2pt}
}

\newcommand{\zr}[1]{}
\newcommand{\edit}[1]{#1}

\newtheorem{assumption}{Assumption}
\newtheorem{theorem}{Theorem}
\newtheorem{corollary}{Corollary}
\newtheorem{remark}{Remark}
\newtheorem{lemma}{Lemma}

\newtheorem{definition}{Definition}
\newcommand{\indep}{\perp \!\!\! \perp}

\def\P{\mathbb{P}}

\def\P{\mathbb{P}}

\newcommand{\E}{\mathbb E}								
\renewcommand{\P}{\mathbb{P}}							
\newcommand{\R}{\mathbb{R}}								
\newcommand{\indicator}{\mathbbm 1}						
\newcommand{\convp}{\overset p \rightarrow}             
\newcommand{\convd}{\overset d \rightarrow}             
\newcommand{\convdp}{\overset {d,p} \longrightarrow}    

\newcommand{\WAIPW}{\mathrm{WAIPW}}

\newcommand{\WIPW}{\mathrm{WIPW}}
\newcommand{\WIPWS}{\mathrm{WIPWS}}

\newcommand{\convpp}{\overset {p,p} \longrightarrow}    
\let\oldnl\nl
\newcommand{\nonl}{\renewcommand{\nl}{\let\nl\oldnl}} 

\allowdisplaybreaks

\title{Assumption-lean weak limits and tests \\
for two-stage adaptive experiments}

\begin{document}

\author{Ziang Niu}
\author{Zhimei Ren}
\affil{Department of Statistics and Data Science, University of Pennsylvania}
\maketitle

\begin{abstract}
	Adaptive experiments are becoming increasingly popular in real-world applications for effectively maximizing in-sample welfare and efficiency by data-driven sampling. Despite their growing prevalence, however, the statistical foundations for valid inference in such settings remain underdeveloped. Focusing on two-stage adaptive experimental designs, we address this gap by deriving new weak convergence results for mean outcomes and their differences. In particular, our results apply to a broad class of estimators, the \textit{weighted inverse probability weighted} (WIPW) estimators. In contrast to prior works, our results require significantly weaker assumptions and sharply characterize phase transitions in limiting behavior across different signal regimes. Through this common lens, our general results unify previously fragmented results under the two-stage setup. \edit{We further establish quantitative convergence rates in bounded-Lipschitz distance that reveal the fundamental trade-off between exploitation and inferential stability.} To address the challenge of potential non-normal limits in conducting inference, we propose a computationally efficient and provably valid \edit{simulation-based} method for \edit{obtaining critical values of the non-normal limiting distributions under the null, enabling practical} hypothesis testing. Our results and approaches are sufficiently general to accommodate various adaptive experimental designs, including batched bandit and subgroup enrichment experiments. \edit{Simulations and semi-synthetic studies demonstrate the practical value of our approach and reveal that neither normality-based nor non-normality-based testing methods uniformly dominate in power; the relative advantage depends on the structure of the outcome distribution.}
\end{abstract}

\section{Introduction}

Adaptive experiments are able to achieve substantial efficiency gains compared with traditional non-adaptive experimental designs. They often allocate resources more effectively and require fewer samples or observations to attain the same statistical power or estimation precision. Such designs have been successfully applied in areas such as clinical trials~\citep{sampson2005drop,hu2006theory,magnusson2013group}, online learning~\citep{slivkins2019introduction,lattimore2020bandit}, mobile health interventions~\citep{klasnja2019efficacy,liao2020personalized}, and online education platforms~\citep{rafferty2019statistical,kizilcec2020scaling}.

However, the adaptive design of these experiments introduces dependencies among observations, violating the independence and identical distribution (i.i.d.) assumptions underlying classical inference methods. As a result, widely used estimators—such as the sample mean and inverse probability weighted estimator—may exhibit bias and non-normal sampling distributions under adaptive data collection~\citep{bowden2017unbiased,shin2019sample,Hadad2021,shin2021bias}. In practice, analyzing adaptive experiments using conventional statistical tools while ignoring the dependencies 
can lead to severe selection bias~\citep{dwork2015reusable}. The limited theoretical understanding of the statistical behavior of adaptive experiments continues to hinder the development of valid and generalizable inference methods. This represents a major barrier to their reliable use in real-world applications.

In this paper, we study the \textit{two-stage adaptive experiments}, 
a design framework that has been widely adopted in practice~\citep{sampson2005drop,sladek2007genome,sill2009drop,wu2010interval,gasperini2019genome,lin2021inference,kasy2021adaptive,che2023adaptive,schraivogel2023pooled}. 
The typical data generating process in these experiments can be outlined as follows: 
there are two stages of data collection, the \textit{pilot stage} and the \textit{follow-up stage}. 
In the pilot stage, i.i.d.~data $\mathcal{D}_P \sim \P_P$ are collected, and informs a selection algorithm $\mathcal{S}(\mathcal{D}_P)$. 
Then in the follow-up stage, new data $\mathcal{D}_F \sim \P_F(\mathcal{D}_P)$ are gathered according to the output of selection algorithm $\mathcal{S}$, 
resulting in data that are conditionally i.i.d.\ given $\mathcal{D}_P$. Figure~\ref{fig:adaptive-experiment-illustration} illustrates the two-stage sampling process.
The choice of $\mathcal{S}$ depends on the goal of the experiment. Common objectives include welfare maximization~\citep{sampson2005drop,wu2010interval,che2023adaptive} and scientific exploration~\citep{sladek2007genome,gasperini2019genome}. 

\begin{figure}[!ht]
	\centering
	\includegraphics[width=.7\textwidth]{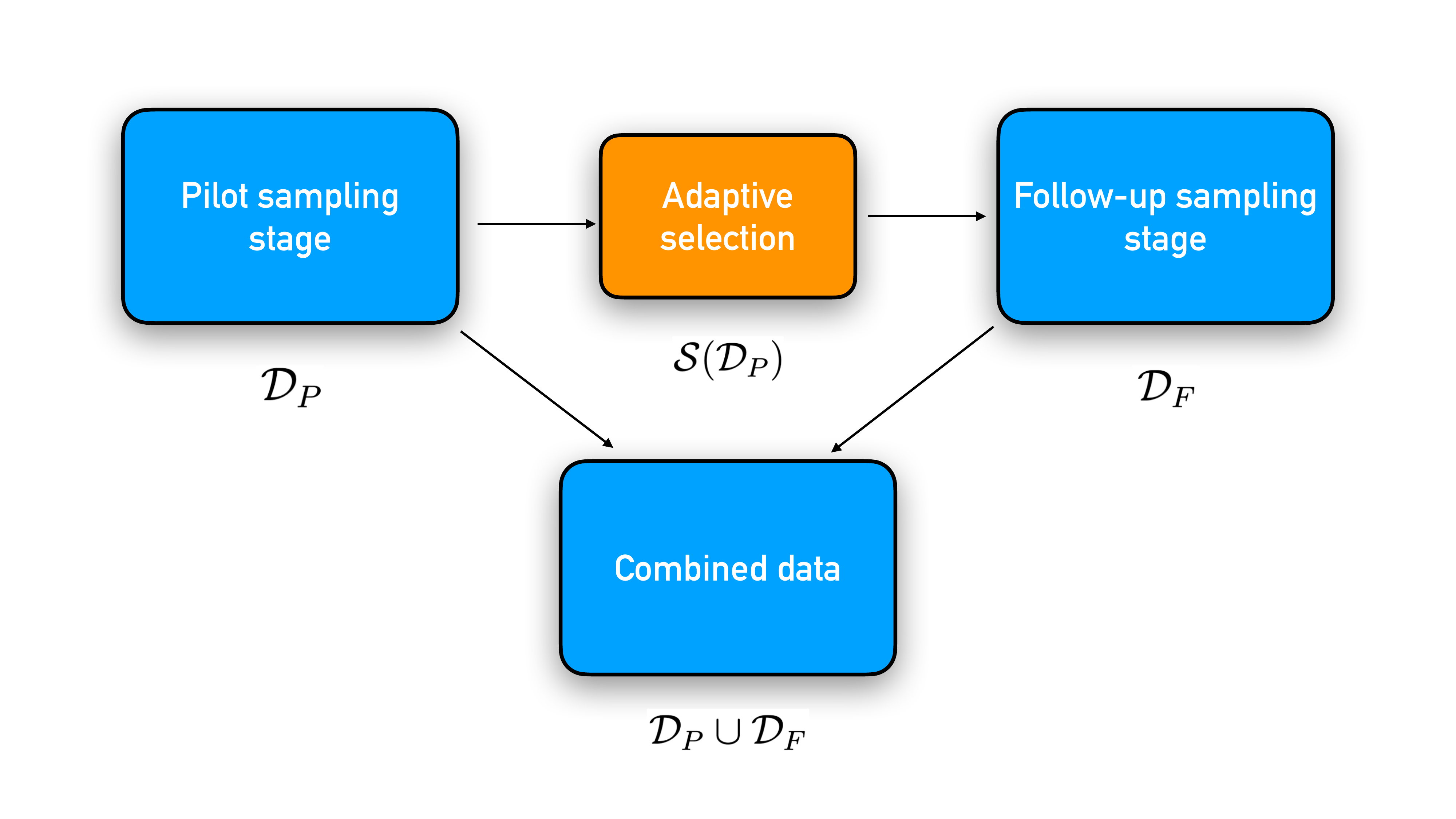}
    \caption{Illustration of two-stage experiment.}
	\label{fig:adaptive-experiment-illustration}
\end{figure}

The main challenge of conducting valid statistical inference with data $\mathcal{D}_P \cup \mathcal{D}_F$ is to handle the complex dependence structure introduced by the selection algorithm $\mathcal{S}$.

\subsection{Relevant literature}

Existing work on inference for adaptive experiments largely falls into two main categories, distinguished by the type of inference they offer.

\textit{Conditional inference} provides valid inference conditional on the output of selection algorithm $\mathcal{S}(\mathcal{D}_P)$. Approaches that achieve this guarantee include data splitting~\citep{cox1975note}, 
data carving~\citep{fithian2014optimal,chen2023optimal}, and randomization-based selective inference~\citep{freidling2024selective}. This type of guarantee captures the effect of the selection procedure and adjusts for conditional bias. However, when the estimand is a marginal parameter (e.g., outcome mean), conditional inference typically incurs an efficiency loss \citep{hu2006theory,marschner2021general}. Moreover, due to the potential complex conditioning event, methods developed for conditional inference can be computationally demanding, with the exception of data splitting. 

\textit{Marginal inference} accounts for all sources of randomness, leading to more straightforward interpretation when the inferential target is the marginal estimand. 
Towards this end, different approaches have been proposed, with finite-sample or asymptotic guarantees. 
Among the former, some seek to achieve exact finite-sample validity~\citep{sampson2005drop,sill2009drop,wu2010interval,neal2011interval,Nair2023}. 
However, these methods are relatively restrictive as they are either computationally intensive and/or are highly sensitive to distributional assumptions. 
Anytime-valid inference methods~\citep{johari2015always,howard2021time,howard2022sequential,maharaj2023anytime,ramdas2023game,waudby2024estimating} provide finite-sample validity via probabilistic bounds. These bounds, however, are usually conservative and can substantially reduce power (although they generalize well beyond two-stage settings as in Figure~\ref{fig:adaptive-experiment-illustration}). In contrast, asymptotic inferential methods~\citep{Zhang2020,Hadad2021,lin2021inference,adusumilli2023optimal,Hirano2023} tend to be statistically more efficient than conditional or anytime-valid approaches. Relying solely on large-sample behavior, asymptotic methods are generally agnostic to outcome distributions.

Our work falls into the category of marginal inference with asymptotic validity. The works of \citet{Zhang2020,Hadad2021,adusumilli2023optimal,Hirano2023} are the most related ones. 
\citet{Zhang2020,Hadad2021} prove asymptotic normality of outcome means under strong signal-strength or outcome-distribution assumptions, and their results extend beyond the two-stage adaptive experiments studied in this paper. \citet{Hirano2023} provide general representation of the limiting distribution of test statistics 
in a multi-stage setup, similar to that in Figure~\ref{fig:adaptive-experiment-illustration}. Later, \citet{adusumilli2023optimal} uses these representations to study the optimality of the tests under the same setup. Their results are built upon Le Cam's theory on limits of experiments~\citep{le1972limits}, 
which are valid only under contiguous alternatives and smooth (semi-)parametric outcome distributions. Despite the elegance of the results, the general representation requires verifying the existence of certain weak limits, which must be addressed on a case-by-case basis and therefore poses a barrier for practitioners. A more detailed comparison with the related literature is provided in Section~\ref{sec:phase_transition}.

From the inferential perspective, strong assumptions about signal strength and data generating distributions in this line of work can have significant practical consequences. 
Misusing the limiting distribution in hypothesis testing may yield Type-I error inflation (see an example in Figure \ref{fig:failure_Hadad} in Appendix \ref{sec:inspection_Victor}). Likewise, the stringent distributional requirements on $\mathbb{P}_{F}$ and $\mathbb{P}_{P}$ are often unrealistic when outcomes display complex or non-standard behavior. A vivid illustration arises in exploratory biological studies, such as single-cell CRISPR screens~\citep{dixit2016perturb,gasperini2019genome}. In these experiments, the main outcomes, gene-expression measurements, typically exhibit overdispersion, measurement error, and technical batch effects, all of which may violate the prespecified distributional assumptions. These caveats highlight the need for a robust inferential framework that can accommodate a broader range of signal strengths and remain agnostic to the outcome distribution.

Beyond inference, assumption-lean weak limits offer a compelling tool for experimental design and power analysis. Alternatively, simulation methods have been proposed to design new adaptive experiments~\citep[Chapter VII;][]{us2019adaptive}, which can bypass the derivation of limiting distributions. Simulation-based strategies typically rely on stringent parametric assumptions, are largely heuristic, and become computationally intensive when navigating a vast parameter space. In contrast, distribution-agnostic weak limits provide a far more efficient alternative—so long as one can sample from the limiting distribution effectively. Pursuing this direction, \citet{che2023adaptive} prove a joint weak-limit result for batched multi-armed bandits and leverage it to design batch-level allocation rules that minimize Bayes simple regret. 
Their framework, however, prioritizes regret minimization rather than hypothesis testing, 
and is not sufficient for  inference or
for designing experiments that aim to maximize statistical testing power. 

\subsection{Our contributions}\label{sec:contribution}

To address these gaps in asymptotic inference of the two-stage experiments, 
we study the asymptotic distribution of a broad class of \textit{weighted inverse probability weighted} (WIPW) statistics. 
This class includes several widely used statistics, such as the IPW statistic~\citep{bowden2017unbiased} and the variance-stabilizing IPW statistic~\citep{luedtke2016statistical,bibaut2021post,Hadad2021}. We establish weak convergence results under minimal assumptions. Building on this foundation, we propose a valid and computationally efficient simulation-based procedure for hypothesis testing in the presence of non-normal limiting null distributions. Specifically, our main contributions are summarized as follows.

\begin{enumerate}
	\item\textbf{Assumption-lean weak limits:} 
	We derive new weak convergence results for WIPW estimators in two-stage experimental settings. These results apply to a wide range of signal strengths under minimal distributional assumptions. 
	Such generality ensures valid inference across the null (zero signal), contiguous (weak signal), and fixed (strong signal) regimes, addressing key demands in hypothesis testing. 
	Our analysis also uncovers a smooth transition of the limiting distribution across signal regimes, 
	offering a unified perspective that connects several existing results in the literature. 
	\edit{Our general results can be readily applied to two real-world adaptive experimental designs:
	\emph{batched bandit experiments} and \emph{subgroup enrichment experiments}.
	Although these experiments arise in distinct scientific contexts,
	both can be naturally accommodated within our theoretical framework.
	The proofs of main results hinge on a set of probabilistic tools that explicitly account for the dependence induced by adaptive data collection.
	To this end, we develop new results on conditional normal approximation (Lemma~\ref{lem:CLT_BL})
	and conditional continuous mapping (Lemmas~\ref{lem:continuous_map_varying}),
	extending classical unconditional tools \citep{raivc2018multivariate}
	to our two-stage adaptive setting. Together, these ingredients yield a quantitative bounded-Lipschitz approximation---i.e., an explicit bound on the convergence rate
	(Theorem~\ref{thm:quantitative_CLT_W_N})---which reveals the trade-off between aggressive exploitation and inferential stability through the choice of the clipping rate. Our proof techniques align and generalize those employed in recent adaptive experiments literature (see Remark~\ref{rmk:CA-approach}).}
	\item\textbf{A fast simulation-based testing procedure:}
	Building on the general weak convergence results, we define a class of asymptotically valid tests using WIPW test statistics. The critical values in the tests are determined by quantiles of the non-normal limiting distributions under the null. To obtain the analytically intractable critical values, we propose a computationally efficient and provably valid simulation-based procedure to estimate the critical values. The procedure rests on the key insight that the derived weak limits can be expressed as a randomly weighted sum of dependent Gaussian random variables. By leveraging the new simulation-based procedure, valid and practical hypothesis tests can be conducted despite the complexity of the limiting distribution. Importantly, the procedure is nonparametric and thus is agnostic to the outcome distribution. Moreover, it has time complexity that is independent of the sample size (conditional on estimated nuisance parameters), making it highly scalable. \edit{We also prove a set of new results on the rate of convergence of Type-I error towards the nominal level using the quantitative CLT approximation~(Corollary~\ref{cor:type-I-error-rate}).}
	\edit{\item\textbf{Numerical benchmarking with synthetic and semi-synthetic data:} We conduct extensive numerical simulations and a semi-synthetic data analysis based on the Systolic Blood Pressure Intervention Trial (SPRINT)~\citep{ambrosius2014design} to benchmark the power of nine testing methods: six WIPW simulation-based tests (with two scaling and three weighting schemes), sample-splitting IPW, the Batched Difference-in-Means (BDM) test based on normal limits, and a concentration-inequality-based test. The results demonstrate the practical utility of our methods in realistic settings. Our results reveal that neither the normality-based BDM test nor the non-normal limit tests uniformly dominate the other in power; the relative advantage depends on the structure of the data generating process. Under discrete outcomes, the BDM test can be more powerful, while under continuous outcomes, the WIPW tests with adaptive weighting achieve higher power.}
\end{enumerate}

Moreover, our results can be readily applied to the design of adaptive experiments, as they offer a clear understanding of the limiting behavior of the test statistics. Code to
reproduce these analyses is available at \url{https://github.com/ZiangNiu6/AdaInf-manuscript}.

\subsection{Organization of the paper}

Section~\ref{sec:warmup} introduces the two-stage adaptive data collection procedure and the WIPW test statistic. 
In Section~\ref{sec:weak_convergence_WIPW}, 
we present the formal results on weak convergence and the simulation-based methodology, instantiate our general theory in various adaptive experiments, 
and establish the connection to existing works. 
In Section \ref{sec:finite-sample}, we evaluate the finite-sample performance of the derived tests. We conclude the paper with a discussion in Section \ref{sec:discussion}.

\section{Data generating procedure and test statistic}\label{sec:warmup}

\subsection{Two-stage adaptive data collection}\label{sec:data_collection_selection}

We denote the sample sizes for the pilot and follow-up stages as \( N_1 \) and \( N_2 \), respectively, and treat them as given. The total sample size is defined as \( N \equiv N_1 + N_2 \), and the sample size ratio for the two stages is fixed as \( q_t \equiv N_t / N \in (0, 1) \) for \( t \in \{1, 2\} \). Throughout this paper, we adopt the triangular array framework, allowing the distribution to vary with \( N \). To emphasize this dependence, we use the subscript \( N \) when defining the random variables. Also, we define \([I] \equiv \{1, \ldots, I\}\) for any integer \(I \geq 1\).

In our setup, there are two competing treatments indexed by \( 0 \) and \( 1 \). Let \( A \in \{0,1\} \) denote the assigned treatment. Suppose $(A_{uN}^{(t)}, Y_{uN}^{(t)})_{u \in [N_t]}$ denotes the observed data at stage~$t$, where $Y_{uN}^{(t)}$ is the observed outcome corresponding to assigned treatment $A_{uN}^{(t)}$. Let $\mathcal{H}_t = \sigma\big((A_{uN}^{(t)}, Y_{uN}^{(t)})_{u \in [N_t]}\big)$ be the $\sigma$-algebra generated by the observed data at stage $t$. Additionally, define \(\mathcal{H}_0 \equiv  \{\varnothing, \Omega\}\). Adopting the potential outcome framework, for \( N \) subjects, 
the potential outcomes are denoted as $\{(Y_{uN}^{(t)}(0), Y_{uN}^{(t)}(1)) : t \in [2], u \in [N_t]\}$. They are independently and identically distributed as \( (Y_{uN}(0), Y_{uN}(1)) \) for any fixed \( N \). To identify the distribution of potential outcome variables, we assume the following \textit{consistency} and \textit{unconfoundedness} conditions throughout this paper. 
\begin{itemize}
    \item \emph{Consistency:}
    $Y_{uN}^{(t)} = A_{uN}^{(t)}Y_{uN}^{(t)}(1) + (1 - A_{uN}^{(t)})Y_{uN}^{(t)}(0), \quad u \in [N_t], \ t \in [2]$;
    \item \emph{Unconfoundedness:} 
    $(Y_{uN}^{(t)}(0), Y_{uN}^{(t)}(1)) \indep A_{uN}^{(t)} \mid \mathcal{H}_{t-1}, \quad u \in [N_t], \ t \in [2]$.
\end{itemize}
These are two assumptions that are commonly made in the literature of causal inference~\citep{imbens2015causal}. The consistency assumption states that the observed outcome is equal to the potential outcome under the assigned treatment. The unconfoundedness assumption states that the potential outcomes are independent of the treatment assignment, given the information from previous stages. Now we describe the observed data generating procedure. 

\begin{enumerate}
    \item \textbf{Pilot stage:} 
    In the pilot stage, we observe \( (A_{uN}^{(1)}, Y_{uN}^{(1)}) \) for \( u \in [N_1] \), with treatment assignment probabilities \( e(s) \equiv \P[A_{uN}^{(1)} = s] \) for \( s \in \{0, 1\} \), where \( e(0) + e(1) = 1 \).  A selection algorithm \( \mathcal{S} \) determines treatment assignment for the follow-up stage. 
	For an estimator $S_N^{(1)}(0)-S_N^{(1)}(1)$ for the difference-in-means $\E[Y_{uN}(0)]-\E[Y_{uN}(1)]$, the sampling probabilities are then updated based on $\mathcal{S}(S_N^{(1)}(0)-S_N^{(1)}(1))$.
	
    \item \textbf{Follow-up stage:} 
    In the follow-up stage, we define the new sampling probabilities as
	{\tightdisplay
	\begin{align*}
		\P[A_{uN}^{(2)} = 0 \mid \mathcal{H}_1] = \mathcal{S}(S_N^{(1)}(0) - S_N^{(1)}(1))\quad\text{and}\quad \P[A_{uN}^{(2)} = 1 \mid \mathcal{H}_1]=1-\P[A_{uN}^{(2)} = 0 \mid \mathcal{H}_1].
	\end{align*}}
	With these updated probabilities, we collect the data \( (A_{uN}^{(2)}, Y_{uN}^{(2)}) \) for \( u \in [N_2] \). These observations are independently and identically distributed, \textit{conditional} on the information from the pilot stage (\( \mathcal{H}_1 \)).
\end{enumerate}

We make one comment on the selection algorithm $\mathcal{S}$. 

\begin{remark}[Generality of $\mathcal{S}$]
	Our main results readily generalize to settings where the selection algorithm depends on more complex functions of the data beyond the simple difference in means. However, for clarity of presentation and broad applicability, we focus on selection algorithms $\mathcal{S}$ based on the estimator of the difference-in-means, \( S_N^{(1)}(1) - S_N^{(1)}(0) \). In fact, such algorithm reflects several important practical objectives in adaptive experimentation. A natural sampling strategy is to assign higher probability to the treatment yielding a higher average outcome in the pilot stage. This strategy aligns well with the goal of revenue maximization, a widely studied objective in e-commerce, online recommendation systems, and inventory control~\citep{bakshy2018ae,che2024optimization}. Similar selection algorithms are also applicable to the identification of optimal policies in political science~\citep{offer2021adaptive}. Moreover, such adaptive assignment rules can support objectives related to welfare and ethics~\citep{burnett2020adding} by minimizing the allocation of inferior treatments. From another perspective, the treatment indicators \(0\) and \(1\) may represent subgroups within a population. Selecting the subgroup that appears more beneficial based on observed outcomes is a common practice in clinical trials. These so-called \emph{subgroup enrichment designs} have been extensively studied in the literature~\citep{magnusson2013group,tanniou2016subgroup,lin2021inference}.
\end{remark}

The choice of the interim statistic $S_N^{(1)}(0)-S_N^{(1)}(1)$ for estimating the difference in treatment means using pilot data is flexible. Throughout the paper, we consider the (scaled) IPW estimator for \( S_N^{(1)}(s) \), defined as:
{\tightdisplay
\begin{align}\label{eq:definition_selection_stat}
	S_N^{(1)}(s) \equiv \frac{1}{N_1^{1/2}} \sum_{u=1}^{N_1} \frac{\indicator(A_{uN}^{(1)} = s) Y_{uN}^{(1)}}{e(s)}\quad\text{for } s \in \{0, 1\}.
\end{align}}
This estimator is a popular choice in the literature on batched bandit algorithms \citep{pmlr-v32-agarwalb14,dimakopoulou2017estimation} and 
multi-stage clinical trials~\citep{shen2014inverse,bowden2017unbiased}. Our general theoretical framework can be extended to accommodate more sophisticated estimators, such as the augmented IPW estimator \citep{dimakopoulou2021online}. For simplicity and consistency between the estimators used for selection and inference, we adopt the IPW estimator.

\subsection{Weighted IPW test statistic}\label{sec:test_statistic}

We are interested in estimating the mean of the potential outcomes $\E[Y_{uN}(s)]$ for $s\in\{0,1\}$, as well as the difference in means $\E[Y_{uN}(0)] - \E[Y_{uN}(1)]$. The WIPW estimator is a natural choice for estimating these quantities. As the name suggests, the WIPW estimator can be viewed as a weighted average of the IPW estimators from the two stages:
\begin{align}\label{eq:WIPW_estimator}
    \WIPW(s)\equiv \sum_{t=1}^2 \frac{N_t h_{N}^{(t)}(s)}{\sum_{t=1}^2 N_t h_{N}^{(t)}(s)} \hat{\Lambda}_{N}^{(t)}(s)\quad\text{where}\quad h_{N}^{(t)}(s)\text{ is a weight function},
\end{align}
and $\hat{\Lambda}_{N}^{(t)}(s)$ is a IPW statistic using data from stage $t$, and defined as
{\tightdisplay
\begin{align*}
	\hat{\Lambda}_{N}^{(t)}(s) \equiv \frac{1}{N_t}\sum_{u=1}^{N_t} \hat{\Lambda}_{uN}^{(t)}(s) \quad\text{where}\quad \hat{\Lambda}_{uN}^{(t)}(s) \equiv \frac{\indicator(A_{uN}^{(t)} = s) Y_{uN}^{(t)}}{e_N(s, \mathcal{H}_{t-1})}
\end{align*}}
and $e_N(s, \mathcal{H}_{t-1}) \equiv \P[A_{uN}^{(t)} = s |\mathcal{H}_{t-1}]$. The choice of weights \(h_{N}^{(t)}(s)\) plays a critical role in the performance of the WIPW estimator. In what follows, we describe how the weights can be selected in practice.

\paragraph{A broad class of weighting choices.}
We consider the class of weights in the form \(h_{N}^{(t)}(s) \equiv e_N^{m}(s, \mathcal{H}_{t-1}) / N^{1/2}\), where different choices of $m$ allow for a variety of weighting strategies:
\begin{itemize}
    \item \textbf{\(m = 0\):} \textit{Constant weighting}, \(h_{N}^{(t)}(s) \equiv 1 / N^{1/2}\);
    \item \textbf{\(m = 1/2\):} \textit{Adaptive weighting}, \(h_{N}^{(t)}(s) \equiv e_N^{1/2}(s,\mathcal{H}_{t-1}) / N^{1/2}\).
\end{itemize}
The constant weighting corresponds to the usual IPW estimator~\citep{bowden2017unbiased}. The adaptive weighting method, is sometimes referred to as the variance-stabilizing weighting \citep{luedtke2016statistical,Hadad2021,bibaut2024demistifying}. It is particularly useful for compensating the high variability in \(\hat{\Lambda}_{N}^{(2)}(s)\), arising from the potential downsampling in the follow-up stage caused by selection algorithm $\mathcal{S}$. Other data-dependent weighting choices under the same class of statistics include \textbf{\(m = 1\)}. The resulting statistic can be shown to be (asymptotically) equivalent to the sample mean statistic after proper augmentation (Lemma~\ref{lem:sample_mean_test_statistic}). Our results in the main text can be applied to $m=0$ and $m=1/2$ and we extend the results to $m=1$ in Appendix \ref{sec:extension_m_1}. 

\paragraph{Two scaling schemes.}

Motivated by the normalized statistic proposed in~\citet{Hadad2021}, we study the following two test statistics based on the WIPW estimator, depending on whether normalization is applied. Define the normalization $\hat S_N\equiv (N\hat V_N(0)+N\hat V_N(1))^{1/2}$, where 
the variance estimator $\hat{V}_N(s)$ is defined as
{\tightdisplay
\begin{align}\label{eq:variance-estimator}
	\hat{V}_N(s) \equiv \sum_{t=1}^2 \Big(\frac{N_t h_{N}^{(t)}(s)}{\sum_{t=1}^2 N_t h_{N}^{(t)}(s)}\Big)^2 \frac{1}{N_t^2}\sum_{u=1}^{N_t} \big(\hat{\Lambda}_{uN}^{(t)}(s) - \WIPW(s)\big)^2.
\end{align}}
Then we can define the corresponding unnormalized and normalized test statistics as $T_N \equiv \WIPW(0) - \WIPW(1)$ and $ 
W_N \equiv T_N / \hat S_N$. It is commonly understood that tests based on normalized and unnormalized statistics are asymptotically equivalent as the normalization factor converges to a constant.
This holds true in many settings with i.i.d. data, as exemplified by the Wald and score tests. 
With adaptively collected data, however, normalization can impact the performance of the testing procedure. We refer the readers to simulation results in Section \ref{sec:simulation}.

\paragraph{A peek at the sampling distribution.}

Understanding the asymptotic distributions of test statistics $\WIPW(s),T_N$ and $W_N$ is important for downstream inferential tasks. To build intuition about the sampling distributions, we begin with a simulation study. Consider the scaled difference in means: $c_N \equiv \sqrt{N}(\E[Y_{uN}(0)] - \E[Y_{uN}(1)])$. \textit{Without loss of generality, we assume \(c_N \leq 0\) in this paper.} We plot the sampling distribution of centered statistic \(\sqrt{N}T_N - c_N\), where $c_N$ takes values in \( \{0, -5, -10, -15\}\) and repeat the computation $5000$ independent times. 
The detailed simulation setup is provided in Appendix \ref{sec:illustration_simulation} and the results are shown in Figure \ref{fig:sampling_distribution}.

\begin{figure}[!ht]
    \centering
    \includegraphics[width=.8\textwidth]{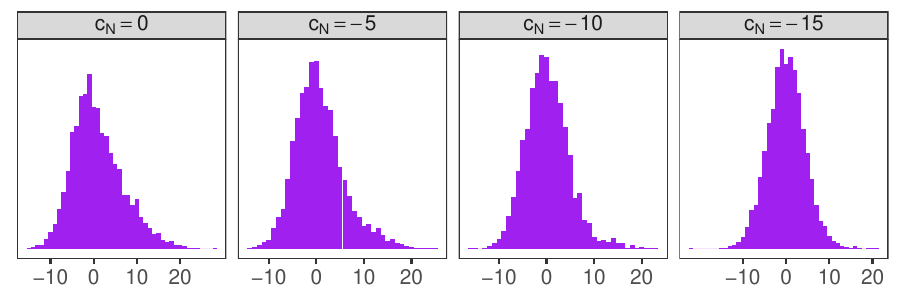}
    \caption{Sampling distribution of \(\sqrt{N} T_N - c_N\) with adaptive weighting (\(m = 1/2\)).}
    \label{fig:sampling_distribution}
\end{figure}

In the figure, we observe the following pattern: when \(c_N = 0\) (left-most panel), 
the sampling distribution of the estimator is highly skewed; 
as the magnitude of \(c_N\) increases, the distribution becomes more symmetric and eventually approaches a near-normal distribution (right-most panel). 
The shape transition of the sampling distribution from \(c_N = 0\) to \(c_N = -15\) suggests the limiting distribution of the test statistic depends on the signal strength \(c_N\). The results presented in the next section provide an exact characterization of such dependence.

\section{Theoretical results}\label{sec:weak_convergence_WIPW} 

The organization of this section is as follows. We state our main weak limit results in Section \ref{sec:weak_convergence}, followed by several remarks in Section \ref{sec:remarks_weak_convergence}. 
Section~\ref{sec:phase_transition} describes the phase transition of the limiting distribution of the test statistic across different signal strengths.
In Section \ref{sec:bootstrap_procedure}, we propose a simulation-based procedure for hypothesis testing based on the weak limits. 

\subsection{General theory: signal-dependent weak limits}\label{sec:weak_convergence}

We first explicitly define the selection algorithm $\mathcal{S}(S_N^{(1)}(0)-S_N^{(1)}(1))$. We allow the $\mathcal{S}$ to vary across sample size $N$ so we will write $\mathcal{S}_N$ to emphasize such dependence. Suppose $e(s,x)\in[0,1]$ is a sampling function for $s\in\{0,1\}$ such that $e(0,x)+e(1,x)=1$ for any $x\in\mathbb{R}$. Then define the sampling function
{\tightdisplay
\begin{align}\label{eq:clip}
	\mathcal{S}_N(S_N^{(1)}(0) - S_N^{(1)}(1))\equiv \min\{1-l_N, \max\{l_N,e(0,S_N^{(1)}(0)-S_N^{(1)}(1))\}\},
\end{align}}
where $l_N$ is a positive sequence $l_N\in[0,1/2)$. 
We note that the minimum and maximum functions are mainly used for ensuring both treatments are assigned with nonzero probability in the follow-up stage when $l_N>0$. This is a reasonable assumption for many practical applications/algorithms, as it is often desirable to maintain a certain level of exploration in the follow-up stage. Such purpose includes reducing the risk of assigning treatment to the inferior group or stabilizing the variance of the downstream test statistic. Such clipping strategy has been widely adopted in the literature of adaptive experiments~\citep{Zhang2020,Hadad2021}. The results in this section assume $l_N\in(0,1/2)$, i.e. strictly positive, but we will extend the results to $l_N=0$ in Section \ref{sec:extension_m_1}. We defer the discussion on the choice of $l_N$ to Remark~\ref{rmk:positivity-l-N} and \ref{rmk:early-dropping}. 

Define the extended $\mathbb{R}$ space as $\bar{\mathbb{R}}\equiv \mathbb{R}\cup \{-\infty\}$. We consider the following assumptions.
\begin{assumption}[Moment conditions]\label{assu:moment_condition}
	For any $s\in\{0,1\}$, we have $0< \inf_{N}\big(\E\left[Y_{uN}^2(s)\right]-\E\left[Y_{uN}(s)\right]^2\big)\leq \sup_{N}\big(\E\left[Y_{uN}^4(s)\right]\big)^{1/2}<\infty$. Moreover, for $p\in\{1,2\}, \lim_{N\rightarrow\infty}\E[Y_{uN}^{p}(s)]$ exist. Recalling $c_N\equiv \sqrt{N}(\E[Y_{uN}(0)]-\E[Y_{uN}(1)])$, we assume $\lim_{N\rightarrow\infty}c_N=c$ for $c\in[-\infty,0]$. 
\end{assumption}

\begin{assumption}[Sampling designs]\label{assu:sampling_design}
	There exist $1>c_u>c_l>0$ such that the first-stage sampling probability $e(s)\in(c_l,c_u)$ for any $s\in\{0,1\}$. The sampling function $e(s,x)$ satisfies $e(0,x)+e(1,x)=1$ for any $x\in\bar{\mathbb{R}}$. Moreover, one of the following assumptions holds:
	\begin{enumerate}
		\item\textbf{Lipschitz condition:} For any $s\in\{0,1\}$, $e(s,x)$ is a Lipschitz function over $x\in\bar{\mathbb{R}}$, with universal Lipschitz constant $L>0$ and $e(s,-\infty)$ takes values in $\{0,1\}$;
		\item\textbf{Step-function condition:} There exist some $m_1\in\mathbb{R},K\in\mathbb{N}$ and continuous function $g:\bar{\mathbb{R}}\rightarrow\bar{\mathbb{R}}$ with $g(-\infty)=-\infty$, such that for any $s\in\{0,1\}$, $e(s,x)=\sum_{k=1}^K c_{k}\indicator(g(x)\in C_k)$ where $ c_k\in [0,1]$ and $C_k$ are disjoint sets. In particular, $C_1=[-\infty, m_1]$ and $C_k$ are open sets for $k\geq 2$, such that $\cup_{k=2}^{K}C_k=(m_1,\infty)$.
 	\end{enumerate}
\end{assumption}

\begin{assumption}[Constant weighting]\label{assu:constant_weighting}
	Suppose $m=0$ is used and clipping rate $l_N$ (introduced in~\eqref{eq:clip}) satisfies $0<c_l<\bar l=l_N<c_u<1/2$ for any $N\in\mathbb{N}$.
\end{assumption}

\begin{assumption}[Adaptive weighting]\label{assu:adaptive_weighting}
	Suppose $m=1/2$ is used and clipping rate $l_N$ satisfies $l_N\in (0,1/2)$ for any $N\in\mathbb{N}$. Moreover, $\lim_{N\rightarrow\infty}l_N=0$ and $Nl_N\rightarrow\infty$.
\end{assumption}

We postpone the discussion on the assumptions to Section~\ref{sec:remarks_weak_convergence}. Before stating the theorem, we first describe the limiting distributions of the WIPW estimator and the test statistics $T_N,W_N$. 

\paragraph{Form of limiting distributions.}

To ease the presentation, we use $(a)_{2\times 2}$ to denote a symmetric matrix with dimension $2$, diagonal values to be $1$ and off-diagonal values to be $a$. The limiting distributions of both $\WIPW(s)$ and the test statistics $T_N$ and $W_N$, 
can be expressed as weighted sums of two dependent Gaussian vectors, $A^{(t)}\equiv (A^{(t)}(0), A^{(t)}(1))^\top$ for $t\in[2]$. Intuitively, $A^{(1)}$ corresponds to the randomness in the pilot stage 
and $A^{(2)}$ to that in the follow-up stage and is dependent on $A^{(1)}$. 
Concretely, the distributions of $A^{(t)}$ can be defined as
{\tightdisplay
\begin{align*}
	A^{(1)} \sim N(\bm 0, \bm \Sigma^{(1)})\quad\text{and}\quad A^{(2)}|A^{(1)} \sim N(\bm 0, \bm \Sigma^{(2)}(A^{(1)})),
\end{align*}}
where $\bm\Sigma^{(1)}\equiv (\mathrm{Cov}^{(1)})_{2\times 2}$ and $\bm \Sigma^{(2)}(A^{(1)})\equiv (\mathrm{Cov}^{(2)}(A^{(1)}))_{2\times 2}$. The covariance $\mathrm{Cov}^{(2)}(A^{(1)})$ depends on the realization of $A^{(1)}$. \edit{To make the dependence explicit, define the limiting assignment probabilities $H^{(1)}(s)\equiv e(s)$ and $H^{(2)}(s)\equiv \max\{\lim_{N\to\infty}l_N,\, e(s, \mathcal{S}^\infty(A^{(1)}, c))\}$, where $\mathcal{S}^\infty$ maps the first-stage Gaussian realization and signal strength to a second-stage sampling probability (see Appendix~\ref{sec:explicit_form_asymptotic_distribution} for the precise definition). The scaled variances are $V^{(t)}(s) \equiv \lim_{N\to\infty}\E[Y_{uN}^2(s)] - H^{(t)}(s)(\lim_{N\to\infty}\E[Y_{uN}(s)])^2$. Then the off-diagonal covariances take the form
{\tightdisplay
\begin{align}\label{eq:cov_main_text}
	\mathrm{Cov}^{(t)} = -\frac{(H^{(t)}(0)\,H^{(t)}(1))^{1/2}}{(V^{(t)}(0)\,V^{(t)}(1))^{1/2}}\lim_{N\to\infty}\E[Y_{uN}(0)]\,\E[Y_{uN}(1)],\quad t=1,2.
\end{align}}
Note $\mathrm{Cov}^{(1)}$ is deterministic, while $\mathrm{Cov}^{(2)}(A^{(1)})$ is random through $H^{(2)}(s)$.}
\begin{itemize}
	\item \textbf{Weak limits of $\WIPW(s)$.} Suppose $\mathcal{W}$ stands for weighting scheme and takes values in $\{\mathcal{A},\mathcal{C}\}$, standing for adaptive and constant weighting, respectively. We use $\bar{\mathbb{W}}_{\mathcal{W}}(s)$ to denote the limiting distributions of $\WIPW(s)$ after proper centering and scaling. We will use $\overset{d}{=}$ to denote equality in distribution. Then $\bar{\mathbb{W}}_{\mathcal{W}}(s)$ can be written as
	{\tightdisplay
	\begin{align}\label{eq:limiting_representation_single_outcome}
		\bar{\mathbb{W}}_{\mathcal{W}}(s)\overset{d}{=}\sum_{t=1}^2 A^{(t)}(s) \bar w_{\mathcal{W}}^{(t)}(s)\quad\text{for any }s\in\{0,1\},
	\end{align}}
	\edit{where the weights under constant weighting are $\bar w_{\mathcal{C}}^{(t)}(s) = (q_t\, V^{(t)}(s) / H^{(t)}(s))^{1/2}$, and under adaptive weighting,
	$\bar w_{\mathcal{A}}^{(t)}(s) = \big(q_t\, (H^{(t)}(s))^{1/2} / (\sum_{t'} q_{t'} (H^{(t')}(s))^{1/2})^2 \cdot V^{(t)}(s) / H^{(t)}(s)\big)^{1/2}$,
	with $q_t \equiv N_t/N$. Note $\bar w_{\mathcal{W}}^{(2)}(s)$ may depend on $A^{(1)}$ through $H^{(2)}(s)$.}
	
	\item \textbf{Weak limits of $T_N$ and $W_N$.} We use $\mathbb{W}_{\mathcal{V}}^{\mathcal{W}}$ to denote the limiting distributions for $T_N$ ($\mathcal{V}=\mathcal{U}$, unnormalized) and $W_N$ ($\mathcal{V}=\mathcal{N}$, normalized) after proper centering and scaling. We emphasize the dependence of $\mathbb{W}_{\mathcal{V}}^{\mathcal{W}}$ on the limiting signal strength $c$ by writing $\mathbb{W}_{\mathcal{V}}^{\mathcal{W}}=\mathbb{W}_{\mathcal{V}}^{\mathcal{W}}(c)$. The limiting distributions can be expressed directly in terms of $\bar w_{\mathcal{W}}^{(t)}(s)$:
	{\tightdisplay
	\begin{align}\label{eq:limiting_representation}
		\mathbb{W}_{\mathcal{U}}^{\mathcal{W}}(c)\overset{d}{=}\sum_{t=1}^2 A^{(t)}(0)\, \bar w_{\mathcal{W}}^{(t)}(0) -\sum_{t=1}^2 A^{(t)}(1)\, \bar w_{\mathcal{W}}^{(t)}(1),\quad
		\mathbb{W}_{\mathcal{N}}^{\mathcal{W}}(c)\overset{d}{=}\frac{\mathbb{W}_{\mathcal{U}}^{\mathcal{W}}(c)}{\big(\sum_{s=0}^1\sum_{t=1}^2 (\bar w_{\mathcal{W}}^{(t)}(s))^2\big)^{1/2}}.
	\end{align}}
	Note both $\bar w_{\mathcal{W}}^{(t)}(s)$ and $A^{(t)}(s)$ also depend on the signal strength $c$ but we omit this dependence for simplicity.
\end{itemize}

We summarize the definition of different limiting distributions in Table~\ref{tab:limit-dist}.
\small
\begin{table}[!ht]
	\centering
	\caption{Summary of different limiting distributions. $\mathcal{W}$ takes values in $\{\mathcal{A},\mathcal{C}\}$.}
	\label{tab:limit-dist}
	\begin{tabular}{c|c|c}
	  \textbf{Test statistic} &  \textbf{Estimand} & \textbf{Weak limit}  \\
	  \hline
	  $\WIPW(s)$ & $\E[Y_{uN}(s)]$  & $\bar{\mathbb{W}}_{\mathcal{W}}$  \\
	  $T_N$ & $\E[Y_{uN}(0)]-\E[Y_{uN}(1)]$  & $\mathbb{W}_{\mathcal{U}}^{\mathcal{W}}$    \\
	  $W_N$ & $\E[Y_{uN}(0)]-\E[Y_{uN}(1)]$ & $\mathbb{W}_{\mathcal{N}}^{\mathcal{W}}$  \\
	\end{tabular}
\end{table}
\normalsize

All limiting weights are expressed through $\bar w_{\mathcal{W}}^{(t)}(s)$, with the normalized case $\mathbb{W}_{\mathcal{N}}^{\mathcal{W}}$ obtained by rescaling $\mathbb{W}_{\mathcal{U}}^{\mathcal{W}}$. Now we are ready to state the main results.

\begin{theorem}[Weak convergence]\label{thm:weak_convergence_W_N}
	Suppose Assumptions \ref{assu:moment_condition}-\ref{assu:sampling_design} hold. Recall the definition $c_N\equiv \lim_{N\rightarrow\infty}\sqrt{N}(\E[Y_{uN}(0)]-\E[Y_{uN}(1)])$. The following statements hold.
	\begin{enumerate}
		\item Suppose Assumption~\ref{assu:constant_weighting} holds. Then, $\sqrt{N}(\WIPW(s)-\E[Y_{uN}(s)])$ converges weakly to $\bar{\mathbb{W}}_{\mathcal{C}}(s)$ for any $s\in\{0,1\}$. Moreover, $\sqrt{N}T_N-c_N$ and $W_N - c_N/\hat S_N$ converge weakly to $\mathbb{W}_{\mathcal{U}}^{\mathcal{C}}(c)$ and $\mathbb{W}_{\mathcal{N}}^{\mathcal{C}}(c)$, respectively.
		\item Suppose Assumption~\ref{assu:adaptive_weighting} holds. Then, $\sqrt{N}(\WIPW(s)-\E[Y_{uN}(s)])$ converges weakly to $\bar{\mathbb{W}}_{\mathcal{A}}(s)$ for any $s\in\{0,1\}$. Moreover, $\sqrt{N}T_N-c_N$ and $W_N - c_N/\hat S_N$ converge weakly to $\mathbb{W}_{\mathcal{U}}^{\mathcal{A}}(c)$ and $\mathbb{W}_{\mathcal{N}}^{\mathcal{A}}(c)$, respectively.
	\end{enumerate}
\end{theorem}
\edit{
\noindent Theorem~\ref{thm:weak_convergence_W_N} states a qualitative CLT result and its proof can be found in Appendix \ref{sec:proof_weak_convergence}. Now, we present a quantitative CLT result.
\begin{theorem}[A quantitative CLT under Lipschitz sampling function]\label{thm:quantitative_CLT_W_N}
	Suppose Assumptions \ref{assu:moment_condition}-\ref{assu:sampling_design}(1) hold. Defining $\mathcal{L}_{Y_{uN}^2}(s)\equiv \big|\E[Y_{uN}^2(s)]-\lim_{N\rightarrow\infty}\E[Y_{uN}^2(s)]\big|$ and $\mathcal{L}_{Y_{uN}}(s)\equiv \big|\E[Y_{uN}(s)]-\lim_{N\rightarrow\infty}\E[Y_{uN}(s)]\big|$ for any $s\in\{0,1\}$, then suppose 
	\begin{align}\label{eq:Y-limit-rate}
		\mathcal{L}_{Y_{uN}^2}(s)=O(N^{-1})\quad\text{and}\quad\mathcal{L}_{Y_{uN}}(s)=O(N^{-1}).
	\end{align}
	Furthermore, define $d_{\mathrm{BL}}(X,Y)\equiv \sup_{\|f\|_{\mathrm{BL}}\leq 1}|\E[f(X)]-\E[f(Y)]|$, where the definition of $\|f\|_{\mathrm{BL}}$ can be found in Appendix~\ref{def:BL-function}. Then, the following statements hold.
	\begin{enumerate}
		\item Suppose Assumption~\ref{assu:constant_weighting} holds. Then,
		\begin{align*} 
			d_{\mathrm{BL}}(\sqrt{N}T_N-c_N,\mathbb{W}_{\mathcal{U}}^{\mathcal{C}}(c))=O(N^{-1/2})\quad\text{and}\quad d_{\mathrm{BL}}(W_N - c_N/\hat S_N,\mathbb{W}_{\mathcal{N}}^{\mathcal{C}}(c))=O(N^{-1/2})
		\end{align*}
		\item Suppose Assumption~\ref{assu:adaptive_weighting} holds. Then, defining $L_N\equiv l_N^{1/2}+N^{-1/2}l_N^{-1/2}$, we have 
		\begin{align*} 
			d_{\mathrm{BL}}(\sqrt{N}T_N-c_N,\mathbb{W}_{\mathcal{U}}^{\mathcal{A}}(c))=O(L_N)\quad\text{and}\quad d_{\mathrm{BL}}(W_N - c_N/\hat S_N,\mathbb{W}_{\mathcal{N}}^{\mathcal{A}}(c))=O(L_N).
		\end{align*}
	\end{enumerate}
\end{theorem}
}

\subsection{Remarks on the assumptions, results and technical challenges}\label{sec:remarks_weak_convergence}

We present several remarks on the assumptions and technical challenges of proving Theorems~\ref{thm:weak_convergence_W_N} and~\ref{thm:quantitative_CLT_W_N}. 

\begin{remark}[Comments on Assumptions~\ref{assu:moment_condition}-\ref{assu:adaptive_weighting}]\label{rmk:positivity-l-N}
	
	Assumption~\ref{assu:moment_condition} is a mild regularity assumption. Assumption~\ref{assu:sampling_design} imposes mild restrictions on the sampling function,
	providing substantial flexibility for the choice of sampling function $e(s,x)$ 
	(and hence the selection algorithm $\mathcal{S}$) to encode different experimental designs. \edit{To demonstrate the generality of these assumptions, we present two classes of experiments in Appendix~\ref{sec:application} (see also Remark~\ref{rmk:case-study}).} Now we comment on Assumption~\ref{assu:constant_weighting} and Assumption~\ref{assu:adaptive_weighting}. 
	Constant weighting, informed by Assumption \ref{assu:constant_weighting}, 
	requires the minimum sampling probability to be uniformly bounded away from $0$. In the causal inference literature, Assumption \ref{assu:constant_weighting} is known as the positivity assumption~\citep{crump2009dealing,imbens2015causal}. The adaptive weighting enables less stringent requirement on the sampling probability in the second stage encouraging further exploitation. Assumption \ref{assu:adaptive_weighting} allows minimum sampling probability in the second stage to go to zero at the rate slower than $1/N$. Similar assumption has also been adopted in \citet{Zhang2020,Hadad2021}. 
\end{remark}

\begin{remark}[Early-dropping experiments]\label{rmk:early-dropping}
	Relevant to Remark~\ref{rmk:positivity-l-N}, one kind of adaptive sampling Theorem~\ref{thm:weak_convergence_W_N} does not cover is the so-called early-dropping experiments~\citep{sampson2005drop,sill2009drop}. In these experiments, the inferior treatment will be dropped from the follow-up stage. In this case, $\mathcal{S}_N(S_N^{(1)}(0) - S_N^{(1)}(1)) $ can be $0$ or $1$. We will show in Appendix~\ref{sec:extension_m_1} that when $m=1$, we can get rid of the clipping~\eqref{eq:clip} by allowing $l_N=0$.
\end{remark}

\edit{
\begin{remark}[Application to different adaptive experiments]\label{rmk:case-study}
	To demonstrate the generality of our framework, we present two classes of adaptive experimental designs in Appendix~\ref{sec:application}: \emph{batched bandit experiments} (modified Thompson sampling and $\varepsilon$-greedy algorithm) and \emph{subgroup enrichment designs} (enrichment based on effect size or interim $p$-values). In each case, we verify that the sampling function satisfies Assumption~\ref{assu:sampling_design}, so that Theorem~\ref{thm:weak_convergence_W_N} applies. Notably, the subgroup enrichment design based on interim $p$-values involves a nuisance parameter, which falls outside the scope of Theorem~\ref{thm:weak_convergence_W_N}; we extend our results to accommodate this setting in Theorem~\ref{thm:weak_convergence_W_N_nuisance} in Appendix~\ref{sec:extension_nuisance}.
\end{remark}
}

\edit{
\begin{remark}[Technical relationship to~\citet{chen2023optimal,che2023adaptive}]\label{rmk:CA-approach}
	Our proof shares a common high-level template with \citet{chen2023optimal,che2023adaptive}. Theorem C.1 in~\citet{chen2023optimal} establishes a Gaussian limit representation for the \emph{path} of assignment probabilities and batch-wise standardized means in multi-stage adaptive experiments: the next-batch increment is asymptotically Gaussian conditional on the past, with a covariance that depends on the limiting assignment probabilities. Our proof of Theorem~\ref{thm:weak_convergence_W_N} follows the same mechanism in the specialized two-stage/two-arm setting: (i) a CLT for the stage-1 root, (ii) a conditional CLT for the stage-2 root given the stage-1 history, and (iii) bounded-Lipschitz and continuous-mapping arguments to obtain joint weak convergence. In fact, the similar technique was also employed in~\citet{che2023adaptive} when proving their Theorem 1.

	The WIPW test statistics considered in this paper incorporate the difference-in-means estimator as a special case (see Appendix~\ref{sec:extension_m_1}), but go beyond it by encompassing a broader family of IPW-type statistics. Our main results concern these \emph{pooled IPW-type statistics}, for which the Gaussian limit representation is not sufficient on its own. In particular, a joint approximation for the \emph{entire} vector of primitives entering the IPW weights or denominators is established (Step~2 in Appendix~\ref{sec:proof_step_2}), requiring some more delicate continuous mapping results (Lemmas~\ref{lem:continuous_map_varying}). Moreover, we strengthen the weak convergence result to quantitative bounded-Lipschitz bounds in Theorem~\ref{thm:quantitative_CLT_W_N} when sampling function is smooth enough. Regarding clipping, vanishing probabilities directly affect moments of IPW-type terms and studentization. We therefore allow $l_N\to 0$ but make its impact explicit via a bounded-Lipschitz error bound of the form $d_{\mathrm{BL}}(\mathcal{L}(E_N),\mathcal{L}(W))\;\lesssim\; l_N^{1/2} + (N l_N)^{-1/2}$, where $E_N$ is the concatenated random variables which will contribute to the final test statistic (see Lemma~\ref{lem:quantitative-CLT}). We defer further discussion to Remark~\ref{rmk:technical_challenges}.
\end{remark}
}

\edit{
\begin{remark}[On the quantitative bounded-Lipschitz CLT and the role of clipping]\label{rmk:technical_challenges}
	Theorem~\ref{thm:quantitative_CLT_W_N} controls the approximation error in bounded-Lipschitz distance, $d_{\mathrm{BL}}(\mathcal{L}(E_N),\mathcal{L}(W))$, which requires establishing convergence rates from the test-function approach \citep[Ch.~11]{Dudley_2002} under the dependence induced by adaptive sampling. Under adaptive weighting with vanishing clipping (Assumption~\ref{assu:adaptive_weighting}), the resulting bound for IPW-type test statistis considered in this paper takes the form $d_{\mathrm{BL}}(\mathcal{L}(E_N),\mathcal{L}(W))\lesssim\; l_N^{1/2} + (N l_N)^{-1/2}$. The term $(N l_N)^{-1/2}$ reflects the conditional normal-approximation error with an \emph{effective sample size} of order $N l_N$, since the minimum assignment probability is of order $l_N$ and the IPW weights amplify higher moments accordingly. The term $l_N^{1/2}$ captures the additional distortion introduced by propagating the CLT through nonlinear maps involving truncation/clipping and square-root/inverse transformations (e.g., the stage-2 sampling map and studentization). This bound makes explicit the trade-off between aggressive exploitation (small $l_N$) and inferential stability. In particular, $N l_N\to\infty$ is necessary for the bound to vanish, and balancing the two terms suggests the heuristic choice $l_N \asymp N^{-1/2}$, yielding $d_{\mathrm{BL}}(\mathcal{L}(E_N),\mathcal{L}(W))=O(N^{-1/4})$. Technically, obtaining such rates requires a conditional quantitative CLT for bounded-Lipschitz test functions (Lemma~\ref{lem:CLT_BL}) to handle the nonlinear dependence of the second-stage contributions on first-stage random quantities. Such quantitative results are particularly appealing when studying the Type-I error rate of simulation-based test procedure proposed in Section~\ref{sec:bootstrap_procedure} (see Corollary~\ref{cor:type-I-error-rate}).
\end{remark}
}

\subsection{Phase transition and implication on hypothesis testing}
\label{sec:phase_transition}
A key strength of our results in Theorem~\ref{thm:weak_convergence_W_N} is the weak limits can be expressed explicitly using weighted sums of two dependent Gaussian variables $A^{(t)}$ as shown in \eqref{eq:limiting_representation_single_outcome} and \eqref{eq:limiting_representation}. This allows us to understand the limiting distributions of different test statistics in a more intuitive way. We will discuss how the limiting distributions of $T_N$ and $W_N$ change as signal strength $c$ changes. Specifically, the shapes of limiting distributions are determined by the covariance $\mathrm{Cov}^{(2)}(A^{(1)})$ and random weights $\bar w_{\mathcal{W}}^{(2)}(s)$, which are both influenced by the signal strength \( c \). Consider the following two regimes:
\begin{itemize}
	\item \textbf{Strong signal regime.} When $c=-\infty$, i.e., the absolute difference between two expected outcomes is much larger than $1/\sqrt{N}$, it can be shown that $\mathrm{Cov}^{(2)}(A^{(1)})$ and $\bar w_{\mathcal{W}}^{(2)}(s)$ are deterministic constants. This implies that the fluctuation of the first stage estimator does not affect the final limit through the limiting covariance. Therefore the final limit $\mathbb{W}_{\mathcal{V}}^{\mathcal{W}}(-\infty)$ follows a Gaussian distribution.   
	\item \textbf{Zero and weak signal regimes.} When $c\in (-\infty,0]$, the limiting distribution may no longer be a normal distribution. This is because $A^{(1)}$ will appear in the conditional distribution of $A^{(2)}|A^{(1)}$ through the limiting covariance $\mathrm{Cov}^{(2)}(A^{(1)})$. Similarly, $\bar w_{\mathcal{W}}^{(2)}(s)$ depends on the realization of $A^{(1)}$. Therefore, when the signal is ``weak'', the non-normal behavior is the ``price'' one needs to pay for choosing to use the adaptive sampling scheme.
\end{itemize}

Additional insights on the non-normal limiting behaviors from double-dipping and data generating process perspectives can be found in Appendix \ref{sec:intuition_limits}. To get better intuition, 
we simulate $\mathbb{W}_{\mathcal{V}}^{\mathcal{W}}(c)$ 
with $\mathcal{V}=\mathcal{U},\mathcal{W}=\mathcal{A}$---this is the limiting distribution 
corresponding to the sampling distribution presented in Figure~\ref{fig:sampling_distribution}. 
We vary the limiting signal strength $c$ and show the simulated results in Figure \ref{fig:transition}, 
which align closely with those in Figure \ref{fig:sampling_distribution}. 
As $c$ approaches $-\infty$, i.e., as the signal gets stronger, the limiting distribution of $\mathbb{W}_{\mathcal{V}}^{\mathcal{W}}(c)$ 
approaches a normal distribution. Such phase transition, as informed by Theorem \ref{thm:smooth_transition}, is smooth with respect to signal strength $c$ under \textit{1-Wasserstein distance}, $d_{W_1}$.

\begin{figure}[!ht]
    \centering
    \includegraphics[width=.8\textwidth]{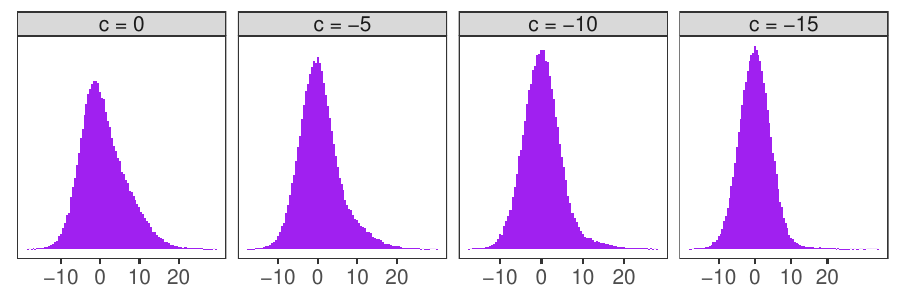}
    \caption{Distribution \(\mathbb{W}_{\mathcal{U}}^{\mathcal{A}}(c)\) as a function of limiting signal strength $c$.}
    \label{fig:transition}
\end{figure}

\begin{theorem}[Smooth transition of limiting distributions]\label{thm:smooth_transition}
	Suppose the assumptions of Theorem \ref{thm:weak_convergence_W_N} hold. Then for $\mathcal{V}\in \{\mathcal{U},\mathcal{N}\}$ and $\mathcal{W}\in \{\mathcal{A},\mathcal{C}\}$, we have $d_{W_1}(\mathbb{W}_{\mathcal{V}}^{\mathcal{W}}(-\infty), \mathbb{W}_{\mathcal{V}}^{\mathcal{W}}(c))$ converges to 0 as $c$ approaches $-\infty$.
\end{theorem}

\noindent Proof of Theorem \ref{thm:smooth_transition} can be found in Appendix \ref{sec:proof_smooth_transition}, where the definition of Wasserstein distance can also be found. Gathering these insights, we now discuss the implications of our results on hypothesis testing.

\paragraph{Implication on hypothesis testing.}

The strength of our results lies in establishing weak convergence under minimal moment conditions, mild assumptions on the sampling functions, and broad signal strength regimes. These assumption-lean properties are not merely of theoretical interest—they carry practical significance for downstream hypothesis testing. To demonstrate the implication of our results on the hypothesis testing, consider the null hypothesis $H_{0N}:\E[Y_{uN}(0)] - \E[Y_{uN}(1)]=0$ and two alternatives within the general hypothesis $H_{1N}: \E[Y_{uN}(0)] - \E[Y_{uN}(1)] \neq 0$:
{\tightdisplay
\[
H_{2N}: \E[Y_{uN}(0)] - \E[Y_{uN}(1)] = \frac{b_2}{\sqrt{N}} \quad \text{and}\quad
H_{3N}: \E[Y_{uN}(0)] - \E[Y_{uN}(1)] = \frac{b_3}{N^\beta},
\]}
where $b_2, b_3 \in (-\infty, 0)$ and $\beta \in [0, 1/2)$. The contiguous alternative $H_{2N}$ corresponds to a weak signal regime, while $H_{3N}$ reflects a strong signal setting. Ideally, a test should control the Type-I error under $H_{0N}$ and achieve non-trivial power under $H_{2N}$, while attaining power approaching one under $H_{3N}$. It is also desirable for the test to remain assumption-lean with respect to the potential outcome distributions.
Our results accommodate this full range of signal strengths while maintaining minimal assumptions on the sampling functions and underlying distributions.

\paragraph{Comparison to existing literature.}
To highlight the significance of our results, we compare our results to those in related work. \citet{Zhang2020} establish asymptotic normality via stage-wise normalization under general hypotheses. Their method, however, relies on restrictive outcome distribution assumptions. \edit{While normality-based approaches may lose information and yield inferior power in some settings, our numerical study (Section~\ref{sec:finite-sample}) reveals that this is not universal: the relative power advantage between normal-limit and non-normal-limit tests depends on the structure of the data generating process. This nuanced picture further highlights the importance of understanding non-normal methods; in particular, our weak convergence results can be readily used to derive the power function and guide the choice of testing procedure.} \citet{Hadad2021} analyze the WIPW estimator for $m = 1/2$, but only in strong signal regimes (e.g., $H_{3N}$), excluding $H_{0N}$ and contiguous local alternatives. Both \citet{Zhang2020} and \citet{Hadad2021} achieve asymptotic normality under strong signals, potentially at the cost of power. \citet{Hirano2023} and \citet{adusumilli2023optimal} derive asymptotic representations for general test statistics under batched designs and contiguous alternatives $H_{2N}$, relying on classical limit-of-experiment theory by~\citet{le1972limits}. However, to apply Le Cam's theory, one must first establish the required weak convergence of certain statistics (see, for example, Theorem 2 in~\citep{Hirano2023}). Their results also assume quadratic mean differentiability (QMD) (or smooth semiparametric models), which may be restrictive in practice. In particular, their results apply only when the score function of the parametric distribution class is known. A complementary line of work uses diffusion approximations with increasing batch numbers \citep{fan2021diffusion,kuang2024weak}, which differs from our two-stage setup with growing per-stage samples. Further discussion appears in Appendix~\ref{sec:inspection_literature}.

\subsection{Asymptotically valid simulation-based tests}\label{sec:bootstrap_procedure}

Theorem~\ref{thm:weak_convergence_W_N} characterizes the limiting behavior of WIPW test statistics, forming the basis for constructing asymptotically valid tests. 
In this section, we focus on testing whether or not the difference in means $\E[Y_{uN}(0)]-\E[Y_{uN}(1)]=0$ for demonstration. Similar results can be established for other hypotheses, for example, single outcome hypothesis $\E[Y_{uN}(s)]=0$ for some $s\in\{0,1\}$. Based on Theorem~\ref{thm:weak_convergence_W_N}, we can define the following asymptotically valid tests: for any $\mathcal{W} \in \{\mathcal{A}, \mathcal{C}\}$,
{\tightdisplay
\begin{align}\label{eq:four_tests}
    \phi_{\mathcal{U}}^{\mathcal{W}} \equiv \indicator \big(\sqrt{N}T_N \geq \mathbb{Q}_{1-\alpha}(\mathbb{W}_{\mathcal{U}}^{\mathcal{W}}(0))\big)\quad\text{and}\quad\phi_{\mathcal{N}}^{\mathcal{W}} \equiv \indicator \big(W_N \geq \mathbb{Q}_{1-\alpha}(\mathbb{W}_{\mathcal{N}}^{\mathcal{W}}(0))\big).
\end{align}}
The critical values \(\mathbb{Q}_{1-\alpha}(\mathbb{W}_{\mathcal{V}}^{\mathcal{W}}(0))\) are the \((1-\alpha)\)-th quantiles of the limiting distribution \(\mathbb{W}_{\mathcal{V}}^{\mathcal{W}}(0)\) defined in Theorem~\ref{thm:weak_convergence_W_N}. However, tests in~\eqref{eq:four_tests} cannot be implemented in practice for two reasons. First, the asymptotic distributions \(\mathbb{W}_{\mathcal{V}}^{\mathcal{W}}(0)\) are generally non-normal (see also the left-most panel in Figure \ref{fig:transition}). Second, the limiting distribution involves unknown nuisance parameters, which have to be estimated using observed data. In this section, we propose a simulation-based procedure to address these challenges and construct valid tests, which can be implemented in practice.

\paragraph{A fast simulation-based procedure.} 

Note that the expression of limiting distribution in~\eqref{eq:limiting_representation} is a weighted sum of two dependent Gaussian variables. Motivated by such observation, we propose a simulation-based procedure to obtain the quantile information $\mathbb{Q}_{1-\alpha}(\mathbb{W}_{\mathcal{V}}^{\mathcal{W}}(0))$. For the ease of presentation, we will omit the definition of nuisance estimators and present a simplified algorithm in Algorithm~\ref{alg:sampling_weighting_simplified}. The complete procedure with the estimators of nuisance parameters can be found in Appendix~\ref{sec:bootstrap_algorithm}.
\small
\begin{algorithm}[!ht]
    \SetAlgoNlRelativeSize{0} 

    \nonl \textbf{Input:} Mean estimators $\hat\E[Y_{uN}(s)]$ and $\hat\E[Y_{uN}^2(s)]$; estimators $\hat{\mathrm{Cov}}^{(1)}$ and $\hat{\mathrm{Cov}}^{(2)}(\cdot)$; weighting scheme $\mathcal{W}\in\{\mathcal{A},\mathcal{C}\}$; sampling function $e(s,\cdot)$; clipping rate $l_N$. \\

    \textbf{First stage sampling:} Compute $\hat{\bm{\Sigma}}^{(1)} = (\hat{\mathrm{Cov}}^{(1)})_{2 \times 2}$. Then sample 
	\begin{align*}
		\tilde{A}^{(1,b)} = (\hat{\bm{\Sigma}}^{(1)})^{1/2} S_1^{(b)}\quad\text{where}\quad S_1^{(b)} \sim N(0, \bm{I}_2).
	\end{align*}
    
    \textbf{Second stage sampling:} Compute $\hat{\bm{\Sigma}}^{(2,b)} = (\hat{\mathrm{Cov}}^{(2)}(\tilde{A}^{(1,b)}))_{2 \times 2}$. Then sample 
	\begin{align*}
		\tilde{A}^{(2,b)} = (\hat{\bm{\Sigma}}^{(2,b)})^{1/2} S_2^{(b)}\quad\text{where}\quad S_2^{(b)} \sim N(0, \bm{I}_2).
	\end{align*}

    \textbf{Weighting procedure:} Compute $\hat{\bar{w}}_{\mathcal{W}}^{(t,b)}(s)$ for $\bar{w}_{\mathcal{W}}^{(t,b)}(s)$ with the sample $\tilde{A}^{(t,b)}$ and inputs $ \hat\E[Y_{uN}(s)],\hat\E[Y_{uN}^2(s)],e(s,\cdot)$ and $l_N$. Defining $(\tilde{A}^{(t,b)}(0),\tilde{A}^{(t,b)}(1))^{\top}\equiv\tilde{A}^{(t,b)}$, generate sample
	{\tightdisplay
    \begin{align*}
		\mathcal{D}_{\mathcal{W}}^{(b)} = \sum_{t=1}^2 \hat{\bar{w}}_{\mathcal{W}}^{(t,b)}(0)\, \tilde{A}^{(t,b)}(0) - \sum_{t=1}^2 \hat{\bar{w}}_{\mathcal{W}}^{(t,b)}(1)\,\tilde{A}^{(t,b)}(1).
	\end{align*}}
	For normalized tests, rescale: $\mathcal{D}_{\mathcal{W},\mathcal{N}}^{(b)} = \mathcal{D}_{\mathcal{W}}^{(b)} / (\sum_{s,t} (\hat{\bar{w}}_{\mathcal{W}}^{(t,b)}(s))^2)^{1/2}$.

    \textbf{Repeated sampling:} Repeat steps 1-3 to get $B$ simulation samples.

    \nonl \textbf{Output:} Simulation sample $\{\mathcal{D}_{\mathcal{W}}^{(b)}:b\in[B]\}$ (or $\{\mathcal{D}_{\mathcal{W},\mathcal{N}}^{(b)}\}$ for normalized tests).
    \caption{Simplified two-stage sampling and weighting procedure}
    \label{alg:sampling_weighting_simplified}
\end{algorithm}
\normalsize

We make the following remarks regarding the proposed simulation-based procedure.

\begin{remark}[Computational efficiency]
	The computational cost of generating a simulation sample in Algorithm~\ref{alg:sampling_weighting_simplified} is $O(1)$, leading to a total cost of $O(B)$ 
	for the whole procedure. This is much more efficient than the general nonparametric bootstrap method
	even with linear computation cost for computing each test statistic, which takes $O(BN)$ 
	to generate $B$ samples.
\end{remark}

\begin{remark}[A nonparametric simulation-based approach]
	Notably, Algorithm~\ref{alg:sampling_weighting_simplified} mimics a two-stage asymptotic experiment to approximate the limiting distribution of the test statistic in~\eqref{eq:limiting_representation}. From this perspective, our approach is conceptually aligned with the simulation-based methods that leverage asymptotic representation results, as proposed in \citet{Hirano2023}. Our method is nonparametric and free of distributional assumptions, thereby enabling more robust and widely applicable inference.
\end{remark}
\paragraph{Asymptotically valid simulation-based tests.}

Based on Algorithm~\ref{alg:sampling_weighting_simplified}, we construct the simulation-based tests using the simulation samples $\mathcal{D}_{\mathcal{W}}^{(b)}$ (and its normalized version $\mathcal{D}_{\mathcal{W},\mathcal{N}}^{(b)}$). Denote the $\sigma$-algebra generated by the observed data as $\mathcal{G}_N\equiv \sigma(\mathcal{H}_1\cup\mathcal{H}_2)$. Writing $\mathcal{D}_{\mathcal{W},\mathcal{U}}\overset{d}{=}\mathcal{D}_{\mathcal{W}}^{(b)}$ and $\mathcal{D}_{\mathcal{W},\mathcal{N}}\overset{d}{=}\mathcal{D}_{\mathcal{W},\mathcal{N}}^{(b)}$, conditional on $\mathcal{G}_N$, define for $\mathcal{W} \in \{\mathcal{A}, \mathcal{C}\}$: 
\begin{align*}
	\hat{\phi}_{\mathcal{U}}^{\mathcal{W}} \equiv \indicator\big(\sqrt{N} T_N > \mathbb{Q}_{1 - \alpha}(\mathcal{D}_{\mathcal{W}, \mathcal{U}}\mid \mathcal{G}_N)\big)\quad\text{and}\quad \hat{\phi}_{\mathcal{N}}^{\mathcal{W}} \equiv \indicator\big(W_N > \mathbb{Q}_{1 - \alpha}(\mathcal{D}_{\mathcal{W},\mathcal{N}}\mid \mathcal{G}_N)\big).
\end{align*}
The following theorem shows the validity of the simulation-based procedure and resulting tests. 

\begin{theorem}[Validity of simulation-based tests $\hat \phi_{\mathcal{V}}^{\mathcal{W}}$]\label{thm:bootstrap}
	Suppose Assumption \ref{assu:moment_condition}-\ref{assu:sampling_design} hold. Then, the following statements hold.
	\begin{enumerate}
		\item Suppose Assumption \ref{assu:constant_weighting} holds, we have for $\mathcal{V}\in \{\mathcal{U},\mathcal{N}\}$,
		{\tightdisplay
		\begin{align*}
			\sup_{x\in\mathbb{R}}\left|\P\left[\mathcal{D}_{\mathcal{C},\mathcal{V}}\leq x|\mathcal{G}_N\right]-\P[\mathbb{W}_\mathcal{V}^{\mathcal{C}}(0)\leq x]\right|\convp 0\quad\text{and}\quad\lim_{N\rightarrow\infty}\E_{H_{0N}}[\hat \phi_{\mathcal{V}}^{\mathcal{C}}]=\alpha;
		\end{align*}}
		\item Suppose Assumption \ref{assu:adaptive_weighting} holds, we have for $\mathcal{V}\in \{\mathcal{U},\mathcal{N}\}$,
		{\tightdisplay
		\begin{align*}
			\sup_{x\in\mathbb{R}}\left|\P\left[\mathcal{D}_{\mathcal{A},\mathcal{V}}\leq x|\mathcal{G}_N\right]-\P[\mathbb{W}_\mathcal{V}^{\mathcal{A}}(0)\leq x]\right|\convp 0\quad\text{and}\quad\lim_{N\rightarrow\infty}\E_{H_{0N}}[\hat \phi_{\mathcal{V}}^{\mathcal{A}}]=\alpha.
		\end{align*}}
	\end{enumerate}
	\edit{
	Furthermore, under Assumption~\ref{assu:sampling_design}(1), we have: for any $\mathcal{W}\in \{\mathcal{A},\mathcal{C}\}$ and $\mathcal{V}\in\{\mathcal{U},\mathcal{N}\}$,
	\begin{align*}
		\sup_{x\in\mathbb{R}}\left|\P\left[\mathcal{D}_{\mathcal{W},\mathcal{V}}\leq x|\mathcal{G}_N\right]-\P[\mathbb{W}_\mathcal{V}^{\mathcal{W}}(0)\leq x]\right|=O_p(N^{-1/4}).
	\end{align*}
	}
\end{theorem}

The proof of Theorem~\ref{thm:bootstrap} can be found in Appendix~\ref{sec:proof_bootstrap}. 
There is an implicit assumption we make behind Theorem \ref{thm:bootstrap}, 
which is the knowledge of the sampling function $e(s,\cdot)$ and $l_N$. In practice, this assumption is reasonable since the selection algorithm or the experimental protocol is usually pre-specified before the data is collected and thus is at the hand of the experiment designer. In fact, this is the case in many empirical studies including \citet{collins2007multiphase,li2010contextual,offer2021adaptive,hannah_a_jin_2023_8192805}.

\edit{
	A direct corollary of Theorem~\ref{thm:quantitative_CLT_W_N} and Theorem~\ref{thm:bootstrap} is the convergence rate of the Type-I error towards the nominal level when Lipschitz sampling function is used. The key step is a standard anti-concentration inequality relating the bounded-Lipschitz distance to the Kolmogorov distance: $\sup_{x}|F_X(x)-F_Y(x)|\lesssim d_{\mathrm{BL}}(X,Y)^{1/2}$, which converts the $d_{\mathrm{BL}}$ rates from Theorems~\ref{thm:quantitative_CLT_W_N} and~\ref{thm:bootstrap} into uniform CDF approximation rates.
	\begin{corollary}[Type-I error rate of simulation-based test]\label{cor:type-I-error-rate}
		Suppose the conditions in Theorem~\ref{thm:quantitative_CLT_W_N} hold. Then, the following statements hold. 
		\begin{enumerate}
			\item Suppose Assumption~\ref{assu:constant_weighting} holds. Then, $|\E_{H_{0N}}[\hat \phi_{\mathcal{V}}^{\mathcal{C}}]-\alpha|=O(N^{-1/4})$;
		    \item Suppose Assumption~\ref{assu:adaptive_weighting} holds. Then, $|\E_{H_{0N}}[\hat \phi_{\mathcal{V}}^{\mathcal{A}}]-\alpha|=O((l_N^{-1/2}N^{-1/2}+l_N^{1/2})^{1/2})$.
		\end{enumerate}
	\end{corollary}
}

\section{Finite-sample evaluation}\label{sec:finite-sample}

In this section, we conduct extensive numerical simulations and a semi-synthetic data analysis to investigate the finite-sample performance of the tests studied in the previous section. 
In particular, we include simulation-based tests with two scaling (normalized and unnormalized) and three weighting schemes ($m=0$, $m=1/2$, and $m=1$) proposed in Section~\ref{sec:bootstrap_procedure}. The significance level is taken to be $0.05$ throughout this section.

\subsection{Numerical simulation}\label{sec:simulation}

\paragraph{Data generation procedure.} \edit{We consider the following five potential outcome distributions:
\begin{enumerate}
	\item \textbf{Gaussian:} $Y_{uN}(0)\sim N(\theta,1),\ Y_{uN}(1)\sim N(0,0.25)$;
	\item \textbf{Bernoulli:} $Y_{uN}(0)\sim \mathrm{Bern}(\theta+0.5),\ Y_{uN}(1)\sim \mathrm{Bern}(0.5)$;
	\item \textbf{Poisson:} $Y_{uN}(0)\sim \mathrm{Pois}(1+\theta),\ Y_{uN}(1)\sim \mathrm{Pois}(1)$;
	\item \textbf{Student:} $Y_{uN}(0)\sim \theta + t(4),\ Y_{uN}(1)\sim t(10)$;
	\item \textbf{Mixture Normal:} $Y_{uN}(0)\sim 0.5\, N(\theta - 1,1) + 0.5\, N(\theta + 1,1),\ Y_{uN}(1)\sim 0.5\, N(-1,1) + 0.5\, N(1,1)$.
\end{enumerate}}
In the pilot stage, we employ the equal sampling in the pilot stage $e(1)=0.5$, which mimics the common practice in real world when there is no prior information which treatment is better. Inspired by batched bandit setup, we consider two selection algorithms: the modified version of Thompson sampling~\eqref{eq:modified-cliped-TS} and $\varepsilon$-greedy algorithm~\eqref{eq:eps-greedy}, both with clipping $l_N=\varepsilon/2$. In the second stage, we sample two treatments based on the results of these selection algorithms. The task is to test if the treatment effect is different from $0$, i.e. $\theta=\E[Y_{uN}(0)]-\E[Y_{uN}(1)]=0$, or not.

\edit{In addition to the four tests with IPW weighting ($m=0$ and $m=1/2$, each with normalized and unnormalized statistics) and sample splitting, we include tests using adaptive weighting with $m=1$ (the difference-in-means exponent; see Appendix~\ref{sec:extension_m_1}), yielding two additional IPW tests (normalized and unnormalized). We also include two non-IPW methods as benchmarks: the Batched Difference-in-Means (BDM) test~\citep{Zhang2020} and a concentration-inequality-based test~\citep{abbasi2011improved}. The BDM test computes within-batch difference-in-means and combines them, avoiding inverse probability weighting altogether. The concentration test constructs confidence intervals using Hoeffding-type concentration inequalities, providing finite-sample validity at the cost of conservativeness. In total, we compare $9$ tests. We refer the readers to Appendix~\ref{sec:additional_simulation} for additional details.}

\paragraph{Parameter setup.} In both selection algorithms, we set \edit{$\varepsilon \in\{ 0.05, 0.1, 0.2\}$. The signal strength $\theta$ varies over a distribution-dependent range: $\theta \in [-0.3, 0.3]$ for Gaussian, Poisson, and Mixture Normal; $\theta \in [-0.4, 0.4]$ for Student; and $\theta \in [-0.2, 0.2]$ for Bernoulli, each with $9$ equally spaced grid points, and $\theta = 0$ corresponds to the null hypothesis. The number of total samples $N = 1000$ and each batch has the same sample size $N_1 = N_2 = 500$. For negative values of $\theta$, we use the left-sided test, and for positive values, we use the right-sided test.}

\begin{figure}[!ht]
	\centering
	\includegraphics[width=0.85\textwidth]{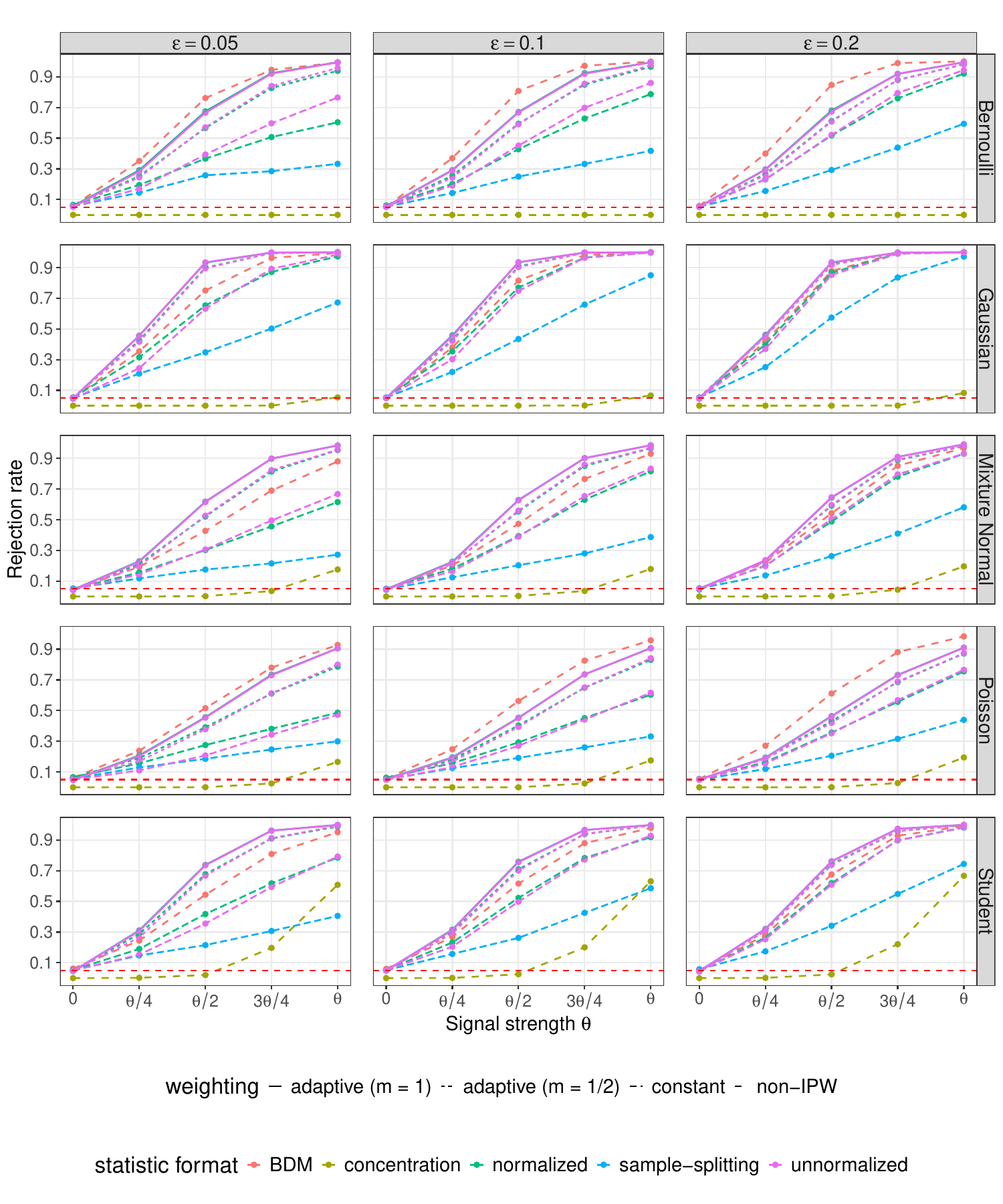}

	\caption{\edit{Right-sided rejection rate for the $9$ tests across five distributions under Thompson sampling. The signal strength is normalized to $[0, \theta]$. The simulation is repeated for $2000$ times with $5000$ simulation samples per test (same below).}}
	\label{fig:simulation-rejection-plot-thompson}
\end{figure}

\begin{figure}[!ht]
	\centering
	\includegraphics[width=0.85\textwidth]{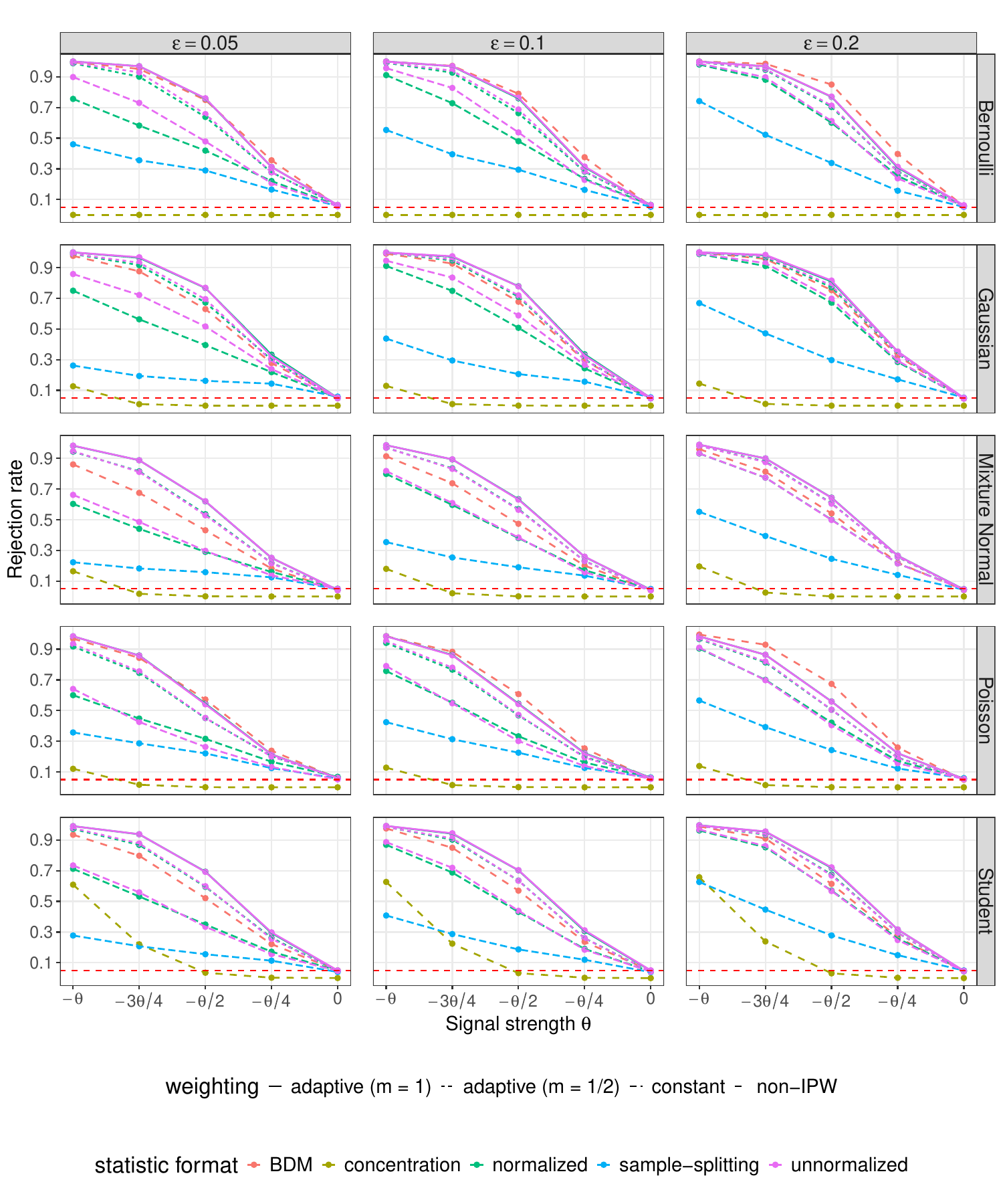}

	\caption{\edit{Left-sided rejection rate for the $9$ tests across five distributions under Thompson sampling.}}
	\label{fig:simulation-rejection-plot-thompson-left}
\end{figure}

\paragraph{Results analysis and interpretation.}

For the ease of presentation, we select the representative results with Thompson sampling being the selection algorithm. For additional simulation results, we refer the readers to Appendix \ref{sec:additional_simulation}. From Figure~\ref{fig:simulation-qq-plot-thompson} in Appendix~\ref{sec:thompson-sampling-simulation}, we observe that all the tests can produce relatively well-calibrated $p$-values. This validates the simulation-based procedure proposed in Algorithm \ref{alg:sampling_weighting_simplified}. The power results are summarized in Figures~\ref{fig:simulation-rejection-plot-thompson} and~\ref{fig:simulation-rejection-plot-thompson-left} for right-sided and left-sided tests, respectively. It is unsurprising to see that our approach using pooled two-stage data can improve power over sample splitting. Moreover, we make the following intriguing observations.

\edit{\paragraph{Power comparison between non-normal and normal-limit tests.} Our extensive simulation reveals a nuanced picture of the power comparison between tests based on non-normal limits and the BDM test, which relies on a normal limit. Under discrete potential outcomes (Bernoulli, Poisson), the BDM test can be more powerful than our methods, while under continuous potential outcomes (Gaussian, Student, Mixture Normal), our tests with adaptive weighting achieve higher power, consistent with the Gaussian case studied in~\citet{Hirano2023}. However, this pattern is not universal: in the semi-synthetic data analysis of Section~\ref{sec:semi-sythetic-data}, which uses discrete but sparse outcomes, the BDM test has lower power than the adaptive IPW methods.}

\edit{\paragraph{Adaptive weighting improves power and is robust to $\varepsilon$.} Tests with adaptive weighting ($m=1/2$ and $m=1$) show substantial power improvement compared to tests with constant weighting, regardless of which scaling is used. The improvement is especially pronounced when $\varepsilon$ is small, because constant weighting reduces the WIPW estimator to the usual IPW estimator, which is highly variable when the downsampling on one arm in the second stage is substantial. Furthermore, tests based on constant weighting (including sample splitting) are more sensitive to $\varepsilon$, whereas the adaptive weighting scheme adjusts the weights based on the observed data, making it robust to the choice of $\varepsilon$. In experimental practice, employing the adaptive weighting scheme for inference enables more aggressive exploitation strategies in sampling.}

\edit{\paragraph{Other empirical insights.} Unlike randomized controlled experiments, normalization can make a difference in the power of the tests, even asymptotically (see Theorem \ref{thm:weak_convergence_W_N} on results with $T_N$ and $W_N$). Additionally, the power performance differs between left-sided and right-sided tests. Under Bernoulli and Mixture Normal outcomes, the left-sided test tends to reject more often than the right-sided test at the same absolute signal magnitude across all methods. The asymmetry is less pronounced under Gaussian outcomes. These observations highlight the potential need for side-dependent experimental design strategies.}

\subsection{Semi-synthetic data analysis}\label{sec:semi-sythetic-data}

To further investigate the performance of different tests on real data,  
we conduct a semi-synthetic data analysis designed to better mimic real-world settings.  
The data is derived from the Systolic Blood Pressure Intervention Trial (SPRINT) \citep{ambrosius2014design}. This is a randomized controlled trial and evaluates whether a new treatment program for lowering systolic blood pressure reduces the risk of cardiovascular disease (CVD). The population is divided into treatment (new treatment) and control (placebo) group and primary clinical outcome is the occurrence of a major CVD event. We use permutation to generate semi-synthetic data from the original SPRINT data, which allows us to control the signal strength and evaluate the performance of different tests under various conditions. Given the potential benefits of the new treatment for patients, we adopt a two-stage experimental design and apply the $\varepsilon$-greedy algorithm~\eqref{eq:eps-greedy} to adaptively adjust data sampling in the follow-up stage. We refer the readers to Appendix~\ref{sec:semi-synthetic-data-generation} for more details on the data generation procedure.

\begin{figure}[!ht]
	\centering  
	\begin{subfigure}{\textwidth}
	  \centering
	  \includegraphics[width=0.85\textwidth]{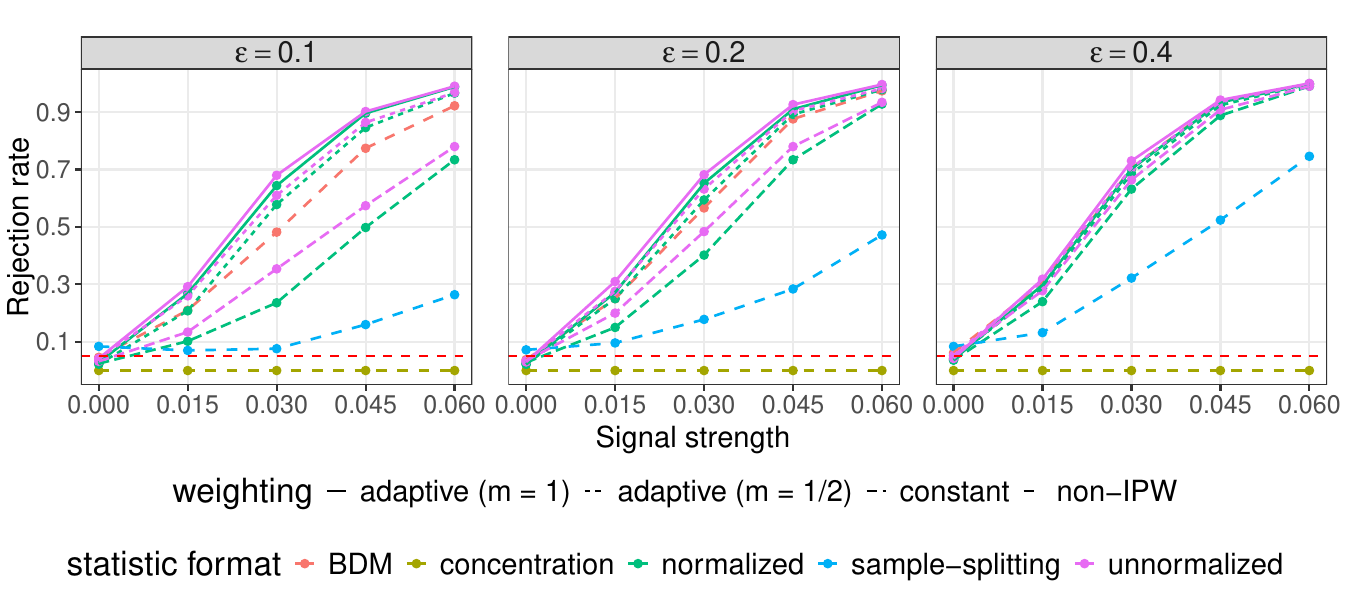}
	\end{subfigure}
	\caption{\edit{Type-I error and power for the nine tests under semi-synthetic data.}}
	\label{fig:semi-synthetic-data}
\end{figure}

The results are presented in Figure \ref{fig:semi-synthetic-data}. Perhaps a bit surprisingly, we observe that the test based on sample splitting suffers from Type-I error inflation in the absence of signal. This issue is primarily due to the sparsity of the outcome: the average rate of CVD occurrence is less than $0.1$. Such sparsity may prevent the central limit theorem from taking effect, posing a particular challenge for sample splitting, which uses only half of the data. We also find that the choice of $\varepsilon$ influences calibration performance: in the $\varepsilon$-greedy algorithm, a smaller $\varepsilon$ results in a smaller effective sample size in the second stage. Consequently, Type-I error inflation in the sample-splitting method is mitigated as $\varepsilon$ increases. In contrast, our methods control Type-I error well and $\varepsilon$ has a much smaller effect on these tests, which combine data from both stages and thus have higher effective sample size. A further investigation of calibration performance is provided in Appendix~\ref{sec:additional_semi_synthetic}. Regarding power, the benefit of adaptive sampling remains evident compared to the sample-splitting approach. Among the WIPW tests proposed in this paper, the unnormalized test exhibits slightly higher power than the normalized test, although the difference is modest. \edit{The BDM test controls Type-I error but has lower power than the adaptive IPW methods. This contrasts with the Bernoulli simulation in Section~\ref{sec:simulation}, where the BDM test is more powerful under high success rates. The reversal arises because the semi-synthetic outcomes are discrete but sparse (CVD event rate $< 0.1$), suggesting that the relative advantage of normal-limit versus non-normal-limit tests depends critically on the sparsity structure of the outcomes, not merely their discreteness. This further highlights the importance of understanding the asymptotic behavior of these test statistics, as our weak convergence results provide the necessary tools for power analysis and method selection. The concentration test is extremely conservative, producing all $p$-values equal to $1$ across all settings. A further investigation of this phenomenon is provided in the Bernoulli small success rate simulation in Appendix~\ref{sec:bernoulli_small_p}.}

\section{Conclusion and discussion}\label{sec:discussion}

In this paper, we establish a set of general and assumption-lean weak convergence results for the WIPW estimator under a two-stage adaptive sampling scheme. These results are largely agnostic to any specific outcome distribution, allowing for broad applicability across a wide range of potential outcomes. Moreover, they accommodate a broad spectrum of signal strengths, making them especially useful for downstream hypothesis testing. To facilitate asymptotically valid tests based on these weak convergence results, we propose a simulation-based procedure for obtaining critical values under the null that is highly scalable. Finally, we validate our theoretical claims through extensive numerical simulations and a semi-synthetic data analysis, illustrating strong finite-sample performance of the proposed tests. 

There are two directions that can be directly pursued with the results and techniques developed in this paper.
\begin{itemize}
	\item \textbf{Experimental design:} In practice, designing the adaptive experiments often requires balancing statistical goals (e.g., power) with non-statistical considerations (e.g., regret, welfare) under budget constraints. Investigating the optimal design of adaptive experiments that trade off these competing objectives is a compelling direction for future study. See, for instance, recent work on adaptive experimental design in \citet{che2023adaptive,liang2023experimental,simchi2023multi,li2024double}. Our results on limiting distributions and proposed simulation-based procedure simplify power calculations, which in turn can inform the design of adaptive experiments. 
	
	\item \textbf{Covariate adjustment:} In randomized controlled experiments, it is well established that appropriately adjusting for predictive covariates can improve the efficiency of statistical inference and increase the power of hypothesis testing \citep{lin2013agnostic}. It would be valuable to explore how such covariate adjustments can be incorporated into the analysis of data collected from adaptive experiments, and how they may enhance the efficiency of the proposed tests. These investigations require the study of asymptotic efficiency of different tests and the techniques developed in this paper may provide a useful starting point for exploring this line of research. In Appendix \ref{sec:extension_augmentation}, we present preliminary results on augmenting the WIPW test statistics.
\end{itemize}

There are several limitations in our current work that point to promising directions for future research. We summarize them below.

\begin{itemize}
	
	\item \textbf{Beyond two-stage experiments:} In this paper, we focus on a two-stage adaptive sampling scheme. However, other adaptive designs exist that fall outside this framework, such as fully adaptive sampling schemes \citep{lai1985asymptotically}, early-dropping experiments~\citep{sampson2005drop,sill2009drop} and experiments with adaptive stopping rules \citep{bauer1994evaluation}. We have sketched the extension of our results to the latter two classes of adaptive sampling strategies in Section~\ref{sec:extension_m_1} and \ref{sec:extension_stopping_time}, respectively. For the fully adaptive sampling schemes, we refer the readers to \citet{khamaru2024inference,han2024ucb,ren2024lai} for recent works on this topic. 
	
	\item \textbf{Statistical optimality:} The statistical optimality of the proposed tests remains an open question. \edit{We investigate this question by providing preliminary comparisons of power between the $m = 1/2$ and $m = 1$ weightings in Appendix~\ref{sec:power-comparison}, where we center the outcomes to isolate the effect of the weighting scheme on power.} It would be interesting to study semiparametrically efficient test statistic for testing the hypothesis $H_{0N}:\E[Y_{uN}(0)] = \E[Y_{uN}(1)]$ under suitable sub-classes of the general nonparametric data generating process. Although \citet{Hirano2023,adusumilli2023optimal} derive power functions via the Neyman-Pearson lemma, devising practical tests that attain the stated optimal power remains an open problem. Furthermore, for more complex scenarios—such as composite alternatives—the corresponding optimality theory remains undeveloped. \edit{Notably, our simulations reveal that normality-based methods (e.g., the BDM test) can be more powerful than non-normal-limit methods under certain outcome distributions, suggesting that the choice of optimal test may depend on the structure of the data generating process. A theory that characterizes when each class of tests is preferred would be of considerable practical value.} 
\end{itemize}

\section*{Acknowledgement}

The authors thank Zijun Gao for sharing the dataset used in the semi-synthetic analyses. The authors also thank the Wharton Student Paper Review Group (SPRG) at the Department of Statistics and Data Science for helpful feedback.

\printbibliography

\newpage
\appendix

\begin{center}
\Large\bfseries Appendix: Supplementary Material
\end{center}
\vspace{1em}

\renewcommand{\cftsubsecnumwidth}{3em}  
\renewcommand{\cftsubsubsecnumwidth}{3.5em}  
\renewcommand{\thetheorem}{S\arabic{theorem}}
\renewcommand{\thelemma}{S\arabic{lemma}}
\renewcommand{\theproposition}{S\arabic{proposition}}
\renewcommand{\thefigure}{S\arabic{figure}}
\renewcommand{\thetable}{S\arabic{table}}
\renewcommand{\thesection}{S\arabic{section}}
\renewcommand{\theequation}{E\arabic{equation}}
\setcounter{theorem}{0}
\setcounter{lemma}{0}
\setcounter{proposition}{0}
\setcounter{figure}{0}
\setcounter{table}{0}
\setcounter{equation}{0}
\setcounter{section}{0}

\paragraph{Notation.}
Throughout the appendix, we will use $(a,b)$ to denote a column vector for $a\in\mathbb{R}^{k},b\in \mathbb{R}^{d}$ when there is no ambiguity. In other words, we use $(a,b)$ to represent $(a^\top,b^\top)^\top$ and we omit the transpose operator for ease of presentation. Thus when we write $f(x,y)$ for $x,y$ as vectors, we will use $f(x,y)$ to denote the function $f((x^\top,y^\top)^\top)$. We use $\mathcal{C}^2(\mathcal{X})$ as the class of functions that are twice continuously differentiable on the domain $\mathcal{X}$. We use $\bm 0_k$ to denote a $k$-dimension vector with each dimension being $0$ and $\bm I_k$ to denote the identify matrix with dimension $k$. We will drop the subscript $k$ if there is no ambiguity. We use $\mathbb{N}_{+}$ to denote the positive natural number. We denote $\partial C_k$ as the boundary set of $C_k$. We denote $\nabla g$ as the gradient of a differentiable function $g$. We denote $a_N\lesssim b_N$ if there exists $c>0$ such that $|a_N/b_N|\leq c$ for large enough $N$. 
\textit{Without loss of generality, we will only prove the results in all the main text when $q_t=1/2$ for $t\in[2]$, i.e., two batches have the same sample size.}

\section{Explicit form of the asymptotic distributions}\label{sec:explicit_form_asymptotic_distribution}

We define the limiting probabilities 
\small
\begin{align}\label{eq:limiting_probabilities}
	H^{(1)}(s)\equiv \lim_{N\rightarrow\infty}e_N(s,\mathcal{H}_0)=e(s),~
	H^{(2)}(s)\equiv \max\Big\{\lim_{N\rightarrow\infty}l_N, e(s,\mathcal{S}^\infty((A^{(1)},V^{(1)}),c))\Big\}.
\end{align}
\normalsize
The function $\mathcal{S}^\infty(x,y):\mathbb{R}^{4}\times \bar{\mathbb{R}}\rightarrow\bar{\mathbb{R}}$ is defined as
\begin{align}\label{eq:def_h_function}
	\mathcal{S}^\infty(x,y)\equiv x_1\cdot x_3^{1/2}/e^{1/2}(0)-x_2\cdot x_4^{1/2}/e^{1/2}(1)+y/\sqrt{2},
\end{align}
where $x=(x_1,x_2,x_3,x_4)^\top\in\mathbb{R}^4$ and $y\in \bar{\mathbb{R}}$.  Also define the scaled asymptotic variance as 
\begin{align}\label{eq:limiting_variance}
	V^{(t)}(s) \equiv \lim_{N\rightarrow\infty}\E[Y_{uN}^2(s)]-H^{(t)}(s)(\lim_{N\rightarrow\infty}\E[Y_{uN}(s)])^2,
\end{align}
and denote $R^{(t)}(s)\equiv (H^{(t)}(s)/V^{(t)}(s))^{1/2}$. Now we define the covariances for $A^{(t)}$ as follows.

\paragraph{Distribution of $A^{(1)}$.}

The covariance $\mathrm{Cov}^{(1)}$ can be defined as
\begin{align}\label{eq:covariance_A_1}
	\mathrm{Cov}^{(1)}\equiv - \left(\frac{H^{(1)}(0)H^{(1)}(1)}{V^{(1)}(0)V^{(1)}(1)}\right)^{1/2}\lim_{N\rightarrow\infty}\left(\E[Y_{uN}(0)]
	\E[Y_{uN}(1)]\right).
\end{align}

\paragraph{Distribution of $A^{(2)}$.}

The asymptotic covariance structure $\mathrm{Cov}^{(2)}(A^{(1)})$ can be written as
\begin{align}\label{eq:covariance_A_2}
	\mathrm{Cov}^{(2)}(A^{(1)})\equiv - \left(\frac{H^{(2)}(0)H^{(2)}(1)}{V^{(2)}(0)V^{(2)}(1)}\right)^{1/2}\lim_{N\rightarrow\infty}\left(\E[Y_{uN}(0)]
	\E[Y_{uN}(1)]\right).
\end{align}
Now we define the weights $\bar w_{\mathcal{W}}^{(t)}(s)$. To this end, we need the following auxiliary random variables,
\begin{align*}
	M_{\mathcal{A}}^{(t)}(s)\equiv q_t\Big(\frac{(H^{(t)}(s))^{1/2}}{\sum_{t=1}^2 q_t (H^{(t)}(s))^{1/2}}\Big)^2\quad\text{and}\quad M_{\mathcal{C}}^{(t)}(s)\equiv q_t.
\end{align*}
Then we can write the weights as 
\begin{align*}
	\bar w_{\mathcal{W}}^{(t)}(s)=\big(M_{\mathcal{W}}^{(t)}(s)/(R^{(t)}(s))^2\big)^{1/2}\quad\text{for any }s\in\{0,1\}\quad\text{and}\quad \mathcal{W}\in\{\mathcal{A},\mathcal{C}\}.
\end{align*}

\edit{
\section{Case study: application to different adaptive experiments}\label{sec:application}

In this section, we demonstrate the applicability of Theorem~\ref{thm:weak_convergence_W_N} to a variety of adaptive experimental designs. Specifically, we focus on two widely used paradigms: \emph{batched bandit experiments} and \emph{subgroup enrichment designs}. These designs have been studied in recent statistical and machine learning literature~\citep{russo2016simple,lin2021inference,che2024optimization,freidling2024selective}.

\paragraph{Batched bandit experiments.}

Batched bandit experiments typically employ adaptive algorithms to balance exploration and exploitation. Two commonly studied strategies are \emph{Thompson sampling} and the \emph{$\varepsilon$-greedy algorithm}. Thompson sampling is a Bayesian approach that selects actions according to their posterior probabilities of being optimal. The $\varepsilon$-greedy algorithm chooses the empirically best arm with probability $1 - \varepsilon/2$ and explores the inferior arm with probability $\varepsilon/2$ when there are two arms. We will show that Theorem~\ref{thm:weak_convergence_W_N} applies to two-batch bandit experiments employing these algorithms.

\begin{itemize}
    \item \textbf{Modified Thompson sampling:}
    Assuming suitable prior distributions, the posterior for each expected outcome $\E[Y_{uN}(s)]$, conditional on pilot data, is (approximately) the normal distribution $N(S_N^{(1)}(s), 1/2)/\sqrt{N_1}$. This yields a sampling function:
	{\tightdisplay
    \begin{align*}
        e(s, x) = (1 - \Phi(x)) \indicator(s = 1) + \Phi(x) \indicator(s = 0),
    \end{align*}}
    which is Lipschitz continuous over $x \in \bar{\mathbb{R}}$ and satisfies the \textbf{Lipschitz condition} of Assumption~\ref{assu:sampling_design}. Incorporating a clipping rate $l_N$ (see Eq.~\eqref{eq:clip}), the follow-up sampling probability $\P[A_{uN}^{(2)} = 0|\mathcal{H}_{1}]$ becomes a modified Thompson sampling rule:
	{\tightdisplay
    \begin{align}\label{eq:modified-cliped-TS}
        \max\{l_N, \min\{1-l_N,\Phi(S_N^{(1)}(0) - S_N^{(1)}(1))\}\}.
    \end{align}}
	This algorithm has been used in~\citet{Hadad2021}.

    \item \textbf{$\varepsilon$-greedy algorithm:}
    Consider the non-smooth sampling function:
	{\tightdisplay
    \begin{align*}
        e(s, x) = \indicator(x < 0)\indicator(s = 1) + \indicator(x \geq 0)\indicator(s = 0),
    \end{align*}}
    which satisfies \textbf{Step-function condition} in Assumption~\ref{assu:sampling_design}. With clipping rate $l_N$, the follow-up sampling probability $\P[A_{uN}^{(2)} = 0|\mathcal{H}_{1}]$ corresponds to an $\varepsilon$-greedy algorithm with $\varepsilon = 2l_N$:
	{\tightdisplay
    \begin{align}\label{eq:eps-greedy}
        (1-l_N)\indicator(S_N^{(1)}(0) \geq S_N^{(1)}(1)) + l_N\indicator(S_N^{(1)}(0) < S_N^{(1)}(1)).
    \end{align}}
\end{itemize}

\paragraph{Subgroup enrichment experiments.}

The assignment variable $A$ may indicate subgroup membership rather than treatment assignment. In this context, $Y_{uN}(0)$ and $Y_{uN}(1)$ represent outcomes for two distinct subgroups. Adaptive enrichment designs aim to identify and focus on the subgroup that benefits more from the treatment, based on interim results from the pilot stage. Common strategies include enrichment based on estimated effect size or interim $p$-value~\citep{us2019adaptive,ben2024adaptive}.

\begin{itemize}
    \item \textbf{Enrichment based on effect size:}
    The sampling function in this case is:
	{\tightdisplay
    \begin{align*}
        e(s, x) = \indicator(x < \beta)\indicator(s = 1) + \indicator(x \geq \beta)\indicator(s = 0),
    \end{align*}}
    where $\beta$ is a pre-specified threshold for the effect size.

    \item \textbf{Enrichment based on interim $p$-values:}
    Let $\hat{\sigma}$ denote the estimated standard deviation of $S_N^{(1)}(0) - S_N^{(1)}(1)$ under the null hypothesis $H_{0N}$. Define left-sided and right-sided $p$-values as $p_l = \Phi(x / \hat{\sigma})$ and $p_r = 1 - p_l$, respectively. Given a pre-specified significance level $\alpha$, the sampling function based on interim $p$-values becomes:
	{\tightdisplay
    \begin{align*}
        e(s, x) = \indicator(p_l < \alpha)\indicator(s = 0) + \indicator(p_r < \alpha)\indicator(s = 1) + \indicator(p_l \in [\alpha, 1 - \alpha]) \cdot 0.5,
    \end{align*}}
    which introduces randomization when the interim result is inconclusive.
\end{itemize}
The follow-up sampling probabilities $\P[A_{uN}^{(2)} = 0|\mathcal{H}_{1}]$ and $\P[A_{uN}^{(2)} = 1|\mathcal{H}_{1}]$ can be similarly obtained based on these sampling functions.
}

\section{Simulation detail for Figures~\ref{fig:sampling_distribution}-\ref{fig:transition}}\label{sec:illustration_simulation}

We present the simulation details for Figure \ref{fig:sampling_distribution} and Figure \ref{fig:transition} respectively. 

\paragraph{Figure \ref{fig:sampling_distribution}.}
To generate Figure \ref{fig:sampling_distribution}, we consider the following potential outcome model: 
\begin{align}\label{eq:simulation-model}
	Y_{uN}(0)\sim N(0,1),\ Y_{uN}(1)\sim N(-c_N/\sqrt{N},9),\ c_N\in\{0,-5,-10,-15\}.
\end{align}
We set the sample size $N=1000$ and the batch size $N_1=N_2=500$. 
The initial sampling $e(0)=e(1)=0.5$ and $\varepsilon$-greedy algorithm is used with $\varepsilon=0.05$. The simulations are repeated $5000$ times.

\paragraph{Figure \ref{fig:transition}.}

We use the covariance structure~\eqref{eq:covariance_A_1} and \eqref{eq:covariance_A_2}, to simulate the limiting distribution~\eqref{eq:limiting_representation}. We set $q_t=1/2$ and assume the knowledge of the first and second moments $\E[Y_{uN}(s)],\E[Y_{uN}^2(s)]$ from model~\eqref{eq:simulation-model}. Set the limiting signal strength $c\in \{0,-5,-10,-15\}$. The simulations are repeated $100,000$ times.

\section{Details of simulation-based algorithm in Section~\ref{sec:bootstrap_procedure}}\label{sec:bootstrap_algorithm}

\subsection{Nuisance parameter estimation}\label{sec:bootstrap_nuisance}

Analogous to the estimator $\hat{\E}[Y_{uN}(s)]$ in \eqref{eq:WIPW_estimator}, we estimate $\E[Y_{uN}^2(s)]$ using:
\small
\begin{align}\label{eq:WIPWS_estimator}
	\WIPWS(s) \equiv \sum_{t=1}^2 \frac{N_th_{N}^{(t)}(s)}{\sum_{t=1}^2 N_t h_{N}^{(t)}(s)} \cdot \frac{1}{N_t} \sum_{u=1}^{N_t} \tilde{\Lambda}_{uN}^{(t)}(s)\quad \text{and}\quad
	\tilde{\Lambda}_{uN}^{(t)}(s) \equiv \frac{\indicator(A_{uN}^{(t)}=s)(Y_{uN}^{(t)})^2}{\P[A_{uN}^{(t)}=s|\mathcal{H}_{t-1}]}.
\end{align}
\normalsize
Using these, we estimate the first-stage variances and covariances as:
\small
\begin{align*}
	\hat{V}^{(1)}(s) = \hat{\E}[Y_{uN}^2(s)] - H^{(1)}(s)(\hat{\E}[Y_{uN}(s)])^2
	\quad \text{and}\quad
	\hat{\mathrm{Cov}}^{(1)} = -\frac{(\bar{H}^{(1)})^{1/2} \hat{\E}[Y_{uN}(0)] \hat{\E}[Y_{uN}(1)]}{(\hat{V}^{(1)}(0) \hat{V}^{(1)}(1))^{1/2}},
\end{align*}
\normalsize
where $\bar{H}^{(1)} \equiv H^{(1)}(0)H^{(1)}(1)$. To define $\hat{\mathrm{Cov}}^{(2)}(\cdot)$, first define $\hat V^{(1)}\equiv (\hat V^{(1)}(0),\hat V^{(1)}(1))$ and consider the function $\hat{H}^{(2)}(x) \equiv \hat{H}^{(2)}(x,0) \cdot \hat{H}^{(2)}(x,1)$, where
\begin{align*}
	\hat{H}^{(2)}(x, s) \equiv  
	\begin{cases}
        \max\{\bar{l}, e(s, \mathcal{S}^\infty((x, \hat{V}^{(1)}), 0))\} & \text{if Assumption \ref{assu:constant_weighting} holds}; \\
        e(s, \mathcal{S}^\infty((x, \hat{V}^{(1)}), 0)) & \text{if Assumption \ref{assu:adaptive_weighting} holds}. 
    \end{cases}
\end{align*}
Furthermore, we define $\hat{V}^{(2)}(x) \equiv \hat{V}^{(2)}(x, 0)\cdot \hat{V}^{(2)}(x, 1)$, where $\hat{V}^{(2)}(x,s) \equiv \hat{\mathbb{E}}[Y_{uN}^2(s)] - \hat{H}^{(2)}(x, s) (\hat{\mathbb{E}}[Y_{uN}(s)])^2$. Finally, we can define the second-stage covariance function as
\begin{align*}
	\hat{\mathrm{Cov}}^{(2)}(x) \equiv -(\hat{H}^{(2)}(x)/\hat{V}^{(2)}(x))^{1/2} \hat{\mathbb{E}}[Y_{uN}(0)] \hat{\mathbb{E}}[Y_{uN}(1)].
\end{align*}

\subsection{Simulation-based algorithm}\label{sec:bootstrap_algorithm_detail}

Now we consider the simulation-based procedure. 

\begin{enumerate}
	\item\textbf{First stage sampling:} Sample $S_1^{(b)} \sim N(0, \bm{I}_2)$ and let $\tilde{A}^{(1,b)} = (\hat{\bm{\Sigma}}^{(1)})^{1/2} S_1^{(b)}$, where $\hat{\bm{\Sigma}}^{(1)} = (\hat{\mathrm{Cov}}^{(1)})_{2 \times 2}$.
	\item \textbf{Second stage sampling:} Sample $S_2^{(b)} \sim N(0, \bm{I}_2)$ and let $\tilde{A}^{(2,b)} = (\hat{\bm{\Sigma}}^{(2,b)})^{1/2} S_2^{(b)}$, where $\hat{\bm{\Sigma}}^{(2,b)} = (\hat{\mathrm{Cov}}^{(2)}(\tilde{A}^{(1,b)}))_{2 \times 2}$.
	\item \textbf{Weighting procedure:} Compute weights $\hat{\bar{w}}_{\mathcal{W}}^{(t,b)}(s)$ by replacing $H^{(1)}(s),H^{(2)}(s)$ and $V^{(1)}(s),V^{(2)}(s)$ in~\eqref{eq:limiting_representation} by $H^{(1)}(s),\hat{H}^{(2)}(\tilde{A}^{(1,b)},s)$, and $\hat{V}^{(1)}(s),\hat{V}^{(2)}(\tilde{A}^{(1,b)},s)$, respectively. Then obtain the simulation sample:
    \[
    \mathcal{D}_{\mathcal{W}}^{(b)} = \sum_{t=1}^2 \hat{\bar{w}}_{\mathcal{W}}^{(t,b)}(0)\, \tilde{A}^{(t,b)}(0) - \sum_{t=1}^2 \hat{\bar{w}}_{\mathcal{W}}^{(t,b)}(1)\, \tilde{A}^{(t,b)}(1),
    \]
	where $\tilde{A}^{(t,b)}(s)$ is the $(s+1)$-th coordinate of $\tilde{A}^{(t,b)}$. For normalized tests, rescale: $\mathcal{D}_{\mathcal{W},\mathcal{N}}^{(b)} = \mathcal{D}_{\mathcal{W}}^{(b)} / (\sum_{s,t} (\hat{\bar{w}}_{\mathcal{W}}^{(t,b)}(s))^2)^{1/2}$.
	\item\textbf{Repeat sampling:} Repeat steps 1-3 for $B$ iterations to obtain $B$ simulation samples.
\end{enumerate}

\section{Additional simulation results}\label{sec:additional_simulation}

\subsection{Additional simulation results with Thompson sampling}\label{sec:thompson-sampling-simulation}

The calibration results are summarized as QQ plots in Figure \ref{fig:simulation-qq-plot-thompson}.

\begin{figure}[!p]
	\centering
	\includegraphics[width=0.95\textwidth]{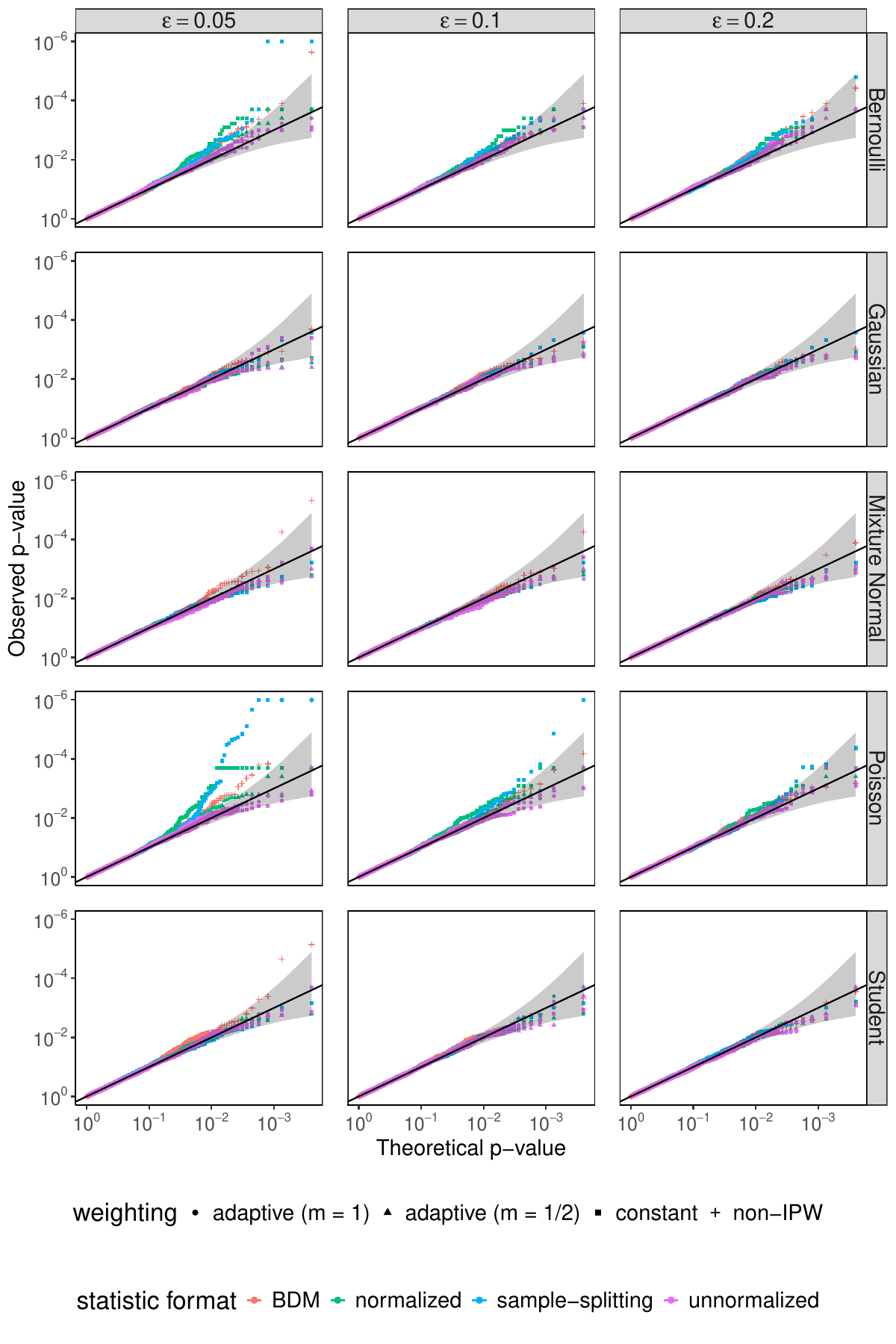}

	\caption{\edit{QQ plots for the $9$ tests (excluding concentration) under Thompson sampling (right-sided). The simulation is repeated for $2000$ times with $5000$ simulation samples per test.}}
	\label{fig:simulation-qq-plot-thompson}
\end{figure}

\edit{
\begin{figure}[!p]
	\centering
	\includegraphics[width=0.95\textwidth]{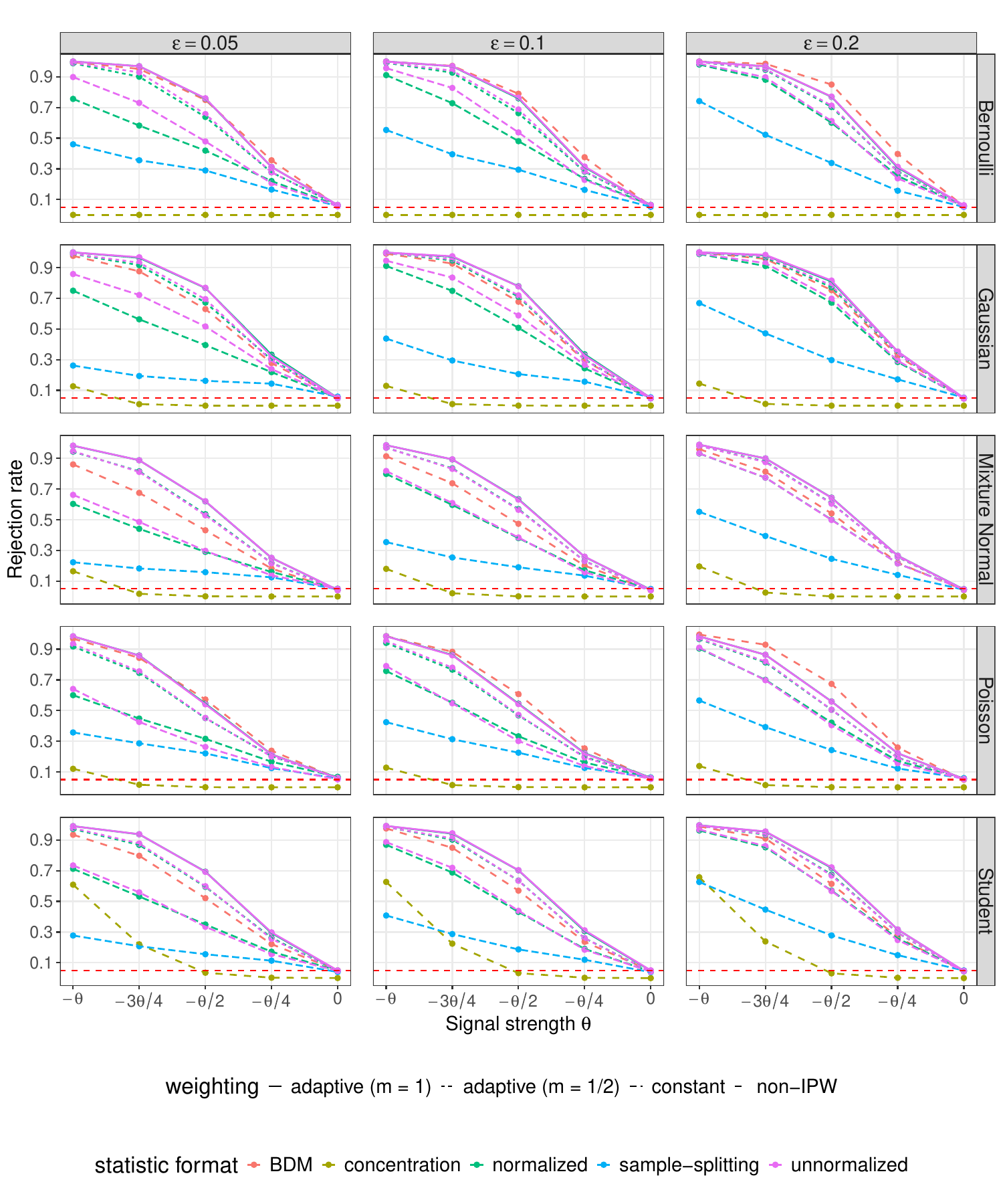}

	\caption{Left-sided rejection rate for the $9$ tests across five distributions under Thompson sampling. The simulation is repeated for $2000$ times with $5000$ simulation samples per test.}
	\label{fig:simulation-rejection-plot-thompson-left}
\end{figure}

\begin{figure}[!p]
	\centering
	\includegraphics[width=0.95\textwidth]{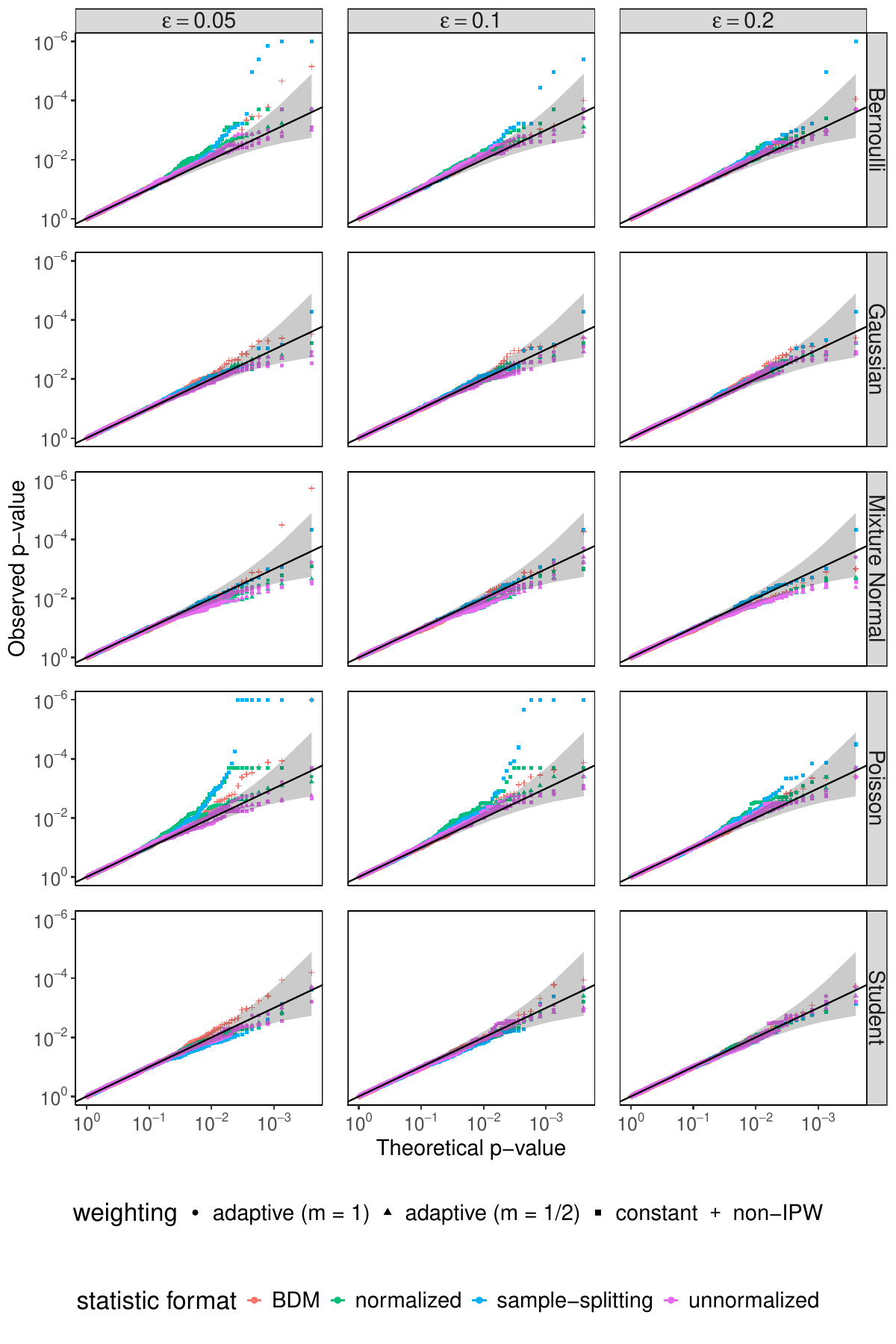}

	\caption{QQ plots for the $9$ tests (excluding concentration) under Thompson sampling (left-sided). The simulation is repeated for $2000$ times with $5000$ simulation samples per test.}
	\label{fig:simulation-qq-plot-thompson-left}
\end{figure}
}

\subsection{Additional simulation results with $\varepsilon$-greedy algorithm}\label{sec:epsilon-greedy-simulation}

\edit{We show the additional results for the simulation in Section \ref{sec:simulation} with $\varepsilon$-greedy selection algorithm applied. The $\varepsilon$ is chosen within $\{0.05, 0.1, 0.2\}$. Results are shown in Figure \ref{fig:simulation-qq-plot-eps-greedy} and \ref{fig:simulation-rejection-plot-eps-greedy}.}

\begin{figure}[!p]
	\centering
	\includegraphics[width=0.93\textwidth]{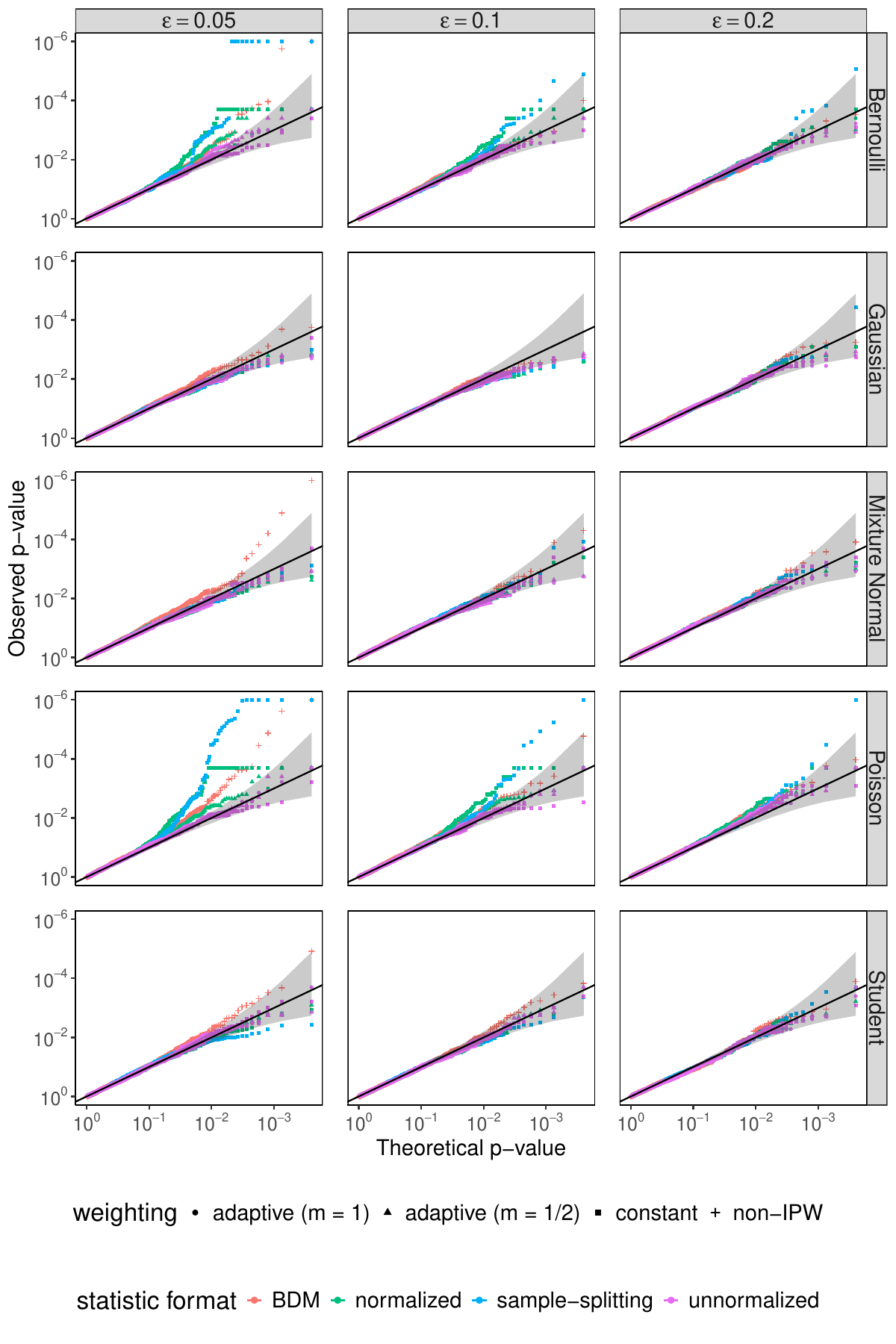}

	\caption{\edit{QQ plots for the $9$ tests (excluding concentration) under $\varepsilon$-greedy sampling (right-sided). The simulation is repeated for $2000$ times with $5000$ simulation samples per test.}}
	\label{fig:simulation-qq-plot-eps-greedy}
\end{figure}

\begin{figure}[!p]
	\centering
	\includegraphics[width=0.93\textwidth]{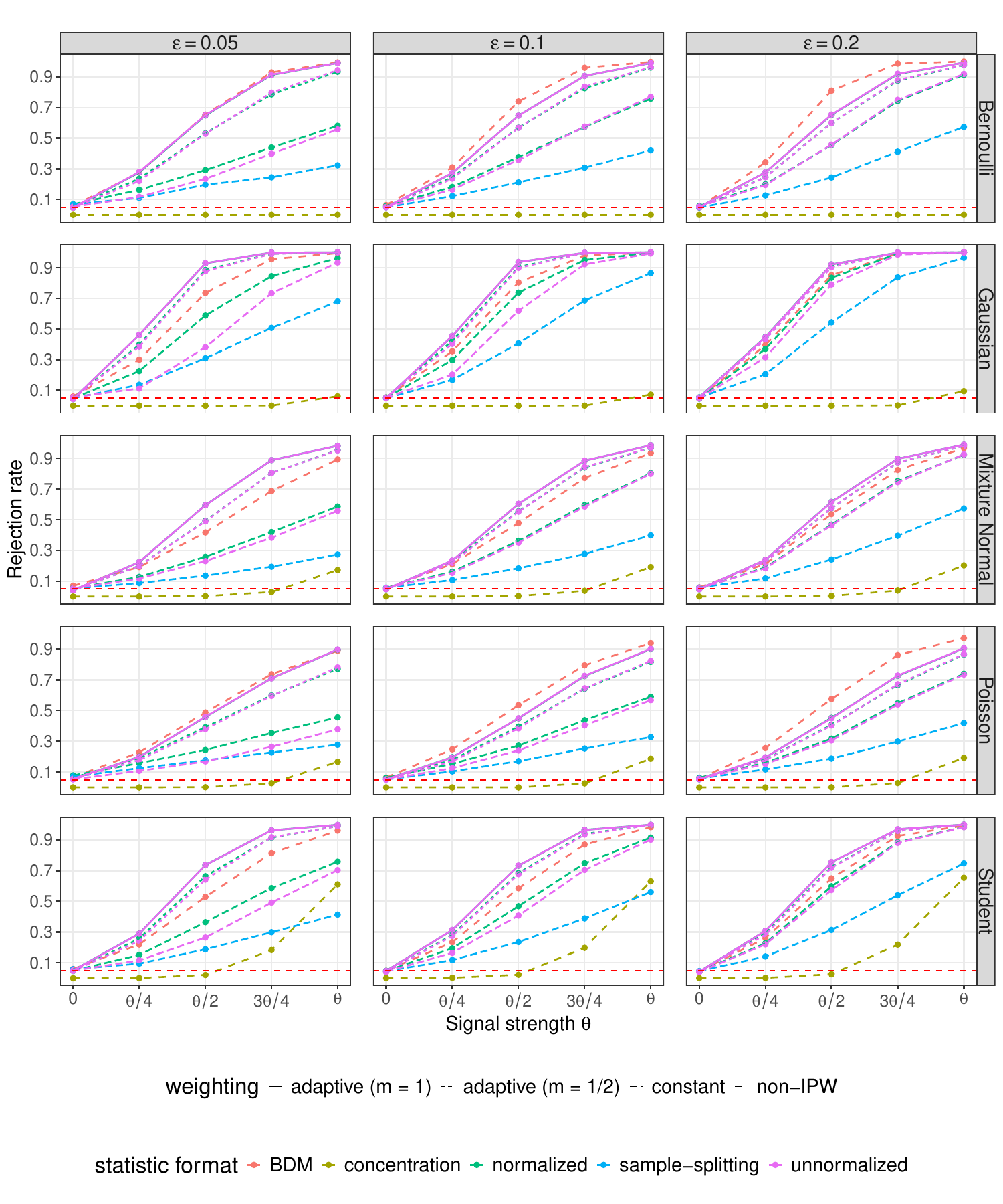}

	\caption{\edit{Right-sided rejection rate for the $9$ tests across five distributions under $\varepsilon$-greedy sampling. The simulation is repeated for $2000$ times with $5000$ simulation samples per test.}}
	\label{fig:simulation-rejection-plot-eps-greedy}
\end{figure}

\edit{
\begin{figure}[!p]
	\centering
	\includegraphics[width=0.93\textwidth]{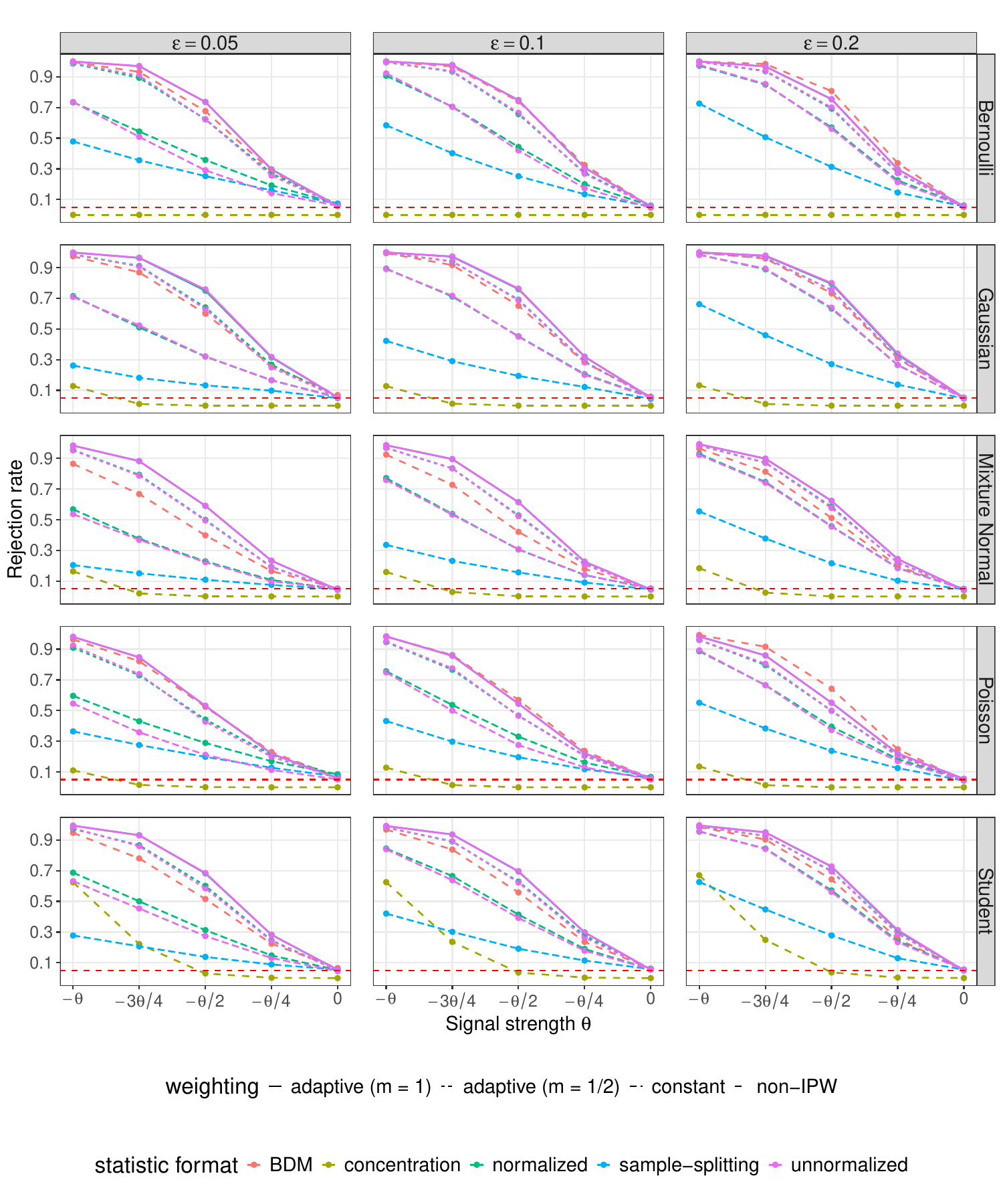}

	\caption{Left-sided rejection rate for the $9$ tests across five distributions under $\varepsilon$-greedy sampling. The simulation is repeated for $2000$ times with $5000$ simulation samples per test.}
	\label{fig:simulation-rejection-plot-eps-greedy-left}
\end{figure}

\begin{figure}[!p]
	\centering
	\includegraphics[width=0.93\textwidth]{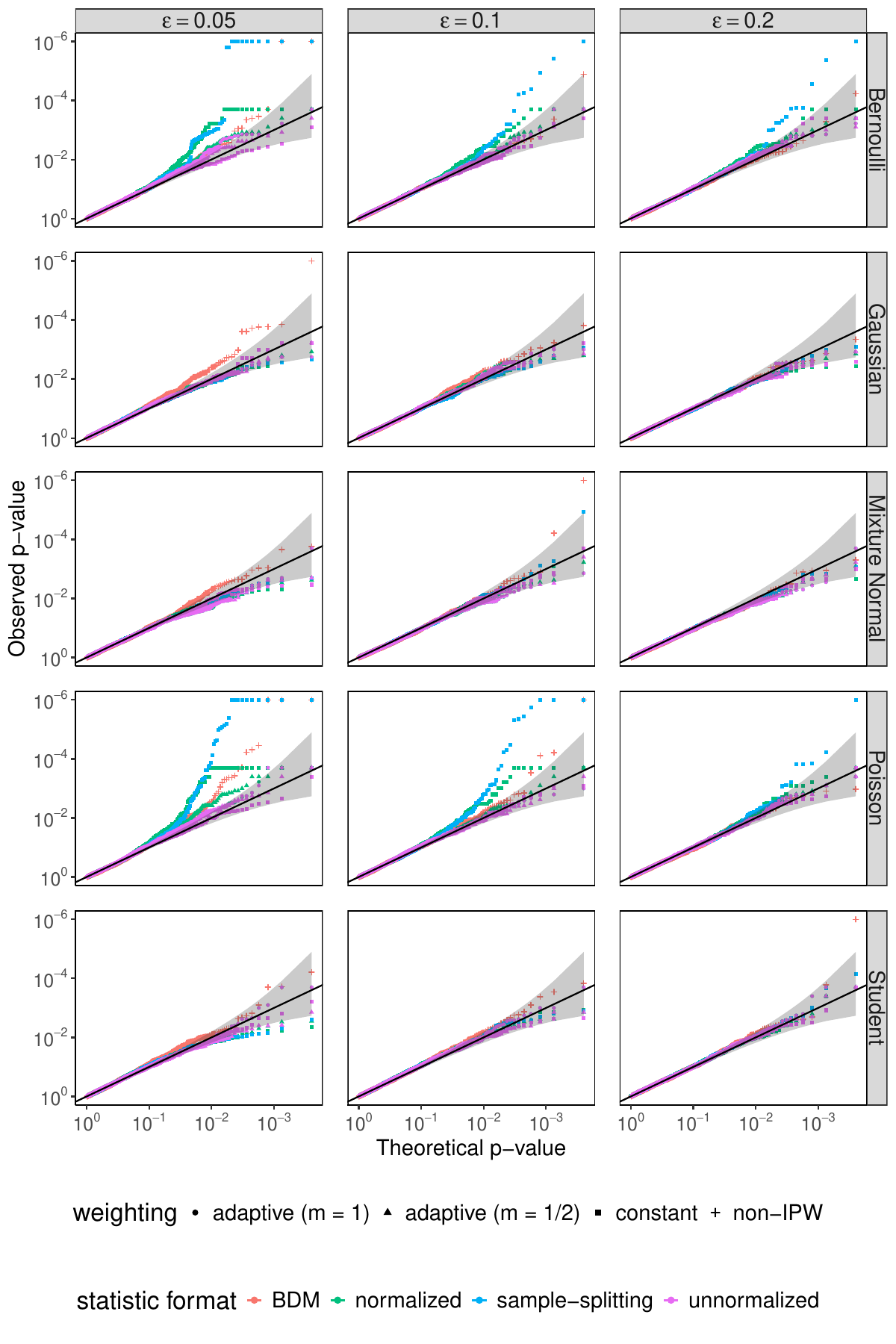}

	\caption{QQ plots for the $9$ tests (excluding concentration) under $\varepsilon$-greedy sampling (left-sided). The simulation is repeated for $2000$ times with $5000$ simulation samples per test.}
	\label{fig:simulation-qq-plot-eps-greedy-left}
\end{figure}
}

\section{Additional details in Section~\ref{sec:semi-sythetic-data}}\label{sec:additional_semi_synthetic}

\subsection{Data generating process}\label{sec:semi-synthetic-data-generation}

We generate the semi-synthetic data, apply the tests, and evaluate their performance through the following procedures.

\begin{enumerate}
    \item \textbf{Permute data to break the dependence.}  
    We first permute the outcomes within the whole population, generating $B = 500$ permuted samples.  
    This permutation effectively removes any treatment effect, ensuring that the treatment and control groups have the same expected outcome level.
    
    \item \textbf{Add signal back to the data.}  
    For these $500$ permuted samples, we manually introduce a treatment effect by increasing the mean outcome (i.e. the major CVD event occurrence) in the control group, 
    since the new treatment is intended to reduce the risk of CVD. Let $N_c^0$ denote the total number of control-group participants who did not experience a CVD event. We set $n_0$ of these zero outcomes to $1$, where $n_0 \sim \mathrm{Bin}(N_c^0,\eta)$. The added signal $\eta$ varies within the set $\{0, 0.015, 0.03, 0.045, 0.06\}$.
    
    \item \textbf{Adaptively sample the data to maximize welfare.}  
    For each permuted sample, we simulate adaptive sampling. We first draw $N_1 = 1000$ random samples. Because the new treatment could be beneficial for the patients, we apply the $\varepsilon$-greedy algorithm~\eqref{eq:eps-greedy} to collect additional $N_2 = 1000$ samples in the second stage, encouraging assignment of new treatment. We vary $\varepsilon \in \{0.1, 0.2, 0.4\}$.
    
    \item \textbf{Evaluate Type-I error control and power.}
    We apply the \edit{nine} tests introduced in Section~\ref{sec:simulation} to the synthetically generated data. We consider the right-sided test to see if the CVD event rate in the control group ($\E[Y_{uN}(0)]$) is higher than that in the treatment group ($\E[Y_{uN}(1)]$). We evaluate Type-I error control before introducing signal and statistical power after introducing the signal.
\end{enumerate}

\subsection{Additional results for semi-synthetic data analysis}

We present additional results for the semi-synthetic data analysis in Section~\ref{sec:semi-sythetic-data}. Following the same procedure outlined in Section~\ref{sec:semi-sythetic-data}, we use $5000$ permuted sample to compute the $p$-values for the \edit{$9$} tests. The QQ-plot is shown in Figure \ref{fig:semi-synthetic-qq-plot}. \edit{The concentration test is excluded from the QQ plot as it produces all $p$-values equal to $1$.} We can see the message is similar to the one in Figure \ref{fig:simulation-qq-plot-eps-greedy} in Section~\ref{sec:semi-sythetic-data} when there is no signal.

\begin{figure}[!htb]
	\centering
	\begin{subfigure}{\textwidth}
		\centering
		\includegraphics[width=0.95\textwidth]{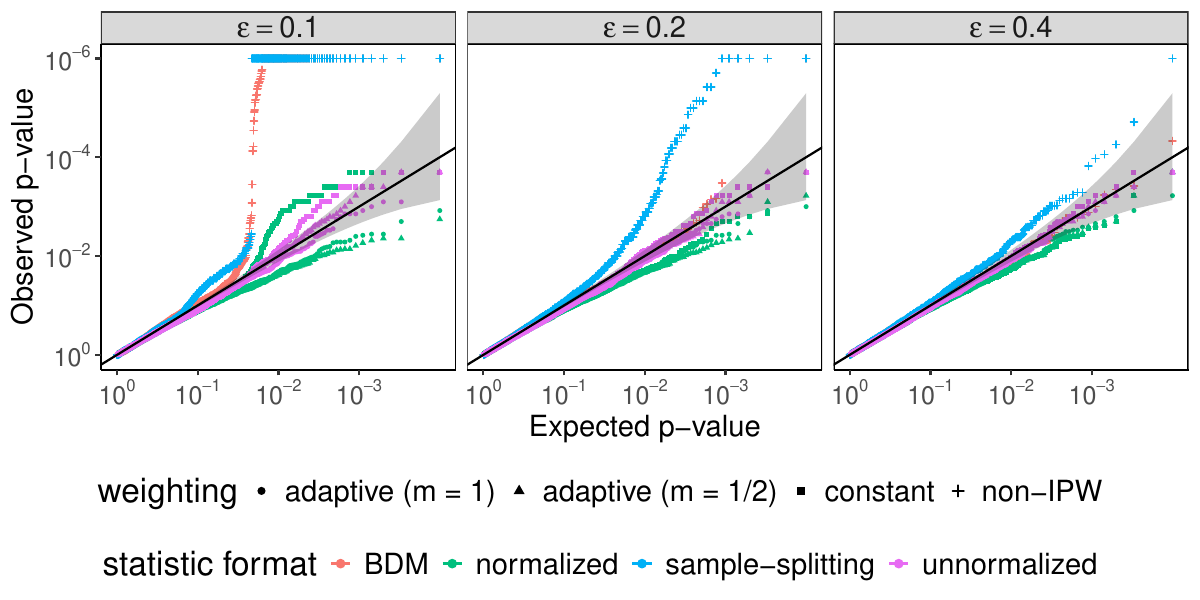}
	\end{subfigure}
	\caption{\edit{QQ-plot for the semi-synthetic data analysis (excluding concentration test).}}
	\label{fig:semi-synthetic-qq-plot}
\end{figure}

\edit{
\section{Bernoulli with small success rate}\label{sec:bernoulli_small_p}

The semi-synthetic data analysis in Section~\ref{sec:semi-sythetic-data} involves a binary outcome with a small baseline success rate (CVD event rate $\approx 0.05$). To better understand inference in this challenging regime, we conduct a dedicated simulation that directly mimics the semi-synthetic setup.

\paragraph{Data generation procedure.}
We generate Bernoulli outcomes with a small baseline success probability $p = 0.05$. The treatment effect is introduced by flipping zeros to ones with probability $\eta$ in the treatment group, giving an effective treatment mean of $0.05 + 0.95\eta$. This mirrors the signal mechanism in the semi-synthetic data analysis, where signal is added by converting non-events to events. We set the total sample size $N = 2000$ ($N_1 = N_2 = 1000$) and vary the signal strength $\eta \in \{0, 0.015, 0.03, 0.045, 0.06, 0.075\}$, matching the semi-synthetic signal range. We consider both Thompson sampling and $\varepsilon$-greedy allocation with $\varepsilon \in \{0.05, 0.1, 0.2\}$.

\paragraph{Results.}
The rejection rate and QQ plots are presented in Figures~\ref{fig:bernoulli-small-p-rejection-thompson}--\ref{fig:bernoulli-small-p-qq-eps-greedy}. Since the signal is one-directional (the treatment can only increase the success rate), we focus on the right-sided test.

\begin{figure}[!htb]
	\centering
	\includegraphics[width=0.95\textwidth]{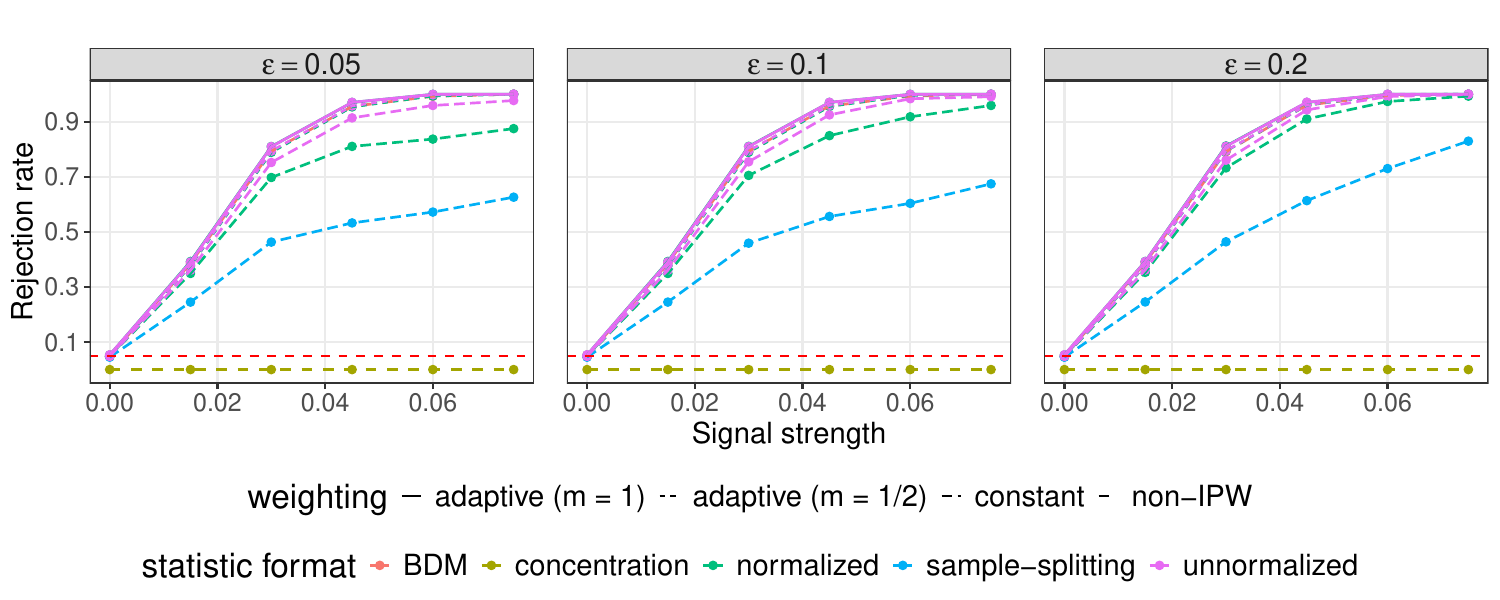}
	\caption{Rejection rate for the Bernoulli small success rate ($p = 0.05$) simulation under Thompson sampling.}
	\label{fig:bernoulli-small-p-rejection-thompson}
\end{figure}

\begin{figure}[!htb]
	\centering
	\includegraphics[width=0.95\textwidth]{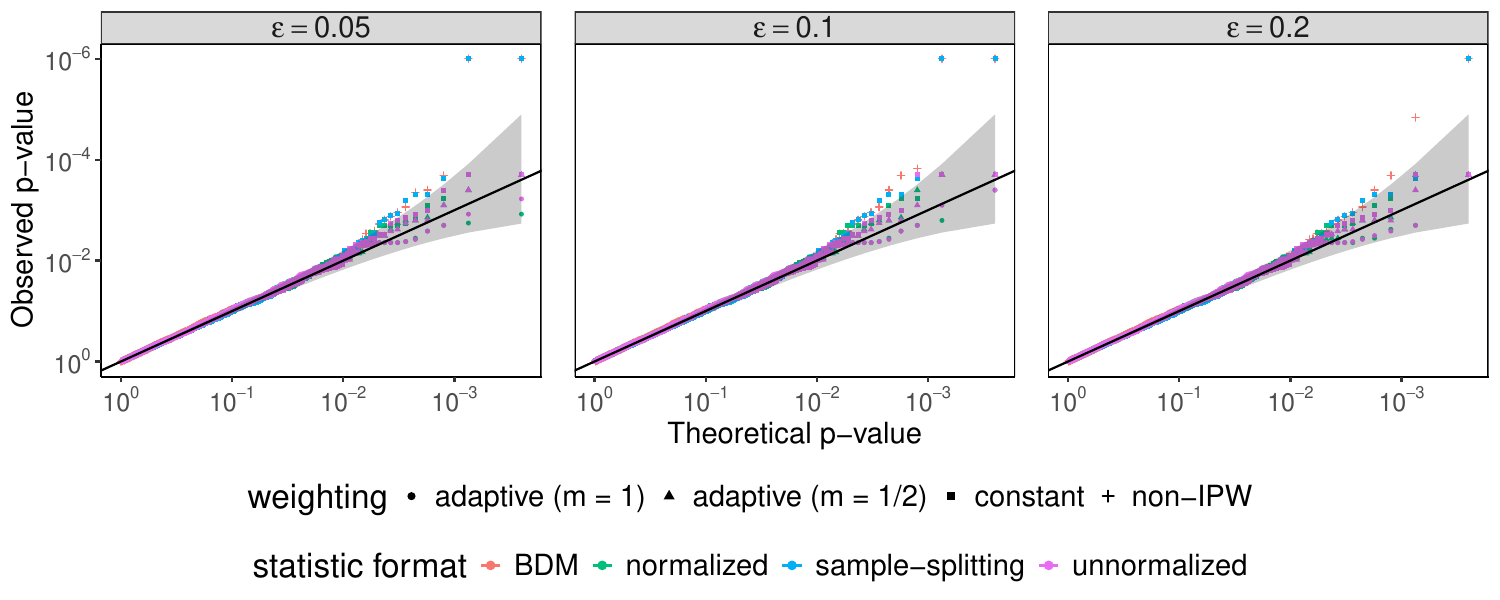}
	\caption{QQ plot for the Bernoulli small success rate ($p = 0.05$) simulation under Thompson sampling (excluding concentration test).}
	\label{fig:bernoulli-small-p-qq-thompson}
\end{figure}

\begin{figure}[!htb]
	\centering
	\includegraphics[width=0.95\textwidth]{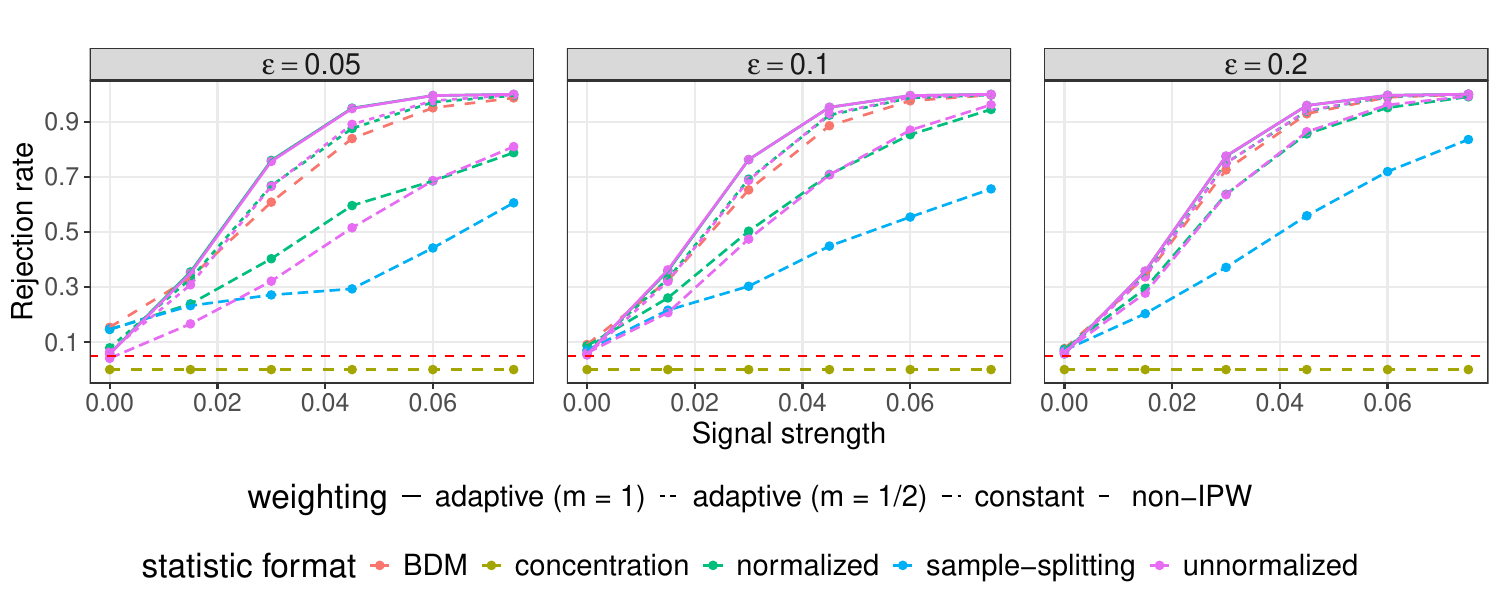}
	\caption{Rejection rate for the Bernoulli small success rate ($p = 0.05$) simulation under $\varepsilon$-greedy sampling.}
	\label{fig:bernoulli-small-p-rejection-eps-greedy}
\end{figure}

\begin{figure}[!htb]
	\centering
	\includegraphics[width=0.95\textwidth]{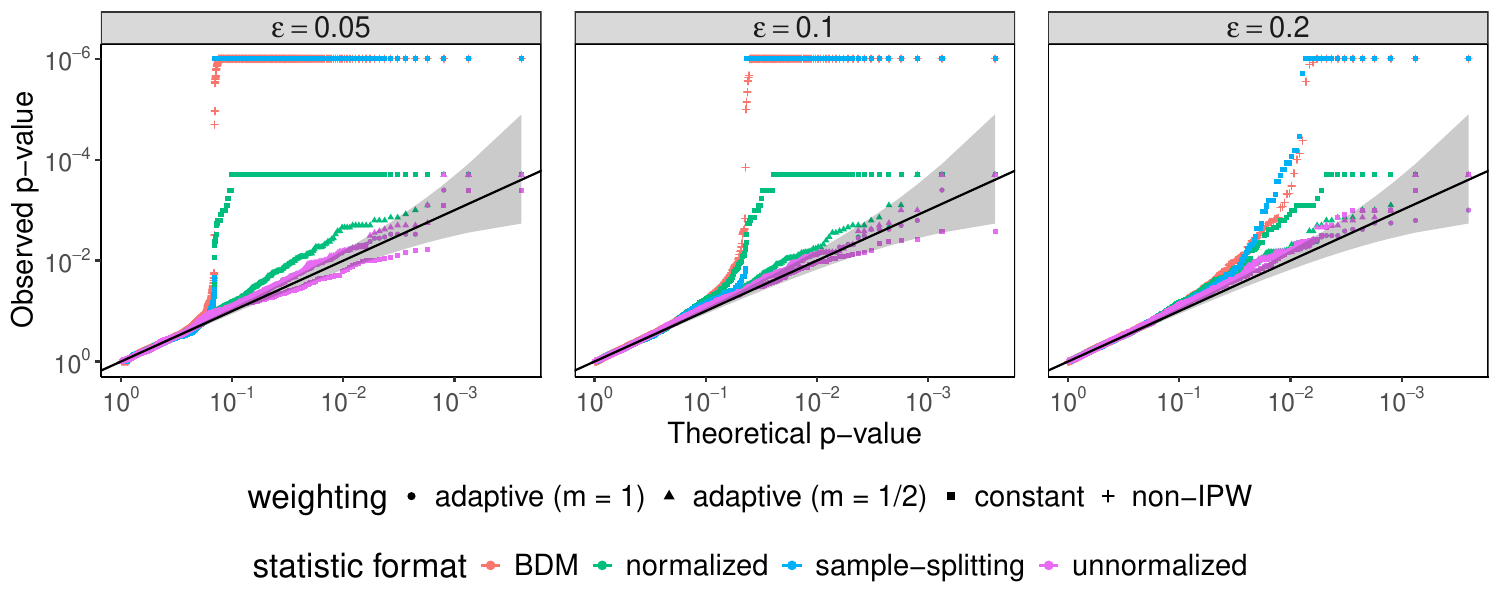}
	\caption{QQ plot for the Bernoulli small success rate ($p = 0.05$) simulation under $\varepsilon$-greedy sampling (excluding concentration test).}
	\label{fig:bernoulli-small-p-qq-eps-greedy}
\end{figure}

The results confirm several findings from the semi-synthetic analysis:
\begin{itemize}
    \item The concentration test is extremely conservative across all settings, producing $p$-values equal to $1$ in virtually all replications. This is expected: concentration-based bounds are designed for worst-case finite-sample validity and are not optimized for power.
    \item The BDM test controls Type-I error but exhibits lower power than the adaptive IPW methods, consistent with the semi-synthetic findings.
    \item Among IPW methods, adaptive weighting provides substantial power gains over constant weighting, and the proposed tests maintain proper calibration even with the sparse binary outcome.
\end{itemize}
}

\section{Intuition from data collection and double-dipping}\label{sec:intuition_limits}

In this section, we discuss the intuition behind the different phases of the limiting distribution, drawing on the data collection procedure outlined in Section~\ref{sec:data_collection_selection}. We argue that the key driver of the various limiting behaviors is the signal strength. To build intuition, consider the $\varepsilon$-greedy selection algorithm as an illustrative example, and suppose we use a fixed clipping rate $l_N = \varepsilon > 0$. Then the updated treatment assignment probability for treatment $s$ after the pilot stage is given by
\small
\[
e_N(s, \mathcal{H}_{1}) = \frac{\varepsilon}{2} \cdot \indicator(D_N(s) < 0) + \left(1 - \frac{\varepsilon}{2} \right) \cdot \indicator(D_N(s) \geq 0),\ D_N(s) \equiv S_N^{(1)}(s) - S_N^{(1)}(1 - s).
\]
\normalsize

When the signal strength $c_N$ converges to a finite constant, the pilot-stage data does not provide sufficiently strong evidence for the $\varepsilon$-greedy algorithm to confidently prefer one arm over the other based on the summary statistic $D_N(s)$. As a result, the assignment remains uncertain, and this added randomness in the selection procedure leads to a non-normal limiting distribution. In contrast, when the signal strength is strong (i.e., $c = -\infty$), the algorithm can confidently conclude that treatment $1$ is superior to treatment $0$ in terms of the expected potential outcome, with ignorable randomness in the selection. In this case, the limiting distribution approaches a normal distribution.

An alternative perspective comes from the concept of ``double-dipping.'' It is widely believed that sample splitting is necessary to avoid selection bias~\citep{fithian2014optimal}. When the signal is weak (i.e., $c$ is finite), the selection procedure exhibits non-negligible randomness, and reusing data without accounting for this selection randomness can be problematic. However, when the signal is strong and the selection becomes deterministic, the impact of ``double-dipping'' becomes negligible. In this regime (i.e., $c = -\infty$), the two-stage data collection process can be viewed as a non-adaptive procedure: sample treatment $0$ ($1$) with probability $e(0)$ ($e(1)$) in the first stage and with probability $\varepsilon/2$ ($1 - \varepsilon/2$) in the second stage. Standard asymptotic inference applies, ensuring asymptotic validity. 

\section{A closer look at literature}\label{sec:inspection_literature}

\subsection{Investigating \citet{Zhang2020}}\label{sec:inspection_Kelly}

They show that the non-normal limiting distribution can happen for classical sample mean statistic under a batched bandit setup (see Figure 1 in their paper). To address this issue, the same paper proposes a batch-wise H\'{a}jek estimator. In particular, for each batch $t\in[2]$, the test statistic can be computed as
\small
\begin{align}\label{eq:batch-wise_Hajek}
	\sqrt{\frac{(\sum_{u=1}^{N_t}A_{uN}^{(t)})(\sum_{u=1}^{N_t}(1-A_{uN}^{(t)}))}{N}}\left(\frac{\sum_{u=1}^{N_t}(1-A_{uN}^{(t)})Y_{uN}^{(t)}}{\sum_{u=1}^{N_t}(1-A_{uN}^{(t)})}-\frac{\sum_{u=1}^{N_t}A_{uN}^{(t)}Y_{uN}^{(t)}}{\sum_{u=1}^{N_t}A_{uN}^{(t)}}-\Delta_n\right)
\end{align}
\normalsize
where $\Delta_n\equiv \E[Y_{uN}(0)]-\E[Y_{uN}(1)]$. A crucial assumption they make to recover the conditional asymptotic normality is that the conditional variance of the observed outcome is constant, i.e., $\mathrm{Var}[Y_{u}^{(t)}|\mathcal{H}_{t-1}]=\textnormal{Cons.}\in (0,\infty)$. This assumption is very stringent and essentially rules out the possible heterogeneity in the distribution of potential outcomes $Y_{u}(0),Y_{u}(1)$. Let us consider a concrete data generating model to illustrate the failure of this assumption. Consider the potential outcome distribution $Y_{u}(0)\sim N(0,1)$ and $Y_u(1)\sim N(0,4)$. Then we can compute, assuming the $\varepsilon$-greedy algorithm between stages, using the consistency assumption
\begin{align*}
	\mathrm{Var}[Y_{u}^{(t)}|\mathcal{H}_{t-1}]
	&
	=\E[(Y_{u}^{(t)})^2|\mathcal{H}_{t-1}]-(\E[Y_{u}^{(t)}|\mathcal{H}_{t-1}])^2\\
	&
	=e(0,\mathcal{H}_{t-1})\E[Y_u(0)^2]+e(1,\mathcal{H}_{t-1})\E[Y_u(1)^2]\\
	&
	=e(0,\mathcal{H}_{t-1})+4e(1,\mathcal{H}_{t-1})\\
	&
	=1+3e(1,\mathcal{H}_{t-1})\\
	&
	=1 + 3\left(1-\frac{\varepsilon}{2}\right)\indicator(S_N^{(1)}(0)<S_N^{(1)}(1))+\frac{3\varepsilon}{2}\indicator(S_N^{(1)}(0)\geq S_N^{(1)}(1)).
\end{align*}

\edit{
\paragraph{On variance standardization.} The computation above shows that the constant conditional variance assumption in \citet{Zhang2020} fails under heterogeneous potential outcome variances. However, arm-specific variance standardization --- dividing by the arm-specific variance $\sigma_s^2$ so that the batch-wise statistics have unit variance per arm --- can circumvent this issue and restore conditional normality for the batch-wise statistics, as observed by \citet{Hirano2023}. Regarding power, \citet{Hirano2023} further show (Section~5.2) that in the Gaussian response setup, tests based on the non-normal pooled statistic are more powerful than those based on batch-wise normalized statistics. Our simulations in Section~\ref{sec:simulation} confirm and extend this finding beyond the Gaussian case: while pooled IPW methods generally outperform batch-wise methods under continuous outcomes, the BDM test (the batch-wise method of \citet{Zhang2020}) can be more powerful under discrete outcomes. This advantage is not universal, however --- under sparse discrete outcomes (Section~\ref{sec:semi-sythetic-data} and Appendix~\ref{sec:bernoulli_small_p}), our methods with adaptive weighting achieve higher power than BDM. These findings suggest that no single method uniformly dominates, and the relative power advantage depends on the outcome distribution and its sparsity structure.
}

\subsection{Investigating \citet{Hadad2021}}\label{sec:inspection_Victor}

\citet{Hadad2021} consider a class of WIPW estimators. Their paper observes that when estimating the expected outcome $\E[Y_{uN}(s)]$, even the classical inverse probability weighted estimator can exhibit non-normal behavior (see Figure 1 in \citeauthor{Hadad2021}). To address this, \citet{Hadad2021} shows asymptotic normality can be recovered through the use of adaptive weighting—similar in spirit to the method proposed in our work (see Theorem 4 in their paper). However, their theoretical guarantees rely on a key assumption: that the ratio of variance estimators converges to a constant. This assumption, however, fails to hold under the null hypothesis $H_{0N}$.

Indeed, as we will demonstrate shortly, the ratio of variance estimators converges to a non-degenerate, positive random variable whenever $c \in (-\infty, \infty)$—that is, under both the null hypothesis $H_{0N}$ and the weak signal regime $H_{2N}$. Moreover, our simulations reveal that applying a normal approximation in the absence of this convergence can lead to inflated Type-I error rates when using the WIPW estimator. Consequently, the theoretical results in \citet{Hadad2021} are not directly applicable for hypothesis testing, due to the unknown limiting distribution under the null.
 
To be specific, Theorem 4 in \citet{Hadad2021} hinges on the assumption that $\hat V_N(0)/\hat V_N(1)$ converges weakly to a constant, where $\hat V_N(s)$ is defined as in Equation~\eqref{eq:variance-estimator}. However, under the null $H_{0N}$ and local alternative $H_{2N}$, this condition generally fails. As shown in \textbf{step 2} and \textbf{step 3} in section \ref{sec:proof_roadmap}: 
with adaptive weighting ($m=1/2$) and $q_t=1/2$,
\begin{align}\label{eq:variance-ratio-convergence}
	\hat V_N(0)/\hat V_N(1)\convd \frac{\sum_{t=1}^2 V^{(t)}(0)/ (\sum_{t=1}^2 (H^{(t)}(0))^{1/2})^2}{\sum_{t=1}^2 V^{(t)}(1)/ (\sum_{t=1}^2 (H^{(t)}(1))^{1/2})^2}.
\end{align}
where $V^{(t)}(s)$ and $H^{(t)}(s)$ are defined as in~\eqref{eq:limiting_variance} and \eqref{eq:limiting_probabilities}.

To further illustrate this issue, we present empirical evidence showing that using the normal distribution to calibrate the test statistic leads to inflated Type-I error. Consider a two-batch setup under the following model:
\begin{align*}
	Y_u(0)\sim N(2, 1),\ Y_u(1)\sim N(2, 9).
\end{align*}
with a total sample size of $N = 500$ and equal batch sizes $N_1 = N_2 = 250$. We apply the $\varepsilon$-greedy algorithm with $\varepsilon = 0.05$, and compute the weighted augmented IPW estimator (WAIPW), in line with the original setup in \citet{Hadad2021}:
\begin{align*}
	\WAIPW(s)\equiv \sum_{t=1}^2 \frac{h_{N}^{(t)}(s)}{\sum_{t=1}^2 h_N^{(t)}(s)}\Gamma_{N}^{(t)}(s),
\end{align*}
where 
\begin{align*}
	\Gamma_{N}^{(t)}(s)\equiv \frac{\sum_{u=1}^{N_t}\Gamma_{u}^{(t)}(s)}{N_t}+\hat \E[Y_u(s)]\quad\text{and}\quad\Gamma_u^{(t)}(s)\equiv \frac{\indicator(A_u^{(t)}=s)(Y_{u}^{(t)}-\hat \E[Y_u(s)])}{\P[A_{u}^{(t)}=s|\mathcal{H}_{t-1}]}.
\end{align*}
According to Theorem 4 in \cite{Hadad2021}, we know 
\begin{align*}
	\frac{\mathrm{WAIPW}(0)-\mathrm{WAIPW}(1)}{(\tilde V_N(0)+\tilde V_N(1))^{1/2}}\convd N(0,1),\ \tilde V_N(s)\equiv \frac{\sum_{t=1}^2(h_{N}^{(t)}(s))^2\sum_{u=1}^{N_t}  \left(\Gamma_{u}^{(t)}(s)\right)^2}{(\sum_{t=1}^2h_{N}^{(t)}(s)N_t)^2}.
\end{align*} 
Setting $\hat \E[Y_u(0)]=2$ and $\hat \E[Y_u(1)]=6$, we can satisfy the required condition that for \textbf{at least one} $s\in\{0,1\},\hat \E[Y_{u}(s)]\rightarrow \E[Y_{u}(s)]$. Then we use $N(0,1)$ to calibrate the test statistic and obtain the Type-I error control results as shown in Figure \ref{fig:failure_Hadad}. The results show substantial Type-I error inflation when the normal approximation is used.

\begin{figure}[!htb]
	\centering
	\includegraphics[width=0.95\textwidth]{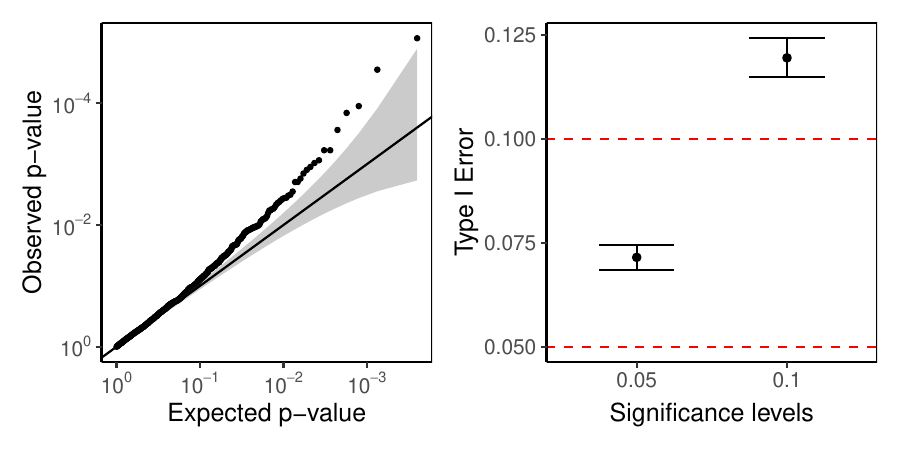}
	\caption{Type-I error inflation using normal approximation in \citet{Hadad2021}. The simulation is repeated for $2000$ times.}
	\label{fig:failure_Hadad}
\end{figure}

\subsection{Investigating \citet{Hirano2023,adusumilli2023optimal}}

To address non-normality directly, \citet{Hirano2023} develop asymptotic representations for a broad class of test statistics under batched designs, which are subsequently used to derive power functions and optimal tests in \citet{adusumilli2023optimal}. The validity of this theory requires establishing weak convergence of a vector of statistics, as well as assuming that the potential outcome distributions are differentiable in quadratic mean (QMD). In contrast, our Theorem~\ref{thm:weak_convergence_W_N} focuses on the widely used WIPW class of statistics and provides explicit weak convergence results that go beyond the QMD framework, based on transparent and interpretable assumptions. 

\subsection{Investigating \citet{che2023adaptive}}\label{sec:inspection_Che}

While \citet{che2023adaptive} broaden the framework of \citet{Hirano2023} to encompass settings beyond QMD, 
their emphasis lies in experimental design within batched bandit designs rather than in inferential validation. Specifically, they analyze the statistic  
\[
\frac{1}{\sqrt{N_t}}\sum_{u=1}^{N_t}A_{uN}^{(t)}\,Y_{uN}^{(t)},
\]  
deriving its asymptotic distribution under conditions similar to ours. Their results, 
however, are not sufficient for the purpose of inference.
For example, one can show that—after an additional scaling by $\sqrt{N_t}$—this statistic generally fails to converge to the target estimand $\E[Y_{uN}(s)]$, except in the degenerate case $\E[Y_{uN}(s)]=0$.
Moreover, the asymptotic distribution therein contains unknown parameters, 
making it infeasible to construct valid hypothesis tests or confidence intervals based on this statistic.

By contrast, our work prioritizes the inferential integrity of adaptive experiments. Our results incorporate consistent and pivotal estimators, which
guarantees that hypothesis tests are valid and that confidence intervals faithfully reflect uncertainty around the desired parameter.

\section{Probabilistic preliminaries}\label{sec:auxiliary_Notation}

\subsection{Bounded Lipschitz test function class}\label{sec:BL_distance}

We will use the following definition of convergence in distribution throughout the appendix.
\begin{definition}[Convergence in distribution]\label{def:convergence_distribution}
	Suppose $W_N\in\mathbb{R}^d$ is a sequence of random variables and $W\in\mathbb{R}^d$ is a random variable. We say $W_N$ converge in distribution to $W$ if for any bounded and continuous function $f:\mathbb{R}^d\rightarrow \mathbb{R}$, we have
	\begin{align*}
		\E[f(W_N)]\rightarrow \E[f(W)].
	\end{align*}
	Moreover, we use the notation $W_N\convd W$ to denote the convergence in distribution. If $W_N$ and $W$ are univariate, we will interchangeably use the following equivalent deifnition
	\begin{align*}
		\P[W_N\leq t]\rightarrow \P[W\leq t], \ \text{for each } t \in \R \text{ at which } t \mapsto \P[W \leq t] \text{ is continuous.}
	\end{align*}
\end{definition}

Beyond the bounded continuous functions, there are other classes of functions that are useful in the context of weak convergence. In particular, we will use the bounded Lipschitz function class to prove our main results. Suppose $f$ is a function from $\mathbb{R}^d$ to $\mathbb{R}$. We call $f$ is a $m$-Lipschitz function as long as $\|f\|_{\mathrm{L}}\leq m$, where \(\|f\|_{\mathrm{L}} \equiv \sup_{x \neq y \in \mathbb{R}^d} |f(x) - f(y)| / \|x - y\|_2\) and $\|\cdot\|_2$ is the Euclidean norm. Then we define 
\begin{align}\label{def:BL-function}
	\|f\|_{\mathrm{BL}}\equiv \|f\|_{\mathrm{L}}+\|f\|_{\infty}\quad\text{where}\quad \|f\|_{\infty}\equiv \sup_{x\in\mathbb{R}^d}|f(x)|.
\end{align}
\noindent The following lemma shows that the BL function class is rich enough to characterize the weak convergence of random variables.
\begin{lemma}[Theorem 11.3.3 in \citet{Dudley_2002}]\label{lem:Dudley_equivalence}
	Suppose the sequence of random variable $W_N\in\mathbb{R}$ and $W\in\mathbb{R}$ have law $\P_N$ and $\P$, respectively. Then the following two statements are equivalent:
	\begin{enumerate}
		\item $W_N\convd W$;
		\item $\left|\int f(x)\mathrm{d}\P_N(x)-\int f(x)\mathrm{d}\P(x)\right|\rightarrow0$ for any $f$ such that $\|f\|_{\mathrm{BL}}<\infty$.
	\end{enumerate}
\end{lemma}

Despite the fruitful results in the literature on normal approximation on independent observations \citep{chatterjee2008multivariate} and weakly dependent observation \citet{chen2004normal}, these existing results do not apply directly to our case since the adaptive sampling scheme introduces a \textit{strong} dependence structure. We will develop new tools for proving our results, based on these existing results. Thus we review the relevant results in the following section. First comes the finite-sample bound proved in \citet{raivc2018multivariate}.

\begin{lemma}[\citet{raivc2018multivariate}, Theorem 2.15 and Equation 3.5]\label{lem:CLT_BL_distance}
	Let $W_{1N},\ldots, W_{NN}$ be a sequence of independent, identically distributed random vectors in $\mathbb{R}^k$. Suppose
	\begin{align*}
		\E[W_{uN}]=0,\ \E[W_{uN}W_{uN}^\top ]=\bm I_{k}.
	\end{align*}
	Let $W_N\equiv \sum_{u=1}^N W_{uN}/\sqrt{N}$ and $Z\sim N(\bm 0,\bm I_k)$. Then for any $g$ such that $\|g\|_{\mathrm{BL}}<\infty$,
	\begin{align*}
		\left|\E[g(W_N)]-\E[g(Z)]\right|\lesssim_{k} \|g\|_{\mathrm{BL}}\frac{\E[\|W_{uN}\|_2^3]+1}{\sqrt{N}}.
	\end{align*}
\end{lemma}

The following Corollary is a direct application of Lemma~\ref{lem:CLT_BL_distance}.
\begin{corollary}[Upper bound with bounded Lipschitz function]\label{cor:W_1_bound}
	Suppose $W_{uN}\in\mathbb{R}^k$ are i.i.d. random variables for any fixed $N\in\mathbb{N}_+$. Then if we have $\E[W_{uN}]=\bm 0,\E[W_{uN}W_{uN}^\top]=\bm I_k$, then for any sequence of function $r_N$ such that $\|r_N\|_{\mathrm{BL}}\leq 1$, we have
	\begin{align*}
		\left|\E\left[r_N\left(\frac{1}{\sqrt{N}}\sum_{u=1}^N W_{uN}\right)\right]-\E\left[r_N\left(Z\right)\right]\right|\lesssim_{k} \frac{\E[\|W_{uN}\|_2^3]+1}{\sqrt{N}}.
	\end{align*}
	where $Z\sim N(\bm 0,\bm I_k)$.
\end{corollary}

\subsection{Preliminaries on regular conditional distribution}\label{sec:RCD_preliminary}

To better understand the argument involving conditional distribution, we briefly discuss the basic definition of regular conditional distribution. Let $\mathcal{B}(\mathbb{R}^N)$ be the Borel $\sigma$-algebra on $\mathbb{R}^N$ and $\Omega,\mathcal{F}_N$ be the sample space and a sequence of $\sigma$-algebras. For any $N\in\mathbb{N}_+,\kappa_N:\Omega\times \mathcal{B}(\mathbb{R}^N)$ is a regular conditional distribution of $W_N\equiv (W_{1N},\ldots,W_{NN})$ given $\mathcal{F}_N$ if 
\begin{align*}
	\omega\mapsto \kappa_N(\cdot,B) \text{ is measurable with respect to $\mathcal{F}_N$ for any fixed $B\in\mathcal{B}(\mathbb{R}^N)$};
	&
	\\
	B\mapsto \kappa_N(\omega,B) \text{ is a probability measure on }(\mathbb{R}^N,\mathcal{B}(\mathbb{R}^N)) \text{ for any }\omega\in\Omega;
	&
	\\
	\kappa_N(\omega,B)=\P[(W_{1N},\ldots,W_{NN})\in B|\mathcal{F}_N](\omega)\text{ for almost all }\omega\in\Omega\text{ and all }B\in\mathcal{B}(\mathbb{R}^N).&
\end{align*}
The following lemma from \citet[Theorem 8.37]{Lista2017} ensures the general existence of regular conditional distribution.
\begin{lemma}[Theorem 8.37 in \citet{Lista2017}]\label{lem:Klenke_Thm_8.37}
  Suppose $(\Omega,\mathcal{G},\P)$ is the Probability triple. Let $\mathcal{F}\subset \mathcal{G}$ be a sub-$\sigma$-algebra. Let $Y$ be a random variable with values in a Borel space $(E,\mathcal{E})$ (for example, $E$ is Polish, $E=\mathbb{R}^k$). Then there exists a regular conditional distribution $\kappa_{Y,\mathcal{F}}$ of $Y$ given $\mathcal{F}$.
\end{lemma}
\noindent Result from \citet[Theorem 8.38]{Lista2017} guarantees that the conditional expectation and the integral of measurable function with respect to regular conditional distribution are almost surely same.

\begin{lemma}[Modified version of Theorem 8.38 in \citet{Lista2017}]\label{lem:Klenke_Thm_8.38}
  Let $Y$ be a random variable $(\Omega,\mathcal{G},\mathbb{P})$ taking values in a Borel space $(E,\mathcal{E})$. Let $\mathcal{F}\subset \mathcal{G}$ be a $\sigma$-algebra and let $\kappa_{Y,\mathcal{F}}$ be a regular conditional distribution of $Y$ given $\mathcal{F}$. Further, let $f:E\rightarrow\mathbb{R}$ be measurable and $\E[|f(Y)|]<\infty$. Then we can define a version of the conditional expectation using regular conditional distribution, i.e., 
  \begin{align*}
    \E[f(Y)|\mathcal{F}](\omega)=\int f(y)\mathrm{d}\kappa_{Y,\mathcal{F}}(\omega,y),\ \forall \omega\in\Omega.
  \end{align*}
\end{lemma}
\noindent Throughout this paper we will fix a version of the conditional expectation and the version is defined by applying Lemma \ref{lem:Klenke_Thm_8.38}. The following lemma is useful as well. 

\begin{lemma}[Conditional expectation with variable measurable with respect to $\mathcal{F}$]\label{lem:new_klenke_thm_8.38}
	Let $Y$ be a random variable $(\Omega,\mathcal{G},\mathbb{P})$ with values in a Borel space $(E,\mathcal{E})$. Let $\mathcal{F}\subset \mathcal{G}$ be a $\sigma$-algebra and let $\kappa_{Y,\mathcal{F}}$ be a regular conditional distribution of $Y$ given $\mathcal{F}$. Suppose $X\in\mathcal{F}$ is a random variable in another Borel space $(B,\mathcal{B})$. Further, let $f:E\times B\rightarrow\mathbb{R}$ be measurable and $\E[|f(X, Y)|]<\infty$. Then we can define a version of the following conditional expectation using regular conditional distribution:
  \begin{align*}
    \E[f(X, Y)|\mathcal{F}](\omega)=\int f(X(\omega),y)\mathrm{d}\kappa_{Y,\mathcal{F}}(\omega,y),\ \forall \omega\in\Omega.
  \end{align*}
\end{lemma}

\subsection{Definition of conditional convergence}

Since we are under the setup where the data is generated adaptively, we need to extensively work with the conditional convergence where the conditioning is on the data collected in the first stage. We need to first define the notions of conditional convergence, with the definition of regular conditional distribution. In particular, we adopt the definition of conditional convergence in distribution and probability from \citet{niu2024reconciling}.

\begin{definition} \label{def:conditional-convergence-distribution}
    For each $N$, let $W_N$ be a random variable and let $\mathcal F_N$ be a $\sigma$-algebra. Then, we say $W_N$ converges in distribution to a random variable $W$ conditionally on $\mathcal F_N$ if
    \begin{equation}
        \P[W_N \leq t \mid \mathcal F_N] \convp \P[W \leq t] \ \text{for each } t \in \R \text{ at which } t \mapsto \P[W \leq t] \text{ is continuous.}
    \end{equation}
    We denote this relation via $W_N \mid \mathcal F_N \convdp W$.
\end{definition}

\begin{definition} \label{def:conditional-convergence-probability}
	For each $N$, let $W_N$ be a random variable and let $\mathcal F_N$ be a $\sigma$-algebra. Then, we say $W_N$ converges in probability to a constant $c$ conditionally on $\mathcal F_N$ if $W_N$ converges in distribution to the delta mass at $c$ conditionally on $\mathcal F_N$ (recall Definition~\ref{def:conditional-convergence-distribution}). We denote this convergence by ${W_N \mid \mathcal F_N \convpp c}$. In symbols, 
	\begin{equation}
		W_N \mid \mathcal F_N \convpp c \quad \text{if} \quad W_N \mid \mathcal F_N \convdp \delta_c.
	\end{equation}
\end{definition}

\subsection{Proof of Lemma \ref{lem:new_klenke_thm_8.38}}

\begin{proof}[Proof of Lemma \ref{lem:new_klenke_thm_8.38}]
	By Lemma \ref{lem:Klenke_Thm_8.38}, we know that 
	\begin{align*}
		\E[f(X,Y)|\mathcal{F}](\omega)=\int f(x,y)\mathrm{d}\kappa_{(X,Y),\mathcal{F}}(\omega,(x,y)).
	\end{align*}
	Now we prove that for almost every $\omega\in\Omega$,
	\begin{align*}
		\kappa_{(X,Y),\mathcal{F}}(\omega,S)=\kappa_{Y,\mathcal{F}}(\omega,S_1)\cdot \indicator(X(\omega)\in S_2),\ \forall S=S_1\times S_2\subset E\times B.
	\end{align*}
	In other words, $\kappa_{(X,Y),\mathcal{F}}(\omega,\cdot)$ is product measure of another measure $\kappa_{Y,\mathcal{F}}(\omega,\cdot)$ and counting measure supported on the value $X(\omega)$. This can be proved by using the definition of regular conditional distribution. For any $S=S_1\times S_2\subset E\times B$, we have
	\begin{align*}
		\kappa_{(X,Y),\mathcal{F}}(\omega,S)=\P[(X,Y)\in S|\mathcal{F}](\omega)
		&
		=\indicator(X(\omega)\in S_2)\P[Y\in S_1|\mathcal{F}](\omega)\\
		&
		=\kappa_{Y,\mathcal{F}}(\omega,S_1)\cdot \indicator(X(\omega)\in S_2)
	\end{align*}
	for almost every $\omega\in\Omega$. Thus by Fubini's theorem, we conclude
	\begin{align*}
		\int f(x,y)\mathrm{d}\kappa_{(X,Y),\mathcal{F}}(\omega,(x,y))=\int f(X(\omega),y)\mathrm{d}\kappa_{Y,\mathcal{F}}(\omega,y).
	\end{align*}
\end{proof}

\section{Useful lemmas and the proofs}\label{sec:aux_lemma}

\subsection{Lemma statements}

\begin{lemma}[Conditional Polya's theorem, Theorem 5 in \citet{niu2024reconciling}]\label{lem:cond_polya} 
	Let $W_N$ be a sequence of random variables. If $W_N \mid \mathcal F_N \convdp W$ for some random variable $W$ with continuous CDF, then
	\begin{equation}
		\sup_{t \in \R}|\P[W_N \leq t \mid \mathcal F_N] - \P[W \leq t]| \convp 0.
	\end{equation}
\end{lemma}

\begin{lemma}[Slutsky's theorem, Theorem 13.18 in \citet{Dudley_2002}]\label{lem:Slutsky}
	Let $X_1,X_2\ldots$ and $Y_1,Y_2,\ldots,$ be random variables with values in $\mathbb{R}^{k}$. Suppose $X_N\convd X$ and $\|X_N-Y_N\|_2\convp 0$. Then $Y_N\convd X$.
\end{lemma}

\begin{lemma}[\citet{durrett2019probability}, Theorem 2.3.2]\label{lem:sub_subseq}
	A sequence of random variables $W_N$ converges to a limit $W$ in probability if and only if every subsequence of $W_N$ has a further subsequence that converges to $W$ almost surely.
\end{lemma}

\begin{lemma}[Skorohod's representation theorem]\label{lem:skorohod}
	Let \( (\mu_N)_{N \in \mathbb{N}} \) be a sequence of probability measures on a metric space \( S \) such that \( \mu_N \) converges weakly to some probability measure \( \mu_{\infty} \) on \( S \) as \( N\to \infty \). Suppose also that the support of \( \mu_{\infty} \) is separable. Then there exist \( S \)-valued random variables \( W_N \) defined on a common probability space \( (\Omega, \mathcal{F}, \P) \) such that the law of \( W_N \) is \( \mu_N \) for all \( N \) (including \( N = \infty \)) and such that \( (W_N)_{N \in \mathbb{N}} \) converges to \( W_{\infty} \), \( \P \)-almost surely.
\end{lemma}

\begin{lemma}[Lipschitz continuity of square root matrix under F-norm]\label{lem:holder_continuity_Frobenius}
	Suppose $A,B\in\mathbb{R}^{k\times k}$ are two positive definite matrices with eigevalues bounded away from $m$, then we have 
	\begin{align*}
		\|A^{1/2}-B^{1/2}\|_{\mathrm{F}}\lesssim_{k,m} \|A-B\|_{\mathrm{F}}.
	\end{align*}
\end{lemma}

\begin{lemma}[Boundedness on the covariance]\label{lem:upper_bound_cov}
	Suppose $X_N,Y_N$ are two sequences of random variables, with finite first moment, satisfying 
	\begin{enumerate}
		\item $\E[X_N^{p}]$ and $\E[Y_N^{p}]$ converge to finite constants for $p=1,2$;
		\item $\liminf_{N\rightarrow\infty}\mathrm{Var}[X_N]>0,\ \liminf_{N\rightarrow\infty}\mathrm{Var}[Y_N]>0$.
	\end{enumerate}
	Suppose a random sequence $a_N\in (0,1)$ almost surely. Then we have 
	\begin{align}\label{eq:upper_bound_cov}
		\limsup_{N\rightarrow\infty}\left|\frac{a_N^{1/2}\E[X_N]}{(\E[X_N^2]-a_N\E[X_N]^2)^{1/2}}\times \frac{(1-a_N)^{1/2}\E[Y_N]}{(\E[Y_N^2]-(1-a_N)\E[Y_N]^2)^{1/2}}\right|<C<1
	\end{align}
	almost surely and the constant $C$ only depends on the limit of the moments $\lim_{N\rightarrow\infty}\E[X_N^p]$ and $\lim_{N\rightarrow\infty}\E[Y_N^p]$ for $p\in\{1,2\}$.
\end{lemma}

\subsection{Proof of Lemma \ref{lem:holder_continuity_Frobenius}}

\begin{proof}[Proof of Lemma \ref{lem:holder_continuity_Frobenius}]
	It suffices to prove 
	\begin{align}\label{eq:holder_continuity_2_Norm}
		\|A^{1/2}-B^{1/2}\|_{2}\leq \|A-B\|_{2},
	\end{align}
	where $\|A\|_2$ is the operator norm $A$. This is because if the statement~\eqref{eq:holder_continuity_2_Norm} holds, then
	\begin{align*}
		\|A^{1/2}-B^{1/2}\|_{\mathrm{F}}\leq \sqrt{k}\|A^{1/2}-B^{1/2}\|_{2}\leq \sqrt{k}\|A-B\|_{2}\leq \sqrt{k}\|A-B\|_{\mathrm{F}}.
	\end{align*}
	Now we prove \eqref{eq:holder_continuity_2_Norm}. If vector $x$ with $\|x\|=1$ is an eigenvector of $\sqrt{A}-\sqrt{B}$ with eigenvalue $\mu$ then 
	\begin{align*}
		x^\top (A-B)x=x^\top (\sqrt{A}-\sqrt{B})\sqrt{A}x+x^\top \sqrt{B}(\sqrt{A}-\sqrt{B})x=\mu x^\top (\sqrt{A}+\sqrt{B})x.
	\end{align*}
	Now if we choose $\mu=\pm \|\sqrt{A}-\sqrt{B}\|_2$ (depending on the sign of the eigenvalue which has the largest magnitude), we have
	\begin{align*}
		\|\sqrt{A}-\sqrt{B}\|_2\leq\frac{1}{x^\top (\sqrt{A}+\sqrt{B})x} |x^\top (A-B)x|\lesssim_{m} \|A-B\|_2.
	\end{align*}
	This completes the proof.
\end{proof}

\subsection{Proof of Lemma \ref{lem:upper_bound_cov}}

\begin{proof}[Proof of Lemma \ref{lem:upper_bound_cov}]

	We now divide the proof into two cases.
	\begin{enumerate}
		\item \textbf{When $\lim_{N\rightarrow\infty}\E[X_N]=0$ or $\lim_{N\rightarrow\infty}\E[Y_N]=0$.} Since 
	\begin{align*}
		\liminf_{N\rightarrow\infty}(\E[X_N^2]-a_N\E[X_N]^2)\geq \liminf_{N\rightarrow\infty}(\E[X_N^2]-\E[X_N]^2)>0
	\end{align*}
	and 
	\begin{align*}
		\liminf_{N\rightarrow\infty}(\E[Y_N^2]-(1-a_N)\E[Y_N]^2)\geq \liminf_{N\rightarrow\infty}(\E[Y_N^2]-\E[Y_N]^2)>0,
	\end{align*}
	we know the claim is true with $C=0$ almost surely.
	\item\textbf{When both $\lim_{N\rightarrow\infty}\E[X_N]\neq 0$ and $\lim_{N\rightarrow\infty}\E[Y_N]\neq 0$:} Define the sequence $E_N\equiv a_N\E[Y_N^2](\E[X_N^2]-\E[X_N^2]) +(1-a_N)(\E[Y_N^2]-\E[Y_N]^2)\E[X_N^2]$. We know 
	\begin{align*}
		D_N 
		&
		\equiv (\E[X_N^2]-a_N\E[X_N]^2)(\E[Y_N^2]-(1-a_N)\E[Y_N]^2)\\
		&
		= a_N(1-a_N)(\E[X_N]\E[Y_N])^2+ E_N\\
		&
		\equiv C_N +E_N,
	\end{align*}
	Note that conclusion~\eqref{eq:upper_bound_cov} is equivalent to proving $\limsup_{N\rightarrow\infty}|C_N^{1/2}/D_N^{1/2}|<1$ almost surely. To this end, we observe that
	\begin{align*}
		\frac{D_N}{C_N}
		&
		=1+\frac{E_N}{C_N}\\
		&
		=1+\frac{a_N\E[Y_N^2](\E[X_N^2]-\E[X_N]^2)}{C_N}+\frac{(1-a_N)(\E[Y_N^2]-\E[Y_N]^2)\E[X_N^2]}{C_N}\\
		&
		=1+\frac{\E[Y_N^2](\E[X_N^2]-\E[X_N]^2)}{(1-a_N)(\E[X_N]\E[Y_N])^2}+\frac{(\E[Y_N^2]-\E[Y_N]^2)\E[X_N^2]}{a_N (\E[X_N]\E[Y_N])^2}\\
		&
		\equiv 1+R_{1N}+R_{2N}.
	\end{align*}
	We can bound 
	\begin{align*}
		\liminf_{N\rightarrow\infty}R_{1N}
		&
		\geq \liminf_{N\rightarrow\infty}\frac{\E[Y_N^2](\E[X_N^2]-\E[X_N]^2)}{(\E[X_N]\E[Y_N])^2}\\
		&
		=\lim_{N\rightarrow\infty}\frac{\E[Y_N^2]}{(\E[Y_N])^2}\lim_{N\rightarrow\infty}\frac{(\E[X_N^2]-\E[X_N]^2)}{(\E[X_N])^2}>0
	\end{align*}
	almost surely, since $\lim_{N\rightarrow\infty}(\E[X_N^2]-\E[X_N]^2)=\lim_{N\rightarrow\infty}\mathrm{Var}[X_N]>0$. Similarly, we can prove $\liminf_{N\rightarrow\infty}R_{2N} >0$ almost surely. Thus we have 
	\begin{align*}
		\liminf_{N\rightarrow\infty}\frac{D_N}{C_N}\geq 1+\liminf_{N\rightarrow\infty}R_{1N}+\liminf_{N\rightarrow\infty}R_{2N}>1
	\end{align*}
	almost surely. Thus the desired bound can be obtained by checking
	\begin{align*}
		&
		\limsup_{N\rightarrow\infty}\left|\frac{a_N^{1/2}\E[X_N]}{(\E[X_N^2]-a_N\E[X_N]^2)^{1/2}}\times \frac{(1-a_N)^{1/2}\E[Y_N]}{(\E[Y_N^2]-(1-a_N)\E[Y_N]^2)^{1/2}}\right|\\
		&
		=
		\limsup_{N\rightarrow\infty}\left|\frac{C_N^{1/2}}{D_N^{1/2}}\right|=\frac{1}{\liminf_{N\rightarrow\infty}D_N^{1/2}/C_N^{1/2}}<1
	\end{align*}
	almost surely.
	\end{enumerate}
	
\end{proof}

\section{New results on conditional CLT and CMT}\label{sec:new_conditional_CLT}

\subsection{A new conditional CLT}

\begin{lemma}[Conditional CLT under Lipschitz class]\label{lem:CLT_BL}
	Consider $\sigma$-alegbras $\mathcal{F}_N$. Suppose $W_{1N},\ldots,W_{NN}\in\mathbb{R}^k$ are i.i.d. random variables conditional on $\mathcal{F}_N$ for any fixed $N\in\mathbb{N}$. Define $W_N\equiv \sum_{u=1}^N W_{uN}/\sqrt{N}$.  Suppose
	\begin{align*}
		\E[W_{uN}|\mathcal{F}_N]=\bm 0,\ \E[W_{uN}W_{uN}^\top|\mathcal{F}_N]=\bm I_{k}.
	\end{align*}
	Furthermore, consider random variables $X_N\in\mathbb{R}^k$ and $X_N$ is measurable with respect to $\mathcal{F}_N$. Then for any measurable function $g_N:\mathbb{R}^k\mapsto\mathbb{R}^d$ and any Lipschitz function $f:\mathbb{R}^{d+k}\rightarrow\mathbb{R}$ such that $\|f\|_{\mathrm{BL}}\leq 1$, we have
	\begin{align*}
		\E[f(g_N(X_N),W_N)|\mathcal{F}_N]-\E[f(g_N(X_N),Z)|\mathcal{F}_N]\lesssim_{k} \frac{\E[\|W_{uN}\|_2^3\mid\mathcal{F}_N]+1}{\sqrt{N}}\quad\text{almost surely},
	\end{align*}
	where $Z\sim N(0,\bm I_{k})$.
\end{lemma}

\noindent Lemma \ref{lem:CLT_BL} is an application of Corollary~\ref{cor:W_1_bound} with $r_N(x)=f(g_N(X_N(\omega)),x)$ and regular conditional distribution $\P[(W_{1N},\ldots, W_{NN})\in B\mid \mathcal{F}_N](\omega)$ for any $\omega\in \Omega$.

\subsection{New conditional CMT results}\label{sec:conditional_CMT}

In this section, we extend the classical continuous mapping theorem to conditional convergence. The following lemma is the classical result.

\begin{lemma}[Classical CMT]\label{lem:continuous_mapping_lem}
	Let $W_N,W$ be random variables on a metric space $S$. Suppose a function $g:S\mapsto S'$, where $S'$ is another metric space, has the set of discontinuity points $D_g$ such that $\P[W\in D_g]=0$. Then if $W_N\convd W$, the following is true: $g(W_N)\convd g(W)$.
\end{lemma}

\noindent To account for the varying sequence and conditioning, we present the following lemmas. These wo results provide new conditional convergence in probability and distribution, respectively. 

\begin{lemma}[CMT: convergence in conditional probability]\label{lem:continuous_map_varying}
	Suppose $X_N,Y_N\in\mathcal{X}\subset \mathbb{R}^k$ and $\|X_N-Y_N\|_2=o_p(1)$. Furthermore, suppose function $f:\mathcal{X}\rightarrow\mathbb{R}$ is uniformly continuous in the support $\mathcal{X}$ and uniformly bounded: $\sup_{x\in\mathcal{X}}|f(x)|<\infty$. Then we have $\E[\left|f(X_N)-f(Y_N)\right|]\rightarrow 0$. Moreover, we have $\E[\left|f(X_N)-f(Y_N)\right||\mathcal{F}_N]\convp 0$.
\end{lemma}

\begin{lemma}[CMT: convergence in conditional distribution]\label{lem:sufficient_condition_CMT}
	Consider the $\sigma$-algebra $\mathcal{F}_N$ and a measurable function $g$ with discontinuity set $D_g$. Suppose random variable $W\sim\mathbb{P}_W$ satisfies $\P[W\in D_g]=0$ and $g(W)$ is a continuous random variable. If the sequence of random variable $W_N$ satisfies
	\begin{align*}
		\left|\E[f(W_N)|\mathcal{F}_N]-\E[f(W)]\right|\convp 0, \text{ for any } f\text{ such that }\|f\|_{\mathrm{BL}}<\infty,
	\end{align*}
	then we have $\sup_{t\in\mathbb{R}}\left|\P[g(W_N)\leq t|\mathcal{F}_N]-\P[g(W)\leq t]\right|\convp 0$.
\end{lemma}

\subsection{Proof of Lemma \ref{lem:continuous_map_varying}}

\begin{proof}[Proof of Lemma \ref{lem:continuous_map_varying}]
	Fix any $\delta>0$, we can choose $\varepsilon(\delta)>0$ such that whenever $\|X_N-Y_N\|_2<\varepsilon(\delta)$ we can guarantee $|f(X_N)-f(Y_N)|\leq \delta$ since $f$ is uniformly continuous in $\mathcal{X}$. Then for any $\delta>0$, we have 
	\begin{align*}
		\E[\left|f(X_N)-f(Y_N)\right|]
		&
		= \E\left[|f(X_N)-f(Y_N)|\mathbf{1}(\|X_N-Y_N\|_2\leq \varepsilon(\delta))\right]\\
		&
		\qquad +\E\left[|f(X_N)-f(Y_N)|\mathbf{1}(\|X_N-Y_N\|_2>\varepsilon(\delta))\right]\\
		&
		\leq \delta+2\sup_{x\in\mathcal{X}}|f(x)|\P[\|X_N-Y_N\|_2>\varepsilon(\delta)].
	\end{align*}
	Letting $N\rightarrow\infty$, we know $\P[\|X_N-Y_N\|_2>\varepsilon(\delta)]\rightarrow0$ so that we have 
	\begin{align*}
		\lim_{N\rightarrow\infty}\E[\left|f(X_N)-f(Y_N)\right|]\leq \delta.
	\end{align*}
	Since $\delta$ is chosen arbitrarily, we have $\lim_{N\rightarrow\infty}\E[\left|f(X_N)-f(Y_N)\right|]=0$ so that we complete the first part proof. By the uniform integrability of $\E[|f(X_N)-f(Y_N)|\mathcal{F}_N]$, guaranteed by the uniform boundedness of $f$, we conclude the second part of the proof.
\end{proof}

\subsection{Proof of Lemma \ref{lem:sufficient_condition_CMT}}

\begin{proof}[Proof of Lemma \ref{lem:sufficient_condition_CMT}]
	Using Lemma \ref{lem:Klenke_Thm_8.38} and definition of regular conditional distribution, we know 
	\begin{align*}
		\left|\int f(x)\mathrm{d}\kappa_{W_N,\mathcal{F}_N}(\omega,x)-\int f(x)\mathrm{d}\mathbb{P}_W(x)\right|\rightarrow0,\text{ for any }f\text{ such that } \|f\|_{\mathrm{BL}}<\infty
	\end{align*}
	for any $\omega\in\mathcal{E}\subset\Omega$, where $\P[\mathcal{E}]\rightarrow1$. Applying Lemma \ref{lem:Dudley_equivalence} to $\kappa_{W_N,\mathcal{F}_N}(\omega,\cdot),\P_W(\cdot)$, we know on the event $\mathcal{E}$, random variable $\tilde{W}_N$ from measure $\kappa_{W_N,\mathcal{F}_N}(\omega,\cdot)$ converges to measure $\P_W(\cdot)$ in distribution. In other words, we have for any bounded and continuous function $h$,
	\begin{align*}
		\E[h(\tilde{W}_N(\omega))]=\int h(x)\mathrm{d}\kappa_{W_N,\mathcal{F}_N}(\omega,x)\rightarrow \int h(x)\mathrm{d}\mathbb{P}_W(x)=\E[h(W)],\ \forall \omega\in\mathcal{E}.
	\end{align*}
	Then by continuous mapping theorem, we have for any $t\in\mathbb{R}$,
	\begin{align*}
		\P[g(\tilde{W}_N(\omega))\leq t]\rightarrow \P[g( W)\leq t],\ \forall \omega\in\mathcal{E}\cap \{W\notin D_g\}.
	\end{align*}
	Again applying Lemma \ref{lem:Klenke_Thm_8.38}, we know $\forall \omega\in\mathcal{E}\cap \{W\notin D_g\}$,
	\begin{align*}
		\left|\P[g(W_N)\leq t|\mathcal{F}_N](\omega)-\P[g(W)\leq t]\right|\rightarrow 0.
	\end{align*}
	Since $\P[\mathcal{E}\cap \{W\notin D_g\}]=\P[\mathcal{E}]\rightarrow 1$ by the assumption that $\P[W\in D_g]=0$, we have 
	\begin{align*}
		\left|\P[g(W_N)\leq t|\mathcal{F}_N]-\P[g(W)\leq t]\right|\convp 0.
	\end{align*}
	Finally, applying Lemma \ref{lem:cond_polya}, we have
	\begin{align*}
		\sup_{t\in\mathbb{R}}\left|\P[g(W_N)\leq t|\mathcal{F}_N]-\P[g(W)\leq t]\right|\convp 0.
	\end{align*}
\end{proof}

\section{Preparation for the proof of Theorem~\ref{thm:weak_convergence_W_N}}

\subsection{Auxiliary definitions for the proof of Theorem~\ref{thm:weak_convergence_W_N}}\label{sec:notation_main_proof}

Recall the definition of $c_N$ and $c$ in Assumption~\ref{assu:moment_condition}. Then consider the function $h(\cdot,c):\mathbb{R}^{12}\rightarrow\bar{\mathbb{R}}$ such that for $x=(x_1,\ldots,x_{10},x_{11},x_{12})\in\mathbb{R}^{12}$,
\begin{align}\label{eq:def_h_c_function}
	h(x,c)= (\sqrt{x_3/x_5},-\sqrt{x_4/x_6})  A (x_1, x_2)^\top+\frac{1}{\sqrt{2}}c\quad\text{where}\quad A=
	\begin{pmatrix}
		x_7 & x_8 \\
		x_9 & x_{10}\\
	\end{pmatrix},
\end{align} 
and $h_N(x)\equiv h(x,c_N)$. Recall the sampling function $e(s,x)$, defined as in Assumption~\ref{assu:sampling_design}. Define for any $x\in\mathbb{R}^{12}$, 
\begin{align*}
	H_N(s,x)\equiv \min\{1-l_N, \max\{l_N, e(s,h_N(x))\}\}
\end{align*}
and 
\begin{align*}
	V_N(s,x)\equiv \E[Y_{uN}^2(s)]-H_N(s,x)\E[Y_{uN}(s)]^2.
\end{align*}
In particular, we define the weight and variance functions: 
\begin{align*}
	H_N^{(2)}(x)\equiv (H_N(0,x),H_N(1,x))\quad\text{and}\quad V_N^{(2)}(x)\equiv (V_N(0,x),V_N(1,x))
\end{align*}
and the covariance function 
\begin{align}\label{eq:cov_N_x_def}
	\mathrm{Cov}_N^{(2)}(x)\equiv \frac{-(H_N(0,x)H_N(1,x))^{1/2}}{(V_{N}(0,x))^{1/2}(V_{N}(1,x))^{1/2}}\E[Y_{uN}(0)]\E[Y_{uN}(1)].
\end{align}
Also, define the matrix function $\bm \Sigma_N^{(2)}(x)\equiv (\mathrm{Cov}_N^{(2)}(x))_{2\times 2}$.

\subsection{Generic lemmas on the random functions}

\begin{lemma}[Asymptotic lower and upper bound of $V_N^{(t)}(s)$ and $V_N(s,x)$]\label{lem:uniform_lower_upper_bound_variance}
	Suppose the Assumption \ref{assu:moment_condition}-\ref{assu:sampling_design} hold and either Assumption \ref{assu:constant_weighting} or Assumption \ref{assu:adaptive_weighting} holds. Then for any $t=1,2$ and $s=0,1$, we have 
	\begin{align*}
		0<\liminf_{N\rightarrow\infty}V_{N}^{(t)}(s)\leq \limsup_{N\rightarrow\infty}V_{N}^{(t)}(s)<\infty,
	\end{align*}
	and 
	\begin{align*}
		0<\liminf_{N\rightarrow\infty}\inf_{x\in\mathbb{R}^{10}}V_{N}(s,x)\leq \limsup_{N\rightarrow\infty}\sup_{x\in\mathbb{R}^{10}}V_{N}(s, x)<\infty.
	\end{align*}
	Similarly, $V^{(t)}(s)$ is uniformly lower and upper bounded for $s=0,1$ and $t=1,2$.
\end{lemma}

\begin{lemma}[Lipschitz property of weight, variance and covariance matrix function]\label{lem:continuity_sqrt_sampling_function}
	Suppose the Assumption \ref{assu:moment_condition}-\ref{assu:sampling_design} hold and either Assumption \ref{assu:constant_weighting} or Assumption \ref{assu:adaptive_weighting} holds. Then there exist universal constants $C_1,C_2,C_3,C_4$ such that for any $x_1,x_2\in\mathbb{R}^{10}$
	\begin{align*}
		|H_{N}(s,x_1)-H_{N}(s,x_2)|
		&
		\leq C_1|e(s,h_N(x_1))-e(s,h_N(x_2))|\\
		|H_{N}^{1/2}(s,x_1)-H_{N}^{1/2}(s,x_2)|
		&
		\leq C_2|e^{1/2}(s,h_N(x_1))-e^{1/2}(s,h_N(x_2))|\\
		|V_{N}(s,x_1)-V_{N}(s,x_2)|
		&
		\leq C_3|e(s,h_N(x_1))-e(s,h_N(x_2))|.
	\end{align*}
	Furthermore, $\mathrm{Cov}_N^{(2)}$ is uniformly upper bounded so that 
	\begin{align*}
		&
		\|(\bm\Sigma_{N}^{(2)}(x_1))^{1/2}-(\bm\Sigma_{N}^{(2)}(x_2))^{1/2}\|_{\mathrm{F}}\\
		&
		\qquad\leq C_4\sum_{s=0,1}\left(|e(s,h_N(x_1))-e(s,h_N(x_2))|+|H_N^{1/2}(s,x_1)-H_N^{1/2}(s,x_2)|\right).
	\end{align*}
\end{lemma}

\subsection{Proof of Lemma \ref{lem:uniform_lower_upper_bound_variance}}

\begin{proof}[Proof of Lemma \ref{lem:uniform_lower_upper_bound_variance}]
	We first prove the result for $V_N^{(t)}(s)$. We divide the proof into two parts.

	\paragraph{Proof of $\liminf_{N\rightarrow\infty}V_N^{(t)}(s)>0$:}
	For any $s,t$, guaranteed by Assumption \ref{assu:moment_condition},
	\begin{align*}
		\liminf_{N\rightarrow\infty}V_{N}^{(t)}(s)
		&
		=\liminf_{N\rightarrow\infty}\left(\E[Y_{uN}^2(s)]-\E[Y_{uN}(s)]^2+(1-e_N(s,\mathcal{H}_{t-1}))\E[Y_{uN}(s)]^2\right)\\
		&
		\geq \liminf_{N\rightarrow\infty}\left(\E[Y_{uN}^2(s)]-\E[Y_{uN}(s)]^2\right)>0.
	\end{align*}
	Thus we have $\liminf_{N\rightarrow\infty}V_{N}^{(t)}(s)>0$ for any $s,t$. 

	\paragraph{Proof of $\limsup_{N\rightarrow\infty}V_N^{(t)}(s)<\infty$:}

	We bound 
	\begin{align*}
		\limsup_{N\rightarrow\infty}V_{N}^{(t)}(s)\leq \limsup_{N\rightarrow\infty}\E[Y_{uN}^2(s)]\leq\limsup_{N\rightarrow\infty}(\E[Y_{uN}^4(s)])^{1/2} <\infty
	\end{align*}
	where the last inequality is due to Assumption \ref{assu:moment_condition}. The proof for $V_N(s,x)$ is similar so that we omit it.
\end{proof}

\subsection{Proof of Lemma \ref{lem:continuity_sqrt_sampling_function}}\label{sec:proof_lem:continuity_sqrt_sampling_function}

\begin{proof}[Proof of Lemma \ref{lem:continuity_sqrt_sampling_function}]

	We prove the claim subsequently.

	\paragraph{Proof for $H_N(s,x)$:}

	Write $H_N(s,x)=\min\{1-l_N, \max\{l_N,e(s,h_N(x))\}\}$. We first notice that $\min\{1-l_N, \max\{l_N,x\}\}$ is a Lipschitz function of $x$ with Lipschitz constant $1$ so that we have
	\begin{align*}
		|H_N(s,x_1)-H_N(s,x_2)|\leq|e(s,h_N(x_1))-e(s,h_N(x_2))|.
	\end{align*}

	\paragraph{Proof for $H_N^{1/2}(s,x)$:}

	Since 
	\begin{align*}
		\sqrt{\min\{1-l_N,\max\{l_N,x\}\}}=\min\{\sqrt{1-l_N},\max\{\sqrt{l_N},\sqrt{x}\}\}
	\end{align*}
	is a Lipschitz function of $\sqrt{x}$, we can conclude 
	\begin{align*}
		|H_N^{1/2}(s,x_1)-H_N^{1/2}(s,x_2)|\leq |e^{1/2}(s,h_N(x_1))-e^{1/2}(s,h_N(x_2))|.
	\end{align*}
	
	\paragraph{Proof for $V_{N}(s, x)$:}

	Recall the definition $V_{N}(s, x_1)=\E[Y_{uN}^2(s)]-H_N(s,x_1)\E[Y_{uN}(s)]^2$ so that we can bound 
	\begin{align*}
		|V_{N}(s, x_1)-V_{N}(s, x_2)|
		&
		\leq \E[Y_{uN}(s)]^2|H_N(s,x_1)-H_N(s,x_2)|\\
		&
		\leq \E[Y_{uN}(s)]^2|e(s,h_N(x_1))-e(s,h_N(x_2))|
	\end{align*}
	where the last inequality is true due to the result already proved for $H_N(s,x)$. Since $\E[Y_{uN}(s)]^2$ is uniformly bounded by Assumption \ref{assu:moment_condition}, we complete the proof for $V_N(s,x)$.

	\paragraph{Proof for $(\bm \Sigma_N^{(2)}(x))^{1/2}$:} It suffices to prove that the eigenvalues of $\bm \Sigma_{N}^{(2)}=\bm\Sigma_{N}^{(2)}(E_{N}^{(1)})$ are uniformly bounded from $0$ and $\infty$. Recall the expression of $\bm\Sigma_{N}^{(2)}$ as in \eqref{eq:sigma_N_E_1n_def} and \eqref{eq:cov_N_x_def} $\bm \Sigma_{N}^{(2)}= (\mathrm{Cov}_N^{(2)}(E_{N}^{(1)}))_{2\times 2}$. To show this, we only need to show that there exists universal constant $C$ such that, adopting the abbreviation $e_N(s)\equiv H_N(s,E_N^{(1)})$, $\limsup_{N\rightarrow\infty}|\mathrm{Cov}_N^{(2)}(E_N^{(1)})|<C<1$. We will apply Lemma \ref{lem:upper_bound_cov} to prove the result. Recall 
    \small
    \begin{align*}
	    \mathrm{Cov}_N^{(2)}(E_N^{(1)})=\frac{(e_N(0)(1- e_N(0)))^{1/2}\E[Y_{uN}(0)]\E[Y_{uN}(1)]}{(\E[Y_{uN}^2(0)]-e_N(0)\E[Y_{uN}(0)]^2)^{1/2}(\E[Y_{uN}^2(1)]-(1- e_N(0))\E[Y_{uN}(1)]^2)^{1/2}}.
    \end{align*}
    \normalsize
    Note $e_N(s)\in (0,1)$ almost surely by the clip assumption, $\lim_{N\rightarrow\infty}\E[Y_{uN}^p(s)]$ converges for $p=1,2$ and $s=0,1$ by Assumption \ref{assu:moment_condition}, and $\liminf_{N\rightarrow\infty}\mathrm{Var}[Y_{uN}(s)]>0$ for $s=0,1$ by Assumption \ref{assu:moment_condition}. Then we can apply Lemma \ref{lem:upper_bound_cov} with $a_N=e_N(0), X_N=Y_{uN}(0)$ and $Y_N=Y_{uN}(1)$, we know $\limsup_{N\rightarrow\infty}|\mathrm{Cov}_N^{(2)}(E_{N}^{(1)})|<C<1$ almost surely where $C$ is a universal constant. Now, we prove the second claim. Applying Lemma~\ref{lem:holder_continuity_Frobenius}, we have
    \begin{align*}
		\|(\bm \Sigma_N^{(2)}(x_1))^{1/2}-(\bm \Sigma_N^{(2)}(x_2))^{1/2}\|_{\mathrm{F}}\lesssim \|\bm \Sigma_N^{(2)}(x_1)-\bm \Sigma_N^{(2)}(x_2)\|_{\mathrm{F}}.
	\end{align*}
	Define $V_N(x)\equiv V_N(0,x)V_N(1,x)$. Then recalling the definition of $\bm \Sigma_N^{(2)}$, we notice
	\small 
	\begin{align*}
		&
		\|\bm \Sigma_N^{(2)}(x_1)-\bm \Sigma_N^{(2)}(x_2)\|_{\mathrm{F}}\\
		&
		=\sqrt{2}|\mathrm{Cov}_N^{(2)}(x_1)-\mathrm{Cov}_N^{(2)}(x_2)|\\
		&
		\leq \sqrt{2}(H_N(0,x_1)H_N(1,x_1))^{1/2}|\E[Y_{uN}(0)]\E[Y_{uN}(1)]|\frac{\left|V_N^{1/2}(x_1)-V_N^{1/2}(x_2)\right|}{V_N^{1/2}(x_1)V_N^{1/2}(x_2)}\\
		&
		\qquad + \sqrt{2}\left|(H_N(0,x_1)H_N(1,x_1))^{1/2}-(H_N(0,x_2)H_N(1,x_2))^{1/2}\right|\frac{|\E[Y_{uN}(0)]\E[Y_{uN}(1)]|}{V_N^{1/2}(x_2)}\\
		&
		\equiv C_{N,1}+C_{N,2}.
	\end{align*}
	\normalsize
	We bound $C_{N,1}$ and $C_{N,2}$ separately.
	\begin{enumerate}
		\item \textbf{Treatment for $C_{N,1}$.} Notice $V_N(0,x_1)$ is uniformly lower and upper bounded, proved as in Lemma \ref{lem:uniform_lower_upper_bound_variance}. Then we denote the uniform lower and upper bounds respectively as $c_v,C_v$, i.e.,
		\begin{align*}
			c_v\leq \liminf_{N\rightarrow\infty}\inf_{x}V_N(s,x)\leq \limsup_{N\rightarrow\infty}\sup_{x}V_N(s,x)\leq C_v\quad\text{ for any }s=0,1.
		\end{align*}
		Then we have
		\small
		\begin{align*}
			\frac{\left|V_N^{1/2}(x_1)-V_N^{1/2}(x_2)\right|}{V_N^{1/2}(x_1)V_N^{1/2}(x_2)}
			&
			\leq \frac{|V_N(x_1)-V_N(x_2)|}{c_v^2(V_N^{1/2}(x_1)+V_N^{1/2}(x_2))}\\
			&
			\leq \frac{|V_N(x_1)-V_N(x_2)|}{2c_v^3}\\
			&
			\leq \frac{C_v}{2c_v^3}(|V_N(0,x_1)-V_N(0,x_2)|+|V_N(1,x_1)-V_N(1,x_2)|)\\
			&
			\leq \frac{C_v}{2c_v^3}(\E[Y_{uN}(s)])^2\sum_{s=0,1}|e(s,h_N(x_1))-e(s,h_N(x_2))|
		\end{align*}
		\normalsize
		where the last inequality is due to the result already proved for $V_N(s,x)$. Therefore, by the bound $H_N(s,x_1)\leq 1$, we have $|C_{N,1}|\lesssim \sum_{s=0,1}|e(s, h_N(x_1))-e(s, h_N(x_2))|$.
	
		\item \textbf{Treatment for $C_{N,2}$.} Since $\E[Y_{uN}(s)]$ and $V_N(x)$ are uniformly lower and upper bounded, we have
		\begin{align*}
			C_{N,2}
			&
			\lesssim \left|(H_N(0,x_1)H_N(1,x_1))^{1/2}-(H_N(0,x_2)H_N(1,x_2))^{1/2}\right|\\
			&
			\leq \sum_{s=0,1}|H_N^{1/2}(s,x_1)-H_N^{1/2}(s,x_2)|.
		\end{align*}
	\end{enumerate}
	We conclude the proof.
\end{proof}

\section{Proof of Theorem \ref{thm:weak_convergence_W_N}}\label{sec:proof_weak_convergence}

\subsection{General proof roadmap for weak convergence result}\label{sec:proof_roadmap}

Before presenting the general proof roadmap, we first define the following notations. 

\paragraph{Random variables in the limiting distributions.}

We first recall the following definitions.
\begin{align}\label{eq:V_2_def}
	V^{(t)}(s)= \lim_{N\rightarrow\infty}\E[Y_{uN}^2(s)]-H^{(t)}(s)\lim_{N\rightarrow\infty}\E[Y_{uN}(s)]^2
\end{align}
and 
\small
\begin{align}\label{eq:H_N_def}
	H^{(1)}(s)= e(s)\quad\text{and}\quad H^{(2)}(s)= 
	\begin{cases}
		\max\{\bar l, e(s,\mathcal{S}^\infty((A^{(1)},V^{(1)}),c))\} & \text{under Assumption \ref{assu:constant_weighting}}\\
		e(s,\mathcal{S}^\infty((A^{(1)},V^{(1)}),c)) & \text{under Assumption \ref{assu:adaptive_weighting}}
	\end{cases}.
\end{align}
\normalsize 
Define $A^{(t)} \equiv (A^{(t)}(0),A^{(t)}(1)),V^{(t)}\equiv (V^{(t)}(0),V^{(t)}(1))$ and $H^{(t)}\equiv (H^{(t)}(0),H^{(t)}(1))$. Furthermore, recall the asymptotic covariance matrix $\bm \Sigma^{(t)}\equiv (\mathrm{Cov}^{(t)})_{2\times 2}$, where the covariance is defined as in Appendix~\ref{sec:explicit_form_asymptotic_distribution}. We will just denote $\mathrm{Cov}^{(2)}=\mathrm{Cov}^{(2)}(A^{(1)})$ for simplicity.

\paragraph{Random variables related to observed data.}

For the ease of presentation, we rewrite the weight vector as
\begin{align*}
	H_N^{(t)}=(H_N^{(t)}(0),H_N^{(t)}(1)),\ H_N^{(1)}(s)=e(s),\ H_N^{(2)}(s)\equiv H_N(s, E_N^{(1)})=e_N(s,\mathcal{H}_{t-1})
\end{align*}
and the variance vector as
\begin{align*}
	V_{N}^{(t)}\equiv (V_{N}^{(t)}(0),V_{N}^{(t)}(1))\quad\text{where}\quad V_N^{(t)}(s)\equiv \E[Y_{uN}^2(s)]-H_N^{(t)}(s)\E[Y_{uN}(s)]^2.
\end{align*}
Further, we will write $(H_N^{(t)})^{1/2}\equiv ((H_N^{(t)}(0))^{1/2},(H_N^{(t)}(1))^{1/2})$ and similarly, we define $(H^{(t)})^{1/2}\equiv((H^{(t)}(0))^{1/2},(H^{(t)}(1))^{1/2})$. Also, to slightly abuse the notation, we define
\begin{align*}
	\Lambda_{N}^{(t)}\equiv (\Lambda_{N}^{(t)}(0),\Lambda_{N}^{(t)}(1))\quad\text{where}\quad\Lambda_N^{(t)}(s)\equiv \frac{(H_N^{(t)}(s))^{1/2}}{(V_{N}^{(t)}(s))^{1/2}}\frac{1}{N_t^{1/2}}\sum_{u=1}^{N_t}(\hat{\Lambda}_{uN}^{(t)}(s)-\E[Y_{uN}(s)]).
\end{align*}
Define the matrix $\bm \Sigma_N^{(t)}\equiv (\mathrm{Cov}_N^{(t)})_{2\times 2}$, where the covariance is defined as
\begin{align}\label{eq:sigma_N_E_1n_def}
	\mathrm{Cov}_N^{(t)}\equiv \frac{-(H_N^{(t)}(0)H_N^{(t)}(1))^{1/2}}{(V_{N}^{(t)}(0))^{1/2}(V_{N}^{(t)}(1))^{1/2}}\E[Y_{uN}(0)]\E[Y_{uN}(1)].
\end{align}

\paragraph{Random vector joining two stages.}

Recalling the definitions of $H^{(t)}(s)$ and $V^{(t)}(s)$ in \eqref{eq:H_N_def} and \eqref{eq:V_2_def}, we define the limiting vectors
\begin{align}\label{eq:mathbb_A_definition}
	\mathbb{W}\equiv (W_1,W_2)\quad\text{where}\quad  W_t\equiv (Z_t ,V^{(t)},H^{(t)}, \mathrm{Vec} ( (\bm\Sigma^{(t)})^{1/2})).
\end{align} 
where $Z_1,Z_2$ are independently generated from $ N(\bm 0,\bm I_2)$. We use $\mathrm{Vec}(\bm\Sigma)$ to denote the vectorization of matrix $\bm\Sigma$ and the order is by column. Now we define the empirical counterparts for $W_t$:
\begin{align}\label{eq:E_N_t_def}
	E_{N}^{(t)}\equiv \left((\bm \Sigma_N^{(t)})^{-1/2}\Lambda_{N}^{(t)}, V_N^{(t)},H_N^{(t)},\mathrm{Vec}((\bm \Sigma_N^{(t)})^{1/2}),(H_N^{(t)})^{1/2}\right),\ t=1,2.
\end{align}
and 
\begin{align}\label{eq:bar-E-N-def}
	\bar E_{N}^{(1)}\equiv \left((\bm \Sigma_N^{(1)})^{-1/2}\Lambda_{N}^{(1)}, V^{(1)},H^{(1)},\mathrm{Vec}((\bm \Sigma^{(1)})^{1/2}),(H^{(1)})^{1/2}\right).
\end{align}
It is easy to see that $S_N^{(1)}(0)-S_N^{(1)}(1)=h_N(E_N^{(1)})$, recalling the definition of $h_N$ as in~\eqref{eq:def_h_c_function}. In fact,
\begin{align*}
	S_N^{(1)}(0)-S_N^{(1)}(1)=\Lambda_{N}^{(1)}(0)\cdot \frac{(V_{N}^{(1)}(0))^{1/2}}{(H_N^{(1)}(0))^{1/2}}-\Lambda_{N}^{(1)}(1)\cdot \frac{(V_{N}^{(1)}(1))^{1/2}}{(H_N^{(1)}(1))^{1/2}}+\frac{1}{\sqrt{2}}c_N=h_N(E_N^{(1)}).
\end{align*}
Next, we define an intermediate \textbf{auxiliary} random vector joining the limiting random vector $M_2$ and empirical random vector $E_N^{(2)}$, with $Z_2$ considered in~\eqref{eq:mathbb_A_definition}:
\begin{align}\label{eq:E_N_a_x}
	E_{N}^{\textbf{a}}(x)\equiv (Z_2, V_N^{(2)}(x),H_N^{(2)}(x), \mathrm{Vec}((\bm \Sigma_N^{(2)}(x))^{1/2}), (H_N^{(2)}(x))^{1/2}),\ \forall x\in\mathbb{R}^{10}.
\end{align}
Recall $V_N^{(2)}(x),H_N^{(2)}(x)$ and $\bm\Sigma_N^{(2)}(x)$ are defined as in Section~\ref{sec:notation_main_proof}.

For the reader's convenience, we summarize the key random vectors used in the proof and their roles:
\begin{itemize}
	\item \textbf{Limiting vectors:} $\mathbb{W}=(W_1,W_2)$ defined in~\eqref{eq:mathbb_A_definition}, where $W_t=(Z_t,V^{(t)},H^{(t)},\mathrm{Vec}((\bm\Sigma^{(t)})^{1/2}))$ collects the stage-$t$ limiting quantities.
	\item \textbf{Empirical vectors:} $E_N^{(t)}$ defined in~\eqref{eq:E_N_t_def}, the finite-sample counterpart of $W_t$. The goal is to show $(E_N^{(1)},E_N^{(2)})\Rightarrow \mathbb{W}$.
	\item \textbf{Auxiliary vectors:} $\bar E_N^{(1)}$ defined in~\eqref{eq:bar-E-N-def}, a hybrid that uses the empirical first-stage Gaussian component $(\bm\Sigma_N^{(1)})^{-1/2}\Lambda_N^{(1)}$ but replaces all other entries by their population limits. Also, $E_N^{\mathbf{a}}(x)$ defined in~\eqref{eq:E_N_a_x}, an intermediate vector that substitutes the limiting Gaussian $Z_2$ for the empirical second-stage component while retaining the empirical variance, probability, and covariance evaluated at a generic first-stage realization $x$.
\end{itemize}

Proofs for both weighting choices in Theorem~\ref{thm:weak_convergence_W_N} follow from the same proof roadmap. Now we present this general proof roadmap.

\paragraph{Proof roadmap.}

First, the following lemma shows the consistency of $\WIPW(s)$ and $\WIPWS(s)$. 

\begin{lemma}[Consistency of $\WIPW$ and $\WIPWS$]\label{lem:consistency_WIPW}
	Under Assumption \ref{assu:moment_condition}-\ref{assu:sampling_design} and either Assumption \ref{assu:constant_weighting} or Assumption \ref{assu:adaptive_weighting}, we have
	\begin{align*}
		\WIPW(s)-\E[Y_{uN}(s)]=O_p(N^{-1/2})\quad\text{and}\quad\WIPWS(s)-\E[Y_{uN}^2(s)]=O_p(N^{-1/2}),
	\end{align*}
	where we recall the estimators $\WIPW(s)$ and $\WIPWS(s)$ as in \eqref{eq:WIPW_estimator} and \eqref{eq:WIPWS_estimator}.
\end{lemma}
\noindent The proof of Lemma~\ref{lem:consistency_WIPW} can be found in Appendix~\ref{sec:proof_lem:consistency_WIPW}. Now we prove the weak convergence. We summarize the roadmap as follows:

\paragraph{Step 1:}
Define
\begin{align*}
	R_{V}^{(t)}(s)\equiv \frac{\sum_{u=1}^{N_t}(\hat{\Lambda}_{uN}^{(t)}(s)-\WIPW(s))^2}{\sum_{u=1}^{N_t}(\hat{\Lambda}_{uN}^{(t)}(s)-\E[Y_{uN}(s)])^2}\quad\text{and}\quad S_{V}^{(t)}(s)\equiv \frac{H_N^{(t)}(s)}{N_tV^{(t)}_{N}(s)}\sum_{u=1}^{N_t}(\hat{\Lambda}_{uN}^{(t)}(s)-\E[Y_{uN}(s)])^2.
\end{align*}
We first prove that for $t=1,2$
\begin{align*}
	|R_V^{(t)}(s)-1|=O_p(N^{-1/2})\quad\text{and}\quad |S_V^{(t)}(s)-1|=
	\begin{cases}
		O_p(N^{-1/2}l_N^{-1/2}). & \text{under Assumption \ref{assu:adaptive_weighting}},\\
		O_p(N^{-1/2}).& \text{under Assumption \ref{assu:constant_weighting}}.
	\end{cases}
\end{align*}

\paragraph{Step 2:}

Prove $(E_{N}^{(1)},E_{N}^{(2)})\convd \mathbb W$ where $E_N^{(t)}$ is defined as in \eqref{eq:E_N_t_def} and $\mathbb W$ is defined as in \eqref{eq:mathbb_A_definition}. This step involves the analysis on the different choice of $h_{N}^{(t)}(s)$ for Assumption \ref{assu:constant_weighting} and \ref{assu:adaptive_weighting}. 
\edit{
In fact, it suffices to show that for any bounded Lipschitz function $f$,
\begin{align}\label{eq:lip_convergence_E1_E2}
	\E\left[f(E_{N}^{(1)},E_{N}^{(2)})\right]\rightarrow\E\left[f(\mathbb{W})\right]
\end{align}
Without loss of generality, we assume 
\begin{align}\label{eq:choice_of_f}
	\|f\|_{\infty}\leq \frac{1}{2},\ \sup_{x,y\in\mathbb{R}^{24}}|f(x)-f(y)|\leq 1,\sup_{x,y\in\mathbb{R}^{24}}\frac{|f(x)-f(y)|}{\|x-y\|}\leq \frac{1}{2}.
\end{align}
We now state a stronger result, provide a quantitative CLT bound for $(E_N^{(1)},E_N^{(2)})$.
\begin{lemma}[Quantitative CLT for $(E_{N}^{(1)},E_{N}^{(2)})$]\label{lem:quantitative-CLT}
	Suppose the conidionts of Theorem~\ref{thm:weak_convergence_W_N}, then, for $f$ satisfying conditions in~\eqref{eq:choice_of_f}, we have 
	\begin{align*}
		\left|\E[f(E_N^{(1)},E_N^{(2)})]-\E[f(\mathbb{W})]\right|\rightarrow0.
	\end{align*} 	
	Furthermore, under conditions in Theorem~\ref{thm:quantitative_CLT_W_N}, we have 
	\begin{align}\label{eq:step2-rate}
		\left|\E[f(E_N^{(1)},E_N^{(2)})]-\E[f(\mathbb{W})]\right|\lesssim 
		\begin{cases}
			\frac{1}{N^{1/2}}& \text{under Assumption~\ref{assu:constant_weighting}};\\[5pt]
			l_N^{1/2}+\frac{1}{\sqrt{N}l_N^{1/2}} & \text{under Assumption~\ref{assu:adaptive_weighting}},
		\end{cases}
	\end{align}
\end{lemma}
}

\paragraph{Step 3:}
We define 
\small
\begin{align*}
	\hat{I}_N\equiv W_N-\frac{\sqrt{N}(\E[Y_{uN}(0)]-\E[Y_{uN}(1)])}{(N\hat{V}_N(0)+N\hat{V}_N(1))^{1/2}}\quad\text{and}\quad\hat{I}_{U}\equiv \sqrt{N}(T_N-(\E[Y_{uN}(0)]-\E[Y_{uN}(1)])).
\end{align*}
\normalsize
In order to find the asymptotic distribution of $\hat{I}_N,\hat{I}_U$, we can rewrite $\hat{I}_N,\hat{I}_U$ with different weighting methods as a function of the weak limit of $(S_V^{(1)}, R_V^{(1)}, E_N^{(1)}, S_V^{(2)}, R_V^{(2)}, E_N^{(2)})$ and apply Slutsky's Lemma and the continuous mapping theorem. We consider the following four cases:
\begin{enumerate}
	\item \textbf{$W_N$ with constant weighting:} we can write $\hat{I}_N$ with constant weighting as
	\begin{align*}
		\frac{\sum_{t=1}^2\Lambda_N^{(t)}(0)(V_N^{(t)}(0)/H_N^{(t)}(0))^{1/2} -\sum_{t=1}^2\Lambda_N^{(t)}(1) (V_N^{(t)}(1)/H_N^{(t)}(1))^{1/2}}{(\sum_{t=1}^2V_N^{(t)}(0) S_V^{(t)}(0) R_V^{(t)}(0)/H_N^{(t)}(0)+\sum_{t=1}^2V_N^{(t)}(1) S_V^{(t)}(1)R_V^{(t)}(1)/H_N^{(t)}(1))^{1/2}};
	\end{align*}
	\item \textbf{$W_N$ with adaptive weighting:} we can write $\hat{I}_N$ with adaptive weighting as
	\begin{align*}
		\frac{R_N(0)\sum_{t=1}^2\Lambda_N^{(t)}(0)(V_N^{(t)}(0))^{1/2} -R_N(1) \sum_{t=1}^2\Lambda_N^{(t)}(1) (V_N^{(t)}(1))^{1/2}}{(R_N^2(0)\sum_{t=1}^2V_N^{(t)}(0) S_V^{(t)}(0) R_V^{(t)}(0)+R_N^2(1)\sum_{t=1}^2V_N^{(t)}(1) S_V^{(t)}(1)R_V^{(t)}(1))^{1/2}},
	\end{align*}
	where $R_N^{-1}(s)\equiv \sum_{t=1}^2 (H_N^{(t)}(s))^{1/2}$;
	\item \textbf{$T_N$ with constant weighting:} we can write $\hat{I}_U$ with constant weighting as
	\begin{align*}
		\frac{1}{\sqrt{2}}\left(\sum_{t=1}^2\Lambda_N^{(t)}(0)(V_N^{(t)}(0)/H_N^{(t)}(0))^{1/2} -\sum_{t=1}^2\Lambda_N^{(t)}(1) (V_N^{(t)}(1)/H_N^{(t)}(1))^{1/2}\right);
	\end{align*}
	\item \textbf{$T_N$ with adaptive weighting:} we can write $\hat{I}_U$ with adaptive weighting as 
	\begin{align*}
		\sqrt{2}\left(R_N(0)\sum_{t=1}^2\Lambda_N^{(t)}(0)(V_N^{(t)}(0))^{1/2} -R_N(1) \sum_{t=1}^2\Lambda_N^{(t)}(1) (V_N^{(t)}(1))^{1/2}\right).
	\end{align*}
\end{enumerate} 

\paragraph{Proof of qualitative CLT.} We use the results $R_V^{(t)}(s),S_V^{(t)}(s)\convp 1,t=1,2$ as well as the weak convergence of $(E_{N}^{(1)},E_{N}^{(2)})$ to derive the weak convergence with the help of Slutsky's Lemma and continuous mapping theorem. This proves the qualitative CLT part. To see this, we will only work out the case for $W_N$ with constant weighting since the proof for the other cases are similar and even simpler. We define the random vector $(\breve E_N^{(1)},\breve E_N^{(2)})$ and $(\check E_N^{(1)},\check E_N^{(2)})$, where $\breve E_N^{(t)}\equiv (E_N^{(t)},S_V^{(t)},R_V^{(t)}),\ \check E_N^{(t)}\equiv (E_N^{(t)},\bm 1_2,\bm 1_2)$ and
\begin{align*}
	S_V^{(t)}\equiv (S_V^{(t)}(0),S_V^{(t)}(1)),\ R_V^{(t)}\equiv (R_V^{(t)}(0),R_V^{(t)}(1)),\ \bm 1_2\equiv (1,1).
\end{align*}
Then we apply Lemma \ref{lem:Slutsky} with $X_N=(\check E_N^{(1)},\check E_N^{(2)})$ and $Y_N = (\breve E_N^{(1)},\breve E_{N}^{(2)})$ by noticing that $X_N$ converge weakly to $(W_1,\bm 1_2,\bm 1_2,W_2,\bm 1_2,\bm 1_2)$ and $R_V^{(t)},S_V^{(t)}\convp \bm 1_2$ for $t=1,2$. Therefore we obtain $Y_N\convd (W_1,\bm 1_2,\bm 1_2,W_2,\bm 1_2,\bm 1_2)$. Then we can use continuous mapping lemma (Lemma \ref{lem:continuous_mapping_lem}) to derive the weak convergence of $\hat I_N$ by writing $\hat I_N$ as a continuous function of $Y_N$.

\paragraph{Proof of quantitative CLT.} Since all 4 random variables can be written as a continuous function of the random vector $(S_V^{(1)}, R_V^{(1)}, E_N^{(1)}, S_V^{(2)}, R_V^{(2)}, E_N^{(2)})$, it suffices to show that the derivatives of these continuous functions are bounded at the domain of the random vector. This is not hard to justify and we will use $T_N$ with adaptive weighting to illustrate. In fact, by Lemma~\ref{lem:uniform_lower_upper_bound_variance}, $V_N^{(t)}(s)$ are lower and upper bouned uniformly. Moreover, $R_N(s)=(\sum_{t=1}^2(H_N^{(t)}(s))^{1/2})^{-1}$ as a function of $((H_N^{(1)}(s))^{1/2},(H_N^{(2)}(s))^{1/2})$, its dervative is uniformly upper bounded because $H_N^{(1)}(s)=e(s)>0$. Therefore, since for any bounded Lipschitz function $f$, the following holds: 
\begin{align*}
	&
	\left|f(S_V^{(1)}, R_V^{(1)}, E_N^{(1)}, S_V^{(2)}, R_V^{(2)}, E_N^{(2)})-f(\bm 1_2, \bm 1_2, W_1, \bm 1, \bm 1_2, W_2)\right|\\
	&
	\qquad \lesssim 
		\begin{cases}
			\frac{1}{N^{1/2}}& \text{under Assumption~\ref{assu:constant_weighting}};\\[5pt]
			l_N^{1/2}+\frac{1}{\sqrt{N}l_N^{1/2}} & \text{under Assumption~\ref{assu:adaptive_weighting}},
		\end{cases}
\end{align*}
then $f(T_N)-f(\mathbb{W}_{\mathcal{U}}^{\mathcal{A}})$ has the same rate of convergence by the above argument.
	
\subsection{Proof of Step 1 in Appendix \ref{sec:proof_roadmap}}\label{sec:R_V_S_V_proof}

We will show the convergence of $R_V^{(t)}(s)$ and $S_V^{(t)}(s)$ for $s=0,1$ with different choice of $h_N^{(t)}(s)$.

\subsubsection{Proof of convergence of $R_V^{(t)}(s)$}\label{sec:proof_R_V_convergence}

To first show the convergence of $R_V^{(t)}(s)$, we give the following lemma.

\begin{lemma}[Convergence of $R_V^{(t)}(s)$]\label{lem:sufficient_condition_R_V_convergence}
	Suppose the Assumption \ref{assu:moment_condition}-\ref{assu:sampling_design} hold and either Assumption \ref{assu:adaptive_weighting} or Assumption \ref{assu:constant_weighting} holds. If the following statements are true: 
	\begin{enumerate}
		\item $\WIPW(s)-\E[Y_{uN}(s)]=O_p(N^{-1/2})$ for any $s\in \{0,1\}$;
		\item $W_N^{(t)}(s)\equiv\sum_{u=1}^{N_t}e_N(s,\mathcal{H}_{t-1})(\hat{\Lambda}_{uN}^{(t)}-\E[Y_{uN}(s)])^2/N_t$ is asymptotically lower bounded;
	\end{enumerate}
	then we have for any $s,t,|R_V^{(t)}(s)-1|=O_p(N^{-1/2})$.
\end{lemma}

\noindent Since the consistency has been proved in Lemma \ref{lem:consistency_WIPW}, it suffices to prove $W_N^{(t)}(s)$ is stochastically lower bounded. We first present a useful lemma. 

\begin{lemma}[Asymptotic representation of $W_N^{(t)}(s)$]\label{lem:weak_law_W_N}
	Suppose the Assumption \ref{assu:moment_condition}-\ref{assu:sampling_design} hold and either Assumption \ref{assu:adaptive_weighting} or Assumption \ref{assu:constant_weighting} holds. Then we have 
	\begin{align*}
		W_N^{(t)}(s)=\E[Y_{uN}^2(s)]-e_N(s,\mathcal{H}_{t-1})(\E[Y_{uN}(s)])^2+		\begin{cases}
		    O_p(N^{-1/2}l_N^{-1/2}). & \text{under Assumption \ref{assu:adaptive_weighting}},\\
		    O_p(N^{-1/2}).& \text{under Assumption \ref{assu:constant_weighting}}.
		\end{cases}
	\end{align*}
\end{lemma}
\noindent By Lemma \ref{lem:weak_law_W_N}, we know $W_N^{(t)}(s)=\E[Y_{uN}^2(s)]-e_N(s,\mathcal{H}_{t-1})(\E[Y_{uN}(s)])^2+o_p(1)$. Since Assumption \ref{assu:moment_condition} guarantees that
\begin{align*}
	\liminf_{N\rightarrow\infty}(\E[Y_{uN}^2(s)]-e_N(s,\mathcal{H}_{t-1})(\E[Y_{uN}(s)])^2)\geq \liminf_{N\rightarrow\infty}(\E[Y_{uN}^2(s)]-(\E[Y_{uN}(s)])^2)>0,
\end{align*}
we know $W_N^{(t)}(s)$ is asymptotically lower bounded.

\subsubsection{Proof of convergence of $S_V^{(t)}(s)$}

We write 
\begin{align*}
	S_{V}^{(t)}(s)
	&
	=\frac{H_N^{(t)}(s)}{N_tV^{(t)}_{N}(s)}\sum_{u=1}^{N_t}(\hat{\Lambda}_{uN}^{(t)}(s)-\E[Y_{uN}(s)])^2\\
	&
	=\frac{1}{N_t}e_N(s,\mathcal{H}_{t-1})\sum_{u=1}^{N_t} (\hat{\Lambda}_{uN}^{(t)}(s)-\E[Y_{uN}(s)])^2\cdot\frac{1}{V_N^{(t)}(s)}\\
	&
	= \frac{W_N^{(t)}(s)}{V_N^{(t)}(s)}.
\end{align*}
Then we know from Lemma \ref{lem:weak_law_W_N} that
\begin{align*}
	&
	W_N^{(t)}(s)-\left(\E[Y_{uN}^2(s)]-e_N(s,\mathcal{H}_{t-1})\E[Y_{uN}(s)]^2\right)\\
	&
	=W_N^{(t)}(s)-V_N^{(t)}(s)\\
	&
	=		
	\begin{cases}
		O_p(N^{-1/2}l_N^{-1/2}). & \text{under Assumption \ref{assu:adaptive_weighting}},\\
		O_p(N^{-1/2}).& \text{under Assumption \ref{assu:constant_weighting}}.
	\end{cases}
\end{align*}
By Lemma \ref{lem:uniform_lower_upper_bound_variance}, we know $\liminf_{N\rightarrow\infty}V_N^{(t)}(s)>0$. Therefore we have 
\begin{align*}
	\left|S_{V}^{(t)}(s)-1\right|=\left|\frac{W_N^{(t)}(s)}{V_{N}^{(t)}(s)}- 1\right|=
	\begin{cases}
		O_p(N^{-1/2}l_N^{-1/2}). & \text{under Assumption \ref{assu:adaptive_weighting}},\\
		O_p(N^{-1/2}).& \text{under Assumption \ref{assu:constant_weighting}}.
	\end{cases}
\end{align*}

\edit{
\subsection{Proof of Step 2 in Appendix \ref{sec:proof_roadmap}}\label{sec:proof_step_2}
}

We first present a useful lemma, characterizing the asymptotic behavior of sampling function in the second stage. 

\begin{lemma}[Convergence of sampling function]\label{lem:as_convergence_sampling_function}
	Suppose the Assumption \ref{assu:moment_condition}-\ref{assu:sampling_design} hold and either Assumption \ref{assu:constant_weighting} or Assumption \ref{assu:adaptive_weighting} holds. Then if a sequence of random variable $M_N$ satisfying
	\begin{align*}
		M_N \rightarrow W_1 \text{ almost surely}
	\end{align*}
	where $W_1$ is defined as in \eqref{eq:mathbb_A_definition}, then we have 
	\begin{align*}
		e(s,h_N(M_N))\rightarrow e(s,h(W_1,c))\quad\text{and}\quad e^{1/2}(s,h_N(M_N))\rightarrow e^{1/2}(s,h(W_1,c))
	\end{align*}
	almost surely. Consequently, we know 
	\begin{align*}
		e(s,h_N(M_N))-e(s, h_N(W_1))\quad\text{and}\quad e^{1/2}(s,h_N(M_N))-e^{1/2}(s, h_N(W_1))
	\end{align*}
	coverge to $0$ almost surely.
\end{lemma}
We divide the proofs into the following steps
\begin{align*}
	\E\left[f(E_{N}^{(1)},E_{N}^{(2)})\right]-\E\left[f(\mathbb{W})\right]
	&
	=\E[f(W_1,E_N^{\textbf{a}}(W_1))]-\E[f(W_1,W_2)]\\
	&
	\qquad+\E[f(E_{N}^{(1)},E_{N}^{(2)})]-\E[f(E_N^{(1)},E_N^{\textbf{a}}(E_N^{(1)}))]\\ 
	&
	\qquad+\E[f(E_N^{(1)},E_N^{\textbf{a}}(E_N^{(1)}))]- \E[f(\bar E_N^{(1)},E_N^{\textbf{a}}(\bar E_N^{(1)}))]\\
	&
	\qquad + \E[f(\bar E_N^{(1)},E_N^{\textbf{a}}(\bar E_N^{(1)}))]-\E[f(W_1,E_N^{\textbf{a}}(W_1))]\\
	&
	\equiv F_1+F_2+F_3+F_4,
\end{align*}
where $\bar E_N^{(1)}$ is defined as in~\eqref{eq:bar-E-N-def}. We will derive the rate for $F_1,F_2,F_3$ and $F_4$, respectively.

\subsubsection{Proof for $F_1$}

By the Lipschitz property of $f$, we can bound 
\begin{align*}
	F_1
	&
	\lesssim \E[\|H_N^{(2)}(W_1)-H^{(2)}\|_2] + \E[\|V_N^{(2)}(W_1)-V^{(2)}\|_2] \\
	&
	\qquad+\E[\|\mathrm{Vec}((\bm \Sigma_N^{(2)}(W_1))^{1/2})-\mathrm{Vec}((\bm \Sigma^{(2)})^{1/2})\|_{2}]+\E[\|(H_N^{(2)}(W_1))^{1/2}-(H^{(2)})^{1/2}\|_2]\\
	&
	\equiv F_{1}^{H_N}+F_{1}^{V_N}+ F_{1}^{\bm\Sigma_N}+F_1^{H_N^{1/2}}
\end{align*}
Now we bound $ F_{1}^{H_N}, F_{1}^{V_N}$ and $F_{1}^{\bm\Sigma_N}$, separately.
   
\paragraph{Upper bound on $F_{N}^{H}$:}
Recall the definition $H_N^{(2)}(W_1)=(H_N(0,W_1),H_N(1,W_1))$, where $H_N(s,W_1)=\min\{1-l_N,\max\{l_N,e(s, h_N(W_1))\}\}$. We will bound $|H_{N}(s,W_1)-H^{(2)}(s)|$ for $s\in \{0,1\}$. 
\begin{enumerate}
	\item \textbf{Under Assumption \ref{assu:constant_weighting}:} In this case, $0<c_l<\bar l=l_N<c_u<1/2$. By the Lipschitz property of $\min\{1-\bar l,\max\{\bar l,x\}\}$ in $x$, we have
	\begin{align*}
		|H_N(s,W_1)-H^{(2)}(s)|\leq |e(s,h_N(W_1))-e(s,h(W_1,c))|.
	\end{align*}
	
	\item \textbf{Under Assumption \ref{assu:adaptive_weighting}:} In this case $\lim_{N\rightarrow\infty}l_N=0$. Then we have
	\begin{align}
		|H_N(s,W_1)-H^{(2)}(s)|
		&\nonumber
		=|\min\{1-l_N,\max\{l_N,e(s,h_N(W_1))\}\}-e(s,h(W_1,c))|\\
		&\nonumber
		\leq |e(s,h_N(W_1))-e(s,h(W_1,c))|\\
		&\nonumber
		\qquad+|l_N-e(s,h(W_1,c))|\indicator(e(s,h_N(W_1))<l_N)\\
		&\nonumber
		\qquad + |1-l_N-e(s,h(W_1,c))|\indicator(e(s,h_N(W_1))>1-l_N)\\
		&\label{eq:bound-on-H-N}
		\leq 3|e(s,h_N(W_1))-e(s,h(W_1,c))|+2l_N.
	\end{align}
\end{enumerate}
Therefore, we can bound 
\begin{align*}
F_1^{H_N}\lesssim e_{H_N}\equiv
\begin{cases}
    \sum_{s=0,1}\E\!\left[\left|e(s,h_N(W_1))-e\bigl(s,h(W_1,c)\bigr)\right|\right] & \text{under Assumption~\ref{assu:constant_weighting}},\\[5pt]
    \sum_{s=0,1}\E\!\left[\left|e(s,h_N(W_1))-e\bigl(s,h(W_1,c)\bigr)\right|\right] + l_N & \text{under Assumption~\ref{assu:adaptive_weighting}}.
\end{cases}
\end{align*}
	
\paragraph{Upper bound on $F_1^{V_N}$:}
	
Recall the expression 
\small
\begin{align*}
	V_N^{(2)}(W_1)=(V_N(0,W_1),V_N(1,W_1))\quad\text{and}\quad V_N(s,W_1)=\E[Y_{uN}^2(s)]-H_N(s,W_1)\E[Y_{uN}(s)]^2.
\end{align*}
\normalsize
Then we can bound, using the bound~\eqref{eq:Y-limit-rate} and bound~\eqref{eq:Y-limit-rate}, 
\begin{align*}
	F_1^{V_N}\lesssim \sum_{s=0,1} \mathcal{L}_{Y_{uN}^2}(s)+\E[|H_N(s,W_1)-H^{(2)}(s)|]+\mathcal{L}_{Y_{uN}}(s)\lesssim \frac{1}{N}+e_{H_N}.
\end{align*}
	
\paragraph{Upper bound on $F_1^{\bm\Sigma_N}$:}
	
We notice that by Lemma \ref{lem:holder_continuity_Frobenius},
\begin{align*}
	F_1^{\bm\Sigma_N}\lesssim \|\bm \Sigma_N^{(2)}(W_1)-\bm \Sigma^{(2)}\|_{\mathrm{F}}= \sqrt{2}|\mathrm{Cov}_N^{(2)}(W_1)-\mathrm{Cov}^{(2)}|.
\end{align*}
Recall the definition 
\begin{align*}
	\mathrm{Cov}_N^{(2)}(W_1)=\frac{-(H_N(0,W_1)H_N(1,W_1))^{1/2}}{V_N^{1/2}(0,W_1)V_N^{1/2}(1,W_1)}\E[Y_{uN}(0)]\E[Y_{uN}(1)]
\end{align*}
and 
\begin{align*}
	\mathrm{Cov}^{(2)}=- \frac{(H^{(2)}(0)H^{(2)}(1))^{1/2}}{(V^{(2)}(0)V^{(2)}(1))^{1/2}}\lim_{N\rightarrow\infty}\left(\E[Y_{uN}(0)]\E[Y_{uN}(1)]\right).
\end{align*}
Now by the calculation for $F_1^{V_N}$, Lemma~\ref{lem:as_convergence_sampling_function} and bound~\eqref{eq:Y-limit-rate}, we have
\begin{align*}
	F_1^{\bm\Sigma_N}\lesssim \frac{1}{N}+\sum_{s=0,1}\E[\big|H_N^{1/2}(s,W_1)-(H^{(2)}(s))^{1/2}\big|]
\end{align*}
Now we bound $|H_N^{1/2}(s,W_1)-(H^{(2)}(s))^{1/2}|$ for any $s\in\{0,1\}$, separately.
\begin{enumerate}
	\item \textbf{Under Assumption \ref{assu:constant_weighting}:} We can easily obtain
	\begin{align*}
		|H_N^{1/2}(s,W_1)-(H^{(2)}(s))^{1/2}|\leq \frac{1}{2\bar l^{1/2}}|e(s,h_N(W_1))-e(s,h(W_1, c))|;
	\end{align*}
	\item \textbf{Under Assumption \ref{assu:adaptive_weighting}:} we develop two type of bounds. First, using bound~\eqref{eq:bound-on-H-N}, we have
	\begin{align*}
		|H_N^{1/2}(s,W_1)-(H^{(2)}(s))^{1/2}|
		&
		=\frac{|H_N(s,W_1)-H^{(2)}(s)|}{H_N^{1/2}(s,W_1)+(H^{(2)}(s))^{1/2}}\\
		&
		\lesssim  \frac{|e(s,h_N(W_1))-e(s,h(W_1,c))|+l_N}{l_N^{1/2}}\\
		&
		\leq l_N^{1/2}+\frac{|e(s,h_N(W_1))-e(s,h(W_1,c))|}{l_N^{1/2}}.
	\end{align*}
	Suppose $e$ is Lipschitz in Assumption~\ref{assu:sampling_design}, then we can further bound using condition~\eqref{eq:Y-limit-rate},
	\begin{align*}
		|H_N^{1/2}(s,W_1)-(H^{(2)}(s))^{1/2}|\lesssim l_N^{1/2}+\frac{|c_N-c|}{l_N^{1/2}}\lesssim l_N^{1/2}+\frac{1}{l_N^{1/2}N^{1/2}}
	\end{align*}
	Now we develop the second type of bound 
	\begin{align*}
		&
		|H_N^{1/2}(s,W_1)-(H^{(2)}(s))^{1/2}|\\
		&
		=|\min\{(1-l_N)^{1/2},\max\{l_N^{1/2},e^{1/2}(s,h_N(W_1))\}\}-e^{1/2}(s,h(W_1, c))|\\
		&
		\leq  |\min\{(1-l_N)^{1/2},\max\{l_N^{1/2},e^{1/2}(s,h_N(W_1))\}\}-e^{1/2}(s,h_N(W_1))|\\
		&
		\qquad + |e^{1/2}(s,h_N(W_1))-e^{1/2}(s,h(W_1,c))|\\
		&
		\lesssim l_N^{1/2}+|e^{1/2}(s,h_N(W_1))-e^{1/2}(s,h(W_1,c))|.
	\end{align*}
\end{enumerate}
Therefore, defining the sequence $e_{\bm\Sigma_{N}}$ as
\begin{align*}
	\begin{cases}
		\sum_{s=0,1}\E\!\left[\left|e(s,h_N(W_1))-e\bigl(s,h(W_1,c)\bigr)\right|\right]& \text{under Assumption~\ref{assu:constant_weighting}};\\[5pt]
		l_N^{1/2}+\sum_{s=0,1}\E\!\left[\left|e^{1/2}(s,h_N(W_1))-e^{1/2}\bigl(s,h(W_1,c)\bigr)\right|\right] & \text{under Assumption~\ref{assu:adaptive_weighting} and~\ref{assu:sampling_design}(2)};\\[5pt]
        l_N^{1/2}+\frac{1}{l_N^{1/2}N^{1/2}} & \text{under Assumption~\ref{assu:adaptive_weighting} and~\ref{assu:sampling_design}(1)},
	\end{cases}
\end{align*}
we can bound $F_1^{\bm\Sigma_N}\lesssim e_{\bm\Sigma_N}$.

\paragraph{Upper bound on $F_1^{H_N^{1/2}}$:} We can bound, by the derivation for $F_1^{\bm\Sigma_N}$: 
\begin{align*}
	F_1^{H_N^{1/2}}\lesssim \sum_{s=0,1}\E[\big|H_N^{1/2}(s,W_1)-(H^{(2)}(s))^{1/2}\big|]\lesssim e_{\bm\Sigma_N}.
\end{align*}

\paragraph{Concluding the bound for $F_1$.} Combining above estimates, we have 
\begin{align*}
	F_1\lesssim F_1^{H_N}+F_1^{V_N}+F_1^{\bm\Sigma_N}+F_1^{H_N^{1/2}}\lesssim \frac{1}{N}+e_{H_N}+e_{\bm\Sigma_N}.
\end{align*}
By Lemma~\ref{lem:as_convergence_sampling_function}, we know $F_1=o(1)$. Moreover, under Assumption~\ref{assu:sampling_design}(1), we have 
\begin{align}\label{eq:F-1-rate}
	F_1\lesssim 		
	\begin{cases}
		\frac{1}{N^{1/2}}& \text{under Assumption~\ref{assu:constant_weighting}};\\[5pt]
		l_N^{1/2}+\frac{1}{\sqrt{N}l_N^{1/2}} & \text{under Assumption~\ref{assu:adaptive_weighting}},
	\end{cases}
\end{align}

\subsubsection{Proof for $F_2$}

We will use Lemma \ref{lem:CLT_BL} to prove the result. Recall the definition of $E_N^{(2)}$ as in definition \eqref{eq:E_N_t_def}
\begin{align*}
	E_N^{(2)}=\left((\bm\Sigma_{N}^{(2)})^{-1/2}\Lambda_{N}^{(2)},V_N^{(2)},H_N^{(2)},\textnormal{Vec}((\bm\Sigma_{N}^{(2)})^{1/2}),(H_N^{(2)})^{1/2}\right).
\end{align*}
Define 
\begin{align*}
	\Lambda_{uN}^{(2)}(s)\equiv \frac{(H_N^{(2)}(s))^{1/2}(\hat{\Lambda}_{uN}^{(2)}(s)-\E[Y_{uN}(s)])}{(V_N^{(2)}(s))^{1/2}}.
\end{align*}
Recall $\Lambda_{N}^{(2)}=\frac{1}{N_2^{1/2}}\sum_{u=1}^{N_2} (\Lambda_{uN}^{(2)}(0),\Lambda_{uN}^{(2)}(1)),V_N^{(2)}=(V_N(0,E_N^{(1)}),V_N(1,E_N^{(1)})),\bm\Sigma_{N}^{(2)}=\bm\Sigma_{N}^{(2)}(E_N^{(1)})$ and $H_N^{(2)}=(H_N(0,E_N^{(1)}),H_N(1,E_N^{(1)}))$. Define $\xi_{u}\equiv  (\bm\Sigma_{N}^{(2)})^{-1/2}(\Lambda_{uN}^{(2)}(0),\Lambda_{uN}^{(2)}(1))$, which plays the role of $W_{uN}$ in Lemma \ref{lem:CLT_BL} with $N=N_2$. It suffices to bound $N_2^{-1/2}\E[\|\xi_{u}\|_2^3|\mathcal{H}^{(1)}_{N}]$. Notice we can bound
\begin{align*}
	\|\xi_{u}\|_2\leq \|(\bm \Sigma_N^{(2)})^{-1/2}\|_2\sum_{s=0,1}(H_N^{(2)}(s))^{1/2}\left|\frac{\hat{\Lambda}_{uN}^{(2)}(s)-\E[Y_{uN}(s)]}{(V_N^{(2)}(s))^{1/2}}\right|.
\end{align*}
so that
\begin{align*}
	\frac{\E\left[\|\xi_{u}\|_2^3|\mathcal{H}^{(1)}_{N}\right]}{N_2^{1/2}}\lesssim \|(\bm \Sigma_N^{(2)})^{-1/2}\|_2^3\sum_{s=0,1}\frac{(H_N^{(2)}(s))^{3/2}\E\left[|\hat{\Lambda}_{uN}^{(2)}(s)-\E[Y_{uN}(s)]|^3|\mathcal{H}^{(1)}_{N}\right]}{N_2^{1/2}(V_{N}^{(2)}(s))^{3/2}}.
\end{align*}
It suffices to prove
\begin{align}\label{eq:upper_bound_Sigma_N_2_eigenvalue}
	\|(\bm \Sigma_N^{(2)})^{-1/2}\|_2=O(1)\quad\text{almost surely},
\end{align}
and for any $s\in\{0,1\}$,
\begin{align}\label{eq:upper_bound_Normalized_variable}
	\frac{(H_N^{(2)}(s))^{3/2}}{N_2^{1/2}(V_N^{(2)}(s))^{3/2}}\E\left[|\hat{\Lambda}_{uN}^{(2)}(s)-\E[Y_{uN}(s)]|^3|\mathcal{H}^{(1)}_{N}\right]\lesssim	
	\begin{cases}
		\frac{1}{N^{1/2}}& \text{under Assumption~\ref{assu:constant_weighting}};\\[5pt]
		\frac{1}{\sqrt{N}l_N^{1/2}} & \text{under Assumption~\ref{assu:adaptive_weighting}},
	\end{cases}.
\end{align}
In fact, claim~\eqref{eq:upper_bound_Sigma_N_2_eigenvalue} has been proved in Lemma~\ref{lem:continuity_sqrt_sampling_function}. Now we prove claim~\eqref{eq:upper_bound_Normalized_variable}. Consider 
\begin{align*}
	\E\left[|\hat{\Lambda}_{uN}^{(2)}(s)-\E[Y_{uN}(s)]|^3|\mathcal{H}^{(1)}_{N}\right]
	&
	\lesssim \E\left[|\hat{\Lambda}_{uN}^{(2)}(s)|^3|\mathcal{H}^{(1)}_{N}\right]+|\E[Y_{uN}(s)]|^3\\
	&
	=\frac{\E[|Y_{uN}^3(s)|]}{(H_N^{(2)}(s))^2}+|\E[Y_{uN}(s)]|^3.
\end{align*}
Then, since $H_N^{(t)}(s)\leq 1$, we have 
\begin{align*}
	(H_N^{(2)}(s))^{3/2}\E\left[|\hat{\Lambda}_{uN}^{(2)}(s)-\E[Y_{uN}(s)]|^3|\mathcal{H}^{(1)}_{N}\right]\lesssim\frac{\E[|Y_{uN}^3(s)|]}{(H_N^{(2)}(s))^{1/2}}+|\E[Y_{uN}(s)]|^3.
\end{align*}
Then by Lemma~\ref{lem:uniform_lower_upper_bound_variance}, we can bound 
\begin{align*}
	\frac{1}{N_2^{1/2}(V_N^{(2)}(s))^{3/2}}(H_N^{(2)}(s))^{3/2}\E\left[|\hat{\Lambda}_{uN}^{(2)}(s)-\E[Y_{uN}(s)]|^3|\mathcal{H}^{(1)}_{N}\right]
	&
	\lesssim \frac{1}{\sqrt{N}(H_N^{(2)}(s))^{1/2}}\\
	&
	\leq \frac{1}{l_N^{1/2}N^{1/2}}.
\end{align*}
Therefore, we have
\begin{align}\label{eq:F-2-estimate}
	F_2\lesssim 	
	\begin{cases}
		\frac{1}{N^{1/2}}& \text{under Assumption~\ref{assu:constant_weighting}};\\[5pt]
		\frac{1}{\sqrt{N}l_N^{1/2}} & \text{under Assumption~\ref{assu:adaptive_weighting}},
	\end{cases}
\end{align}

\subsubsection{Proof for $F_3+F_4$}

We will separate the proof for the convergence and derivation of the rate.

\paragraph{Proof of the convergence: $F_3+F_4=o(1)$.}
We first show that $E_N^{(1)}\convd W_1$ in the following lemma.

\begin{lemma}[Weak convergence of $E_N^{(1)}$]\label{lem:weak_convergence_E_N_1}
	Suppose the Assumption \ref{assu:moment_condition}-\ref{assu:sampling_design} hold and either Assumption \ref{assu:adaptive_weighting} or Assumption \ref{assu:constant_weighting} holds. Then we have $E_N^{(1)}\convd W_1$.
\end{lemma}

\noindent By Skorohod's representation theorem (Lemma \ref{lem:skorohod}) and Lemma \ref{lem:weak_convergence_E_N_1}, there exists a sequence of random variables $\tilde{E}_{N}^{(1)}$ such that $\tilde{E}_{N}^{(1)}\overset{d}{=}E_N^{(1)}$ and $\lim_{N\rightarrow\infty}\tilde{E}_{N}^{(1)}=W_1$ almost surely. Then we write 
\begin{align*}
	(\tilde E_N^{(1)},E_N^{\textbf{a}}(\tilde E_N^{(1)}))=(\tilde E_N^{(1)},Z_2, V_N^{(2)}(\tilde E_N^{(1)}), H_N^{(2)}(\tilde E_N^{(1)}),\mathrm{Vec}((\bm\Sigma_N^{(2)}(\tilde E_N^{(1)}))^{1/2}),(H_N^{(2)}(\tilde E_N^{(1)}))^{1/2})
\end{align*}
and 
\begin{align*}
	(W_1, E_N^{\textbf{a}}(W_1))=(W_1,Z_2, V_N^{(2)}(W_1), H_N^{(2)}(W_1),\mathrm{Vec}((\bm\Sigma_N^{(2)}(W_1))^{1/2}), (H_N^{(2)}(W_1))^{1/2}).
\end{align*}
We intend to apply Lemma \ref{lem:continuous_map_varying} and in order to do so, we need to show, in addition to $\tilde E_N^{(1)}\rightarrow W_1$, that 
\begin{align*}
	\|V_N^{(2)}(\tilde{E}_N^{(1)})-V_N^{(2)}(W_1)\|_2,\ \|H_N^{(2)}(\tilde{E}_N^{(1)})- H_N^{(2)}(W_1)\|_2
\end{align*}
and 
\begin{align*}
	\|(\bm\Sigma_N^{(2)}(\tilde{E}_N^{(1)}))^{1/2}-(\bm\Sigma_N^{(2)}(W_1))^{1/2}\|_{\mathrm{F}},\ \|(H_N^{(2)}(\tilde{E}_N^{(1)}))^{1/2}- (H_N^{(2)}(W_1))^{1/2}\|_2
\end{align*}
converge to $0$ almost surely. In fact, by Lemma \ref{lem:continuity_sqrt_sampling_function}, it suffices to prove
\small
\begin{align*}
	e(s,h_N(\tilde E_N^{(1)}))-e(s,h_N(W_1))=o(1)\quad\text{and}\quad e^{1/2}(s,h_N(\tilde E_N^{(1)}))-e^{1/2}(s,h_N(W_1))=o(1).
\end{align*}
\normalsize
Applying Lemma \ref{lem:as_convergence_sampling_function} by noticing $\tilde E_N^{(1)}\rightarrow W_1$ almost surely, we complete the proof for $F_3+F_4=o(1)$. 

\paragraph{Derivation of the rate of $F_3+F_4$ under Assumption~\ref{assu:sampling_design}(1).} 

It is not hard to show $F_3\lesssim N^{-1/2}l_N^{-1/2}$ under Assumption~\ref{assu:adaptive_weighting} and $F_3\lesssim N^{-1/2}$ under Assumption~\ref{assu:constant_weighting} so that we will only work on $F_4$. Now define, for $x\in\mathbb{R}^2$,
\begin{align}
	\bar h_N(x)\equiv \left(\frac{(V^{(1)}(0))^{1/2}}{(H^{(1)}(0))^{1/2}},-\frac{(V^{(1)}(1))^{1/2}}{(H^{(1)}(1))^{1/2}}\right)\bm\Sigma^{1/2}x+\frac{1}{\sqrt{2}}c_N.
\end{align}
Then we know $\bar h_N(\bm\Sigma_{N}^{-1/2}\Lambda_{N}^{(1)})=h_N(\bar E_N^{(1)})$ and $\bar h_N(Z_1)=h_N(W_1)$. It is not hard to observe that $\bar h_N(x)$ is a Lipschitz function for any $N\in\mathbb{N}_{+}$ with a uniformly bounded Lipschitz constant. Thus revisiting $E_N^{\textbf{a}}(\bar E_N^{(1)})$ and $E_N^{\textbf{a}}(W_1)$, we find the map only depends on the first two arguments, $\bm\Sigma_{N}^{-1/2}\Lambda_{N}^{(1)}$ and $Z_1$. Thus we can bound, under Assumption~\ref{assu:adaptive_weighting},
\begin{align*}
	F_3+F_4\lesssim\frac{1}{\sqrt{N}}+\left(1+\frac{1}{l_N^{1/2}}\right)d_{\mathrm{BL}}(\bm\Sigma_{N}^{-1/2}\Lambda_{N}^{(1)}, Z_1)\lesssim \frac{1}{N^{1/2}}+\frac{1}{l_N^{1/2}N^{1/2}},
\end{align*}
and under Assumption~\ref{assu:constant_weighting}, we have 
\begin{align*}
	F_3+F_4\lesssim \frac{1}{\sqrt{N}}+d_{\mathrm{BL}}(\bm\Sigma_{N}^{-1/2}\Lambda_{N}^{(1)}, Z_1)\lesssim \frac{1}{\sqrt{N}}.
\end{align*}
Therefore, under Assumption~\ref{assu:sampling_design}(1), we have 
\begin{align}\label{eq:F-3-4-estimate}
	F_3+F_4\lesssim
	\begin{cases}
		\frac{1}{N^{1/2}}& \text{under Assumption~\ref{assu:constant_weighting}};\\[5pt]
		l_N^{1/2}+\frac{1}{\sqrt{N}l_N^{1/2}} & \text{under Assumption~\ref{assu:adaptive_weighting}},
	\end{cases}
\end{align}

\subsubsection{Combining the estimates}

Combining~\eqref{eq:F-1-rate},\eqref{eq:F-2-estimate} and~\eqref{eq:F-3-4-estimate}, we have proved claim~\eqref{eq:step2-rate}.

\section{Proof of lemmas in Appendix \ref{sec:proof_weak_convergence}}\label{sec:supporting_lemmas}

\subsection{Proof of lemmas in Appendix \ref{sec:proof_roadmap}}\label{sec:proof_E_2}

\subsubsection{Proof of Lemma \ref{lem:consistency_WIPW}}\label{sec:proof_lem:consistency_WIPW}

\begin{proof}[Proof of Lemma \ref{lem:consistency_WIPW}]
	We will only prove the consistency for $\E[Y_{uN}(s)]$. The proof for the second part is similar. We divide the proof into two cases depending if Assumption~\ref{assu:adaptive_weighting} or Assumption~\ref{assu:constant_weighting} holds.
	\begin{enumerate}
		\item \textbf{Under Assumption \ref{assu:constant_weighting}:} Compute 
		\small
		\begin{align*}
			\mathrm{Var}[\WIPW(s)-\E[Y_{uN}(s)]]
			&
			=\E\left[\frac{1}{N^2}\sum_{t=1}^2\sum_{u=1}^{N_t}\E\left\{\left(\frac{\indicator(A_{uN}^{(t)}=s)}{e_N(s,\mathcal{H}_{t-1})}Y_{uN}^{(t)}-\E[Y_{uN}(s)]\right)^2|\mathcal{H}_1\right\}\right]\\
			&
			=\frac{1}{N^2}\sum_{t=1}^2\sum_{u=1}^{N_t}\E\left[\left(\frac{\E[Y_{uN}^2(s)]}{e_N(s,\mathcal{H}_{t-1})}-\E[Y_{uN}(s)]^2\right)\right]\\
			&
			\leq \E[Y_{uN}^2(s)]\E\left[\frac{1}{2Ne_N(s,\mathcal{H}_{0})}+\frac{1}{2Ne_N(s,\mathcal{H}_{1})}\right].
		\end{align*}
		\normalsize
		Thus we know by Assumption \ref{assu:constant_weighting} that $Ne_N(s,\mathcal{H}_t)\geq N\min\{e(s),\bar l\}$ for any $t=0,1$. This implies
		\begin{align*}
			\mathrm{Var}[\WIPW(s)-\E[Y_{uN}(s)]]=O(N^{-1/2}),
		\end{align*}
		and thus 
		\begin{align*}
			\WIPW(s)-\E[Y_{uN}(s)]=O_p(N^{-1/2}).
		\end{align*}

		\item \textbf{Under Assumption \ref{assu:adaptive_weighting}:} We first can show that 
		\begin{align}
		W_N(s)\equiv \sum_{t=1}^2\sum_{u=1}^{N_t}h_{N}^{(t)}(s)
		&\nonumber
		= \sum_{t=1}^{2} N_t h_N^{(t)}(s)\\
		&\nonumber
		=\frac{1}{2}(N^{1/2}e_N^{1/2}(s,\mathcal{H}_0)+N^{1/2}e_N^{1/2}(s,\mathcal{H}_1))\\
		&\nonumber
		\geq \frac{1}{2}\left(N^{1/2}l_N^{1/2}+N^{1/2}e^{1/2}(s)\right)\\
		&\label{eq:lower_bound_W_N_aw}
		\geq \frac{1}{2}N^{1/2}e^{1/2}(s).
	\end{align}
	Compute 
	\begin{align*}
		&
		\mathrm{Var}[\WIPW(s)-\E[Y_{uN}(s)]]\\
		&
		=\E\left[\frac{\left(\sum_{t=1}^2\sum_{u=1}^{N_t}h_N^{(t)}(s)\left(\frac{\indicator(A_{uN}^{(t)}=s)}{e_N(s,\mathcal{H}_{t-1})}Y_{uN}^{(t)}-\E[Y_{uN}(s)]\right)\right)^2}{W_N^2(s)}\right]\\
		&
		\leq \frac{4}{Ne(s)}\E\left[\left(\sum_{t=1}^2\sum_{u=1}^{N_t}h_N^{(t)}(s)\left(\frac{\indicator(A_{uN}^{(t)}=s)}{e_N(s,\mathcal{H}_{t-1})}Y_{uN}^{(t)}-\E[Y_{uN}(s)]\right)\right)^2\right]\\
		&
		=\frac{4}{N^2e(s)}\sum_{t=1}^2\sum_{u=1}^{N_t}\left(\E[Y_{uN}^2(s)]-e_N(s,\mathcal{H}_{t-1})\E[Y_{uN}(s)]^2\right)\\
		&
		\leq \frac{4\E[Y_{uN}^2(s)]}{Ne(s)}.
	\end{align*}
	Then it suffices to show 
	\begin{align*}
		|\E[\WIPW(s)]-\E[Y_{uN}(s)]|\rightarrow0.
	\end{align*}
	In fact, we can compute 
	\begin{align*}
		|\E[\WIPW(s)]-\E[Y_{uN}(s)]|
		&
		=\left|\E\left[\frac{h_N^{(1)}(s)\sum_{u=1}^{N_1}(\hat{\Lambda}_{uN}^{(1)}(s)-\E[Y_{uN}(s)])}{W_N(s)}\right]\right|\\
		&
		\leq\E\left[\frac{h_N^{(1)}(s)|\sum_{u=1}^{N_1}(\hat{\Lambda}_{uN}^{(1)}(s)-\E[Y_{uN}(s)])|}{W_N(s)}\right]\\
		&
		\leq \E\left[\frac{|\sum_{u=1}^{N_1}(\hat{\Lambda}_{uN}^{(1)}(s)-\E[Y_{uN}(s)])|}{N_1}\right]\\
		&
		\leq \sqrt{\E\left[\left(\frac{\sum_{u=1}^{N_1}(\hat{\Lambda}_{uN}^{(1)}(s)-\E[Y_{uN}(s)])}{N_1}\right)^2\right]}\\
		&
		=\sqrt{\frac{\E[Y_{uN}^2(s)]-e(s)(\E[Y_{uN}(s)])^2}{N_1}}\rightarrow0,
	\end{align*}
	where the second inequality is due to the lower bound \eqref{eq:lower_bound_W_N_aw} and the third inequality is due to Jensen's inequality. Then we know 
	\begin{align*}
		\WIPW(s)-\E[Y_{uN}(s)]=O_p(N^{-1/2}).
	\end{align*}
	\end{enumerate}

\end{proof}

\subsection{Proof of lemmas in Appendix \ref{sec:R_V_S_V_proof}}\label{sec:proof_E_3}

\subsubsection{Proof of Lemma \ref{lem:sufficient_condition_R_V_convergence}}

\begin{proof}[Proof of Lemma \ref{lem:sufficient_condition_R_V_convergence}]
	The idea is to decompose $R_V^{(t)}(s)$ into different pieces and show the convergence rate of each piece. Recall the definition of $W_N^{(t)}(s)$ as in Appendix \ref{sec:proof_R_V_convergence}. Then consider
	\begin{align*}
		R_{V}^{(t)}(s)
		&
		= \frac{\sum_{u=1}^{N_t}(\hat{\Lambda}_{uN}^{(t)}(s)-\WIPW(s))^2}{\sum_{u=1}^{N_t}(\hat{\Lambda}_{uN}^{(t)}(s)-\E[Y_{uN}(s)])^2}\\
		&
		=\frac{e_N(s,\mathcal{H}_{t-1})}{N_tW_N^{(t)}(s)}\sum_{u=1}^{N_t}\left(\hat{\Lambda}_{uN}^{(t)}(s)-\E[Y_{uN}(s)]+\E[Y_{uN}(s)]-\WIPW(s)\right)^2\\
		&
		= 1+\frac{2\left(\E[Y_{uN}(s)]-\WIPW(s)\right)}{W_N^{(t)}(s)}\frac{e_N(s,\mathcal{H}_{t-1})}{N_t}\sum_{u=1}^{N_t}\left(\hat{\Lambda}_{uN}^{(t)}(s)-\E[Y_{uN}(s)]\right)\\
		&
		\qquad +\frac{e_N(s,\mathcal{H}_{t-1})}{W_N^{(t)}(s)}(\E[Y_{uN}(s)]-\WIPW(s))^2.
	\end{align*}
	By the assumption, it suffices to derive the rate for 
	\begin{align}\label{eq:R_V_convergence_step3}
		\frac{e_N(s,\mathcal{H}_{t-1})}{N_t}\sum_{u=1}^{N_t}(\hat{\Lambda}_{uN}^{(t)}(s)-\E[Y_{uN}(s)])=O_p(N^{-1/2}).
	\end{align}
	The intuition behind the validity of \eqref{eq:R_V_convergence_step3} is the summand is mean zero and thus we only need to show the variance of the summand converges to $0$. To this end, we can compute 
	\begin{align*}
		&
		\mathrm{Var}\left[e_N(s,\mathcal{H}_{t-1})\frac{1}{N_t}\sum_{u=1}^{N_t}(\hat{\Lambda}_{uN,s}^{(t)}-\E[Y_{uN}(s)])\right]\\
		&
		=\E\left[\mathrm{Var}\left[\frac{1}{N_t}e_N(s,\mathcal{H}_{t-1})\sum_{u=1}^{N_t}\left(\hat{\Lambda}_{uN}^{(t)}(s)-\E[Y_{uN}(s)]\right)|\mathcal{H}_{t-1}\right]\right]\\
		&
		=\E\left[\frac{e_N^2(s,\mathcal{H}_{t-1})}{N_t^2}\sum_{u=1}^{N_t}\E\left[\left(\hat{\Lambda}_{uN}^{(t)}(s)-\E[Y_{uN}(s)]\right)^2|\mathcal{H}_{t-1}\right]\right]\\
		&
		=\frac{1}{N_t}\E\left[e_N(s,\mathcal{H}_{t-1})\left(\E\left[Y_{uN}^2(s)\right]-e_N(s,\mathcal{H}_{t-1}) \E[Y_{uN}(s)]^2\right)\right]\leq \frac{1}{N_t}\E[Y_{uN}^2(s)]=O(N^{-1/2}).
	\end{align*}
	Thus we proved \eqref{eq:R_V_convergence_step3} by Markov's inequality.
\end{proof}

\subsubsection{Proof of Lemma \ref{lem:weak_law_W_N}}

\begin{proof}[Proof of Lemma \ref{lem:weak_law_W_N}]
	We use Makrov's inequality to prove the claim. To see this, we can compute 
	\begin{align*}
		&
		\E\left[(W_N^{(t)}(s)-(\E[Y_{uN}^2(s)]-e_N(s,\mathcal{H}_{t-1})(\E[Y_{uN}(s)])^2))^2|\mathcal{H}_{t-1}\right]\\
		&
		\leq\frac{1}{N_t}\E\left[e_N^2(s,\mathcal{H}_{t-1})(\hat{\Lambda}_{uN}^{(t)}(s)-\E[Y_{uN}(s)])^4|\mathcal{H}_{t-1}\right]\\
		&
		\leq \frac{8\E[Y_{uN}^4(s)]}{N_te_N(s,\mathcal{H}_{t-1})}+8e_N^2(s,\mathcal{H}_{t-1})\frac{(\E[Y_{uN}(s)])^4}{N_t}.
	\end{align*}
	Since $Ne_N(s,\mathcal{H}_{t-1})=\Omega_p (Nl_N)$ under Assumption \ref{assu:adaptive_weighting} or $Ne_N(s,\mathcal{H}_{t-1})=\Omega_p (N)$ under Assumption~\ref{assu:constant_weighting} and Assumption \ref{assu:moment_condition} guarantees that $\E[Y_{uN}(s)],\E[Y_{uN}^4(s)]$ are uniformly bounded, then we know $\frac{(\E[Y_{uN}(s)])^4}{N}=O(N^{-1})$ and 
	\begin{align*}
		\frac{\E[Y_{uN}^4(s)]}{Ne_N(s,\mathcal{H}_{t-1})}=	
		\begin{cases}
		    O_p(N^{-1/2}l_N^{-1/2}). & \text{under Assumption \ref{assu:adaptive_weighting}},\\
		    O_p(N^{-1/2}).& \text{under Assumption \ref{assu:constant_weighting}}.
		\end{cases}
	\end{align*}
	Therefore, by Markov's inequality, we have
	\begin{align*}
		&
		W_N^{(t)}(s)-\E[e_N(s,\mathcal{H}_{t-1})(\hat{\Lambda}_{uN}^{(t)}(s)-\E[Y_{uN}(s)])^2|\mathcal{H}_{t-1}]\\
		&
		=	
		\begin{cases}
		    O_p(N^{-1/2}l_N^{-1/2}). & \text{under Assumption \ref{assu:adaptive_weighting}},\\
		    O_p(N^{-1/2}).& \text{under Assumption \ref{assu:constant_weighting}},
	\end{cases}
	\end{align*}
\end{proof}

\subsection{Proof of lemmas in Appendix \ref{sec:proof_step_2}}\label{sec:proof_E_4}

\subsubsection{Proof of Lemma \ref{lem:as_convergence_sampling_function}}

\begin{proof}[Proof of Lemma \ref{lem:as_convergence_sampling_function}]
	We notice by the definition of $h(\cdot,c)$ in \eqref{eq:def_h_c_function}, and the definition of $W_1$ in definition \eqref{eq:mathbb_A_definition}, we know $h_N(M_N)\rightarrow h(W_1,c),\text{ almost surely}$, using continuous mapping theorem, for any $c\in [-\infty,0]$. Based on the Assumption \ref{assu:sampling_design}, we divide the proof into two cases.

	\begin{enumerate}
		\item \textbf{When $e(s,x)$ is Lipschitz continuous on $x$.} 
		We note by Lemma \ref{lem:continuous_map_varying} that 
		\begin{align*}
			|e(s,h_N(M_N))-e(s,h(W_1,c))|\overset{a.s.}{\rightarrow}0
		\end{align*}
		is true. Moreover, if a nonnegative function $f$ is Lipschitz continuous and the range is in $[0,1]$, then $f^{1/2}$ is uniformly continuous. This is because $\sqrt{x}$ is a uniformly continuous function in the compact support $[0,1]$. Thus we apply Lemma \ref{lem:continuous_map_varying} again with $f=e(s,x)$ to get
		\begin{align*}
			|e^{1/2}(s,h_N(M_N))-e^{1/2}(s,h(W_1,c))|\overset{a.s.}{\rightarrow}0.
		\end{align*}
		
		\item \textbf{When $e(s,x)$ takes the form $\sum_{k=1}^K c_k \indicator(g(x)\in C_k)$.}
		For both functions $e^{1/2}(s,x)$ and $e(s,x)$, we only need to prove that 
		\begin{align*}
			\indicator(g(h_N(M_N))\in C_k)-\indicator(g(h(W_1,c))\in C_k)\overset{a.s.}{\rightarrow}0,\ \forall k\in[K]
		\end{align*}
		is true. Notice when $c=-\infty$, we know by Assumption \ref{assu:sampling_design} that $g(-\infty)=-\infty\in C_1$. Then we know  
		\begin{align*}
			\indicator(g(h_N(M_N))\in C_1)-\indicator(g(h(W_1,-\infty))\in C_1)=\indicator(g(h_N(M_N))\in C_1)-1\overset{a.s.}{\rightarrow}0.
		\end{align*}
		When $c\in (-\infty,0]$, we know $g(h(W_1,c))$ is a continuous random variable. Indeed, $h(W_1,c)$ is a continuous random variable and $g$ is a continuous function. This means $\P[g(h(W_1,c))\in \partial C_k]=0$ since $\partial C_k$ is of Lebesgue measure zero by the definition of $C_k$ in Assumption \ref{assu:sampling_design}. Then by Lemma \ref{lem:continuous_mapping_lem}, we know 
		\begin{align*}
			\indicator(g(h_N(M_N))\in C_k)-\indicator(g(h(W_1,c))\in C_k)\overset{a.s.}{\rightarrow}0.
		\end{align*}
	\end{enumerate}

	This completes the proof.
\end{proof}

\subsubsection{Proof of Lemma \ref{lem:weak_convergence_E_N_1}}

\begin{proof}[Proof of Lemma \ref{lem:weak_convergence_E_N_1}]
	Recall the expression of $E_N^{(1)}$. We have 
	\begin{align*}
		E_{N}^{(1)}\equiv \left((\bm \Sigma_N^{(1)})^{-1/2}\Lambda_{N}^{(1)}, V_N^{(1)},H_N^{(1)},\mathrm{Vec}((\bm \Sigma_N^{(1)})^{1/2})\right).
	\end{align*}
	The proof can be decomposed to two steps. 
	\begin{enumerate}
		\item We first prove that  
		\begin{align}\label{eq:E_N_1_weak_convergence_step1}
			\|H_N^{(1)}-H^{(1)}\|_2=o(1),\|V_N^{(1)}-V^{(1)}\|_2=o(1),\|(\bm \Sigma_N^{(1)})^{1/2}-(\bm \Sigma^{(1)})^{1/2}\|_{\mathrm{F}}=o(1).
		\end{align}
		\item Then we prove 
		\begin{align}\label{eq:E_N_1_weak_convergence_step2}
			(\bm \Sigma_N^{(1)})^{-1/2}\Lambda_{N}^{(1)}\convd Z,\ Z\sim N(\bm 0,\bm I_2).
		\end{align}
	\end{enumerate}

	\paragraph{Proof of \eqref{eq:E_N_1_weak_convergence_step1}:}

	The convergence of $H_N^{(1)}$ and $V_N^{(1)}$ are obvious. For $\bm \Sigma_N^{(1)}$, we use Lemma \ref{lem:holder_continuity_Frobenius} so that it suffices to prove 
	\begin{align*}
		\|\bm \Sigma_N^{(1)}-\bm \Sigma^{(1)}\|_{\mathrm{F}}=\sqrt{2}|\mathrm{Cov}_N^{(1)}-\mathrm{Cov}^{(1)}|=o(1).
	\end{align*}
	To this end, recall the definition of $\mathrm{Cov}_N^{(1)}$ as in \eqref{eq:sigma_N_E_1n_def},
	\begin{align*}
		\mathrm{Cov}_N^{(1)}= \frac{-(H_N^{(1)}(0)H_N^{(1)}(1))^{1/2}}{(V_{N}^{(1)}(0))^{1/2}(V_{N}^{(1)}(1))^{1/2}}\E[Y_{uN}(0)]\E[Y_{uN}(1)].
	\end{align*}
	Since $\|V_{N}^{(1)}-V^{(1)}\|_2,\|H_N^{(1)}-H^{(1)}\|_2=o(1)$, and 
	\begin{align*}
		0<\liminf_{N\rightarrow\infty}V_N^{(1)}(s)\leq \limsup_{N\rightarrow\infty}V_N^{(1)}(s)<\infty
	\end{align*}
	as proved in Lemma \ref{lem:uniform_lower_upper_bound_variance}, we know $|\mathrm{Cov}_N^{(1)}-\mathrm{Cov}^{(1)}|=o(1)$. The completes the proof for \eqref{eq:E_N_1_weak_convergence_step1}. 

	\paragraph{Proof of \eqref{eq:E_N_1_weak_convergence_step2}:}

	This can be proved easily by applying Lemma \ref{lem:CLT_BL}. We omit the proof. 

\end{proof}

\section{Proof of Theorem \ref{thm:smooth_transition}}\label{sec:proof_smooth_transition}

Given two random variables \(X\) and \(Y\), the \(1\)-Wasserstein distance is defined as 
\[
d_{W_1}(X, Y) \equiv \sup_{\|f\|_{\mathrm{L}} \leq 1} \big|\E[f(X)] - \E[f(Y)]\big|.
\]

\subsection{Proof preparation for Theorem \ref{thm:smooth_transition}}

We will only prove the result for $\mathbb{W}_{\mathcal{U}}^{\mathcal{C}}(c)$ since the other proofs are very similar. We will drop the subscript in $\bar w_{\mathcal{C}}^{(t)}(s)$ and just write $w^{(t)}(s)$. Similarly, we use $\mathbb{W}(c)$ to denote $\mathbb{W}_{\mathcal{U}}^{\mathcal{C}}(c)$. Recall the definition of $\mathbb{W}(c)$ as 
\begin{align*}
	\mathbb{W}(c)=\sum_{t=1}^2 A^{(t)}(0) w^{(t)}(0) -\sum_{t=1}^2 A^{(t)}(1) w^{(t)}(1),\ w^{(t)}(s)=\frac{(2V^{(t)}(s))^{1/2}}{2\sqrt{H^{(t)}(s)}}.
\end{align*}
To further ease the burden of notation, we define $V_p(s)\equiv \lim_{N\rightarrow\infty}\E[Y_{uN}^p(s)]$. Then we can rewrite $V^{(2)}(s)\equiv V_2(s)-H^{(2)}(s)V_1^2(s)$, where 
\begin{align*}
	H^{(2)}(s)=\min\{1-\bar l,\max\{\bar l,e(s,\mathcal{S}^\infty((A^{(1)},V^{(1)}),c))\}\}.
\end{align*}
We will use $V^{(2)}(s,c),H^{(2)}(s,c)$ to denote $V^{(2)}(s)$ and $H^{(2)}(s)$ stress the dependence on $c$. In particular, we can write 
\small
\begin{align*}
	H^{(2)}(s,-\infty)\equiv \min\{1-\bar l,\max\{\bar l,e(s,-\infty)\}\}\quad\text{and}\quad V^{(2)}(s,-\infty)\equiv V_2(s)-H^{(2)}(s,-\infty)V_1^2(s).
\end{align*}
\normalsize
Similarly, we can define 
\begin{align*}
	\mathrm{Cov}^{(2)}(-\infty)=-(H^{(2)}(0,-\infty)H^{(2)}(1,-\infty))^{1/2}/(V^{(2)}(0,-\infty)V^{(2)}(1,-\infty))^{1/2}V_1(0)V_1(1).
\end{align*}
Then with $V^{(2)}(s,-\infty),H^{(2)}(s,-\infty)$, we can define the corresponding weight $w_{-\infty}^{(2)}(s)\equiv (V^{(2)}(s,-\infty))^{1/2}/(2H^{(2)}(s,-\infty))^{1/2}$. Moreover, we define
\begin{align*}
	(A_{-\infty}^{(2)}(0),A_{-\infty}^{(2)}(1))^\top \sim N(\bm 0,\bm \Sigma^{(2)}_{-\infty})\quad\text{where}\quad \bm \Sigma^{(2)}_{-\infty}\equiv (\mathrm{Cov}^{(2)}(-\infty))_{2\times 2}.
\end{align*}
Finally, define $S_{2,s}\equiv A_{-\infty}^{(2)}(s) w_{-\infty}^{(2)}(s)$.

\subsection{Proof of Theorem~\ref{thm:smooth_transition}}

\begin{proof}[Proof of Theorem \ref{thm:smooth_transition}]
	We rewrite the random variable $\mathbb{W}(-\infty)$ as 
	\begin{align*}
		\mathbb{W}(-\infty)= A^{(1)}(0) w^{(1)}(0)-A^{(1)}(1)  w^{(1)}(1)+ S_{2,0} - S_{2,1}.
	\end{align*}
	Notice that $\mathbb{W}(-\infty)$ is a Gaussian random variable with mean zero since $S_{2,0}-S_{2,1}$ is independent with $A^{(1)}(0) w^{(1)}(0)-A^{(1)}(1) w^{(1)}(1)$. The proof will be divded into three steps:
	\begin{itemize}
		\item We first decompose the desired $d_{W_1}(\mathbb{W}(-\infty),\mathbb{W}(c))$ distance into different pieces and bound different pieces by $d_{W_1}$ distances; 
		\item We bound the $d_{W_1}$ distance by $| w^{(2)}(s)- w_{-\infty}^{(2)}(s)|$ and further obtain a bound $| w^{(2)}(s)- w_{-\infty}^{(2)}(s)|$, which just involves $|e(s, \mathcal{S}^\infty((A^{(1)},V^{(1)}),c))-e(s, -\infty)|$; 
		\item We collect all the results to prove the claim.
	\end{itemize}

	\paragraph{Decomposition of $d_{W_1}(\mathbb{W}(-\infty),\mathbb{W}(c))$.}

	Now we first decompose the KS distance into two parts, using the triangle inequality:
	\begin{align*}
		d_{W_1}(\mathbb{W}(-\infty),\mathbb{W}(c))\leq d_{W_1}(A^{(2)}(0) w^{(2)}(0),S_{2,0})+d_{W_1}(A^{(2)}(1) w^{(2)}(1),S_{2,1})\equiv K_0+K_1.
	\end{align*}
	By triangle inequality, we have for $s\in\{0,1\}$,
	\begin{align*}
		K_s\leq d_{W_1}(A^{(2)}(s)w^{(2)}(s),A^{(2)}(s) w_{-\infty}^{(2)}(s))+d_{W_1}(A^{(2)}(s) w_{-\infty}^{(2)}(s),S_{2,s}).
	\end{align*}
	In fact, $d_{W_1}(A^{(2)}(s)w_{-\infty}^{(2)}(s),S_{2,s})=0$ since $w_{-\infty}^{(2)}(s)$ is a constant and $A^{(2)}(s), A_{-\infty}^{(2)}(s)$ have the same distribution. It suffices to study $d_{W_1}(A^{(2)}(s)w^{(2)}(s),A^{(2)}(s)w_{-\infty}^{(2)}(s))$. 

	\paragraph{Bounding $d_{W_1}(A^{(2)}(s)w^{(2)}(s),A^{(2)}(s)w_{-\infty}^{(2)}(s))$.}

	We compute
	\begin{align*}
		&
		d_{W_1}(A^{(2)}(s) w^{(2)}(s),A^{(2)}(s) w_{-\infty}^{(2)}(s))\\
		&
		=\sup_{\|f\|_{\mathrm{L}}\leq 1}\left|\E\left[f(A^{(2)}(s)w^{(2)}(s))\right]-\E\left[f(A^{(2)}(s) w_{-\infty}^{(2)}(s))\right]\right| \\
		&
		\leq \E\left[|A^{(2)}(s) w^{(2)}(s)-A^{(2)}(s) w_{-\infty}^{(2)}(s)|\right]\\
		&
		\leq \sqrt{\E[|A^{(2)}(s)|^2]\E[| w^{(2)}(s)-w_{-\infty}^{(2)}(s)|^2]}=\sqrt{\E[|w^{(2)}(s)- w_{-\infty}^{(2)}(s)|^2]}.
	\end{align*}
	Define 
	\begin{align*}
		D(c,-\infty)\equiv (2V^{(2)}(s,c)H^{(2)}(s,-\infty))^{1/2}+(2V^{(2)}(s,-\infty)H^{(2)}(s,c))^{1/2}.
	\end{align*}
	Notice that when $N$ is large,
	\begin{align*}
		D(c,-\infty)\geq (2V^{(2)}(s,c)H^{(2)}(s,-\infty))^{1/2}\geq  (2\liminf_{N\rightarrow\infty}\mathrm{Var}[Y_{uN}(s)]\bar l)^{1/2}.
	\end{align*}
	Then we can bound
	\small
	\begin{align}
		\left| w^{(2)}(s)- w_{-\infty}^{(2)}(s)\right|
		&\nonumber
		=\left|\frac{(2V^{(2)}(s,c)H^{(2)}(s,-\infty))^{1/2}-(2V^{(2)}(s,-\infty)H^{(2)}(s,c))^{1/2}}{2(H^{(2)}(s,c)H^{(2)}(s,-\infty))^{1/2}}\right|\\
		&\label{eq:upper_bound_diff_bar_w}
		= \left|\frac{V^{(2)}(s,c)H^{(2)}(s,-\infty)-V^{(2)}(s,-\infty)H^{(2)}(s,c)}{(H^{(2)}(s,c)H^{(2)}(s,-\infty))^{1/2}}\right|\frac{1}{D(c,-\infty)}\\
		&\nonumber
		\lesssim \left|\frac{V^{(2)}(s,c)H^{(2)}(s,-\infty)-V^{(2)}(s,-\infty)H^{(2)}(s,c)}{(H^{(2)}(s,c)H^{(2)}(s,-\infty))^{1/2}}\right|.
	\end{align}
	\normalsize
	Next, notice that $\min\{H^{(2)}(s,-\infty),H^{(2)}(s,c)\}\geq \bar l$, so that we can futher bound 
	\small
	\begin{align*}
		\left| w^{(2)}(s)- w_{-\infty}^{(2)}(s)\right|
		&
		\lesssim\left|V^{(2)}(s,c)H^{(2)}(s,-\infty)-V^{(2)}(s,-\infty)H^{(2)}(s,c)\right|\\
		&
		\lesssim V^{(2)}(s,c)|H^{(2)}(s,-\infty)-H^{(2)}(s,c)|+|V^{(2)}(s,c)-V^{(2)}(s,-\infty)|H^{(2)}(s,c).
	\end{align*}
	\normalsize
	We now further bound the RHS by the quantity invovling just $|e(s, \mathcal{S}^\infty((A^{(1)},V^{(1)}),c))-e(s, -\infty)|$. Then we show respectively that 
	\begin{enumerate}
		\item $|H^{(2)}(s,c)-H^{(2)}(s,-\infty)|\leq |e(s, \mathcal{S}^\infty((A^{(1)},V^{(1)}),c))-e(s, -\infty)|$;
		\item $|V^{(2)}(s,c)-V^{(2)}(s,-\infty)|\leq V_1^2(s)|e(s, \mathcal{S}^\infty((A^{(1)},V^{(1)}),c))-e(s, -\infty)|$.
	\end{enumerate}
	For $H^{(2)}(s,c)-H^{(2)}(s,-\infty)$, the claim is true by the Lipschitz property of $\min\{1-l_N,\max\{l_N,x\}\}$. For $V^{(2)}(s,c)-V^{(2)}(s,-\infty)$, we can compute 
	\begin{align*}
		|V^{(2)}(s,c)-V^{(2)}(s,-\infty)|
		&
		=V_1^2(s)|H^{(2)}(s,c)-H^{(2)}(s,-\infty)|\\
		&
		\leq V_1^2(s)|e(s, \mathcal{S}^\infty((A^{(1)},V^{(1)}),c))-e(s, -\infty)|.
	\end{align*}
	By Lemma \ref{lem:uniform_lower_upper_bound_variance}, we know $V_2(s)-V_1^2(s)\lesssim V^{(2)}(s,c)\lesssim V_2(s)$. By the definition of $H^{(2)}(s,c)$ we know $H^{(2)}(s,c)\in(0,1)$. Then combining bound \eqref{eq:upper_bound_diff_bar_w}, we have 
	\begin{align*}
		\left|w^{(2)}(s)-w_{-\infty}^{(2)}(s)\right|\lesssim (V_2(s)+V_1^2(s))|e(s, \mathcal{S}^\infty((A^{(1)},V^{(1)}),c))-e(s, -\infty)|.
	\end{align*}
	Therefore, we obtain the bound 
	\begin{align*}
		d_{W_1}(A^{(2)}(s) w^{(2)}(s),A^{(2)}(s) w_{-\infty}^{(2)}(s))
		&
		\leq \sqrt{\E[|w^{(2)}(s)- w_{-\infty}^{(2)}(s)|^2]}\\
		&
		\lesssim \left(\E[|e(s, \mathcal{S}^\infty((A^{(1)},V^{(1)}),c))-e(s, -\infty)|^2]\right)^{1/2}.
	\end{align*}

	\paragraph{Concluding the proof:}

	Therefore we can bound 
	\begin{align*}
		K_s\leq d_{W_1}(A^{(2)}(s) w^{(2)}(s),A^{(2)}(s) w_{-\infty}^{(2)}(s))\lesssim \left(\E[|e(s, \mathcal{S}^\infty((A^{(1)},V^{(1)}),c))-e(s, -\infty)|^2]\right)^{1/2}.
	\end{align*}
	Thus we conclude 
	\begin{align*}
		d_{W_1}(\mathbb{W}(-\infty),\mathbb{W}(c))\leq C\sum_{s=0,1}\left(\E[|e(s, \mathcal{S}^\infty((A^{(1)},V^{(1)}),c))-e(s, -\infty)|^2]\right)^{1/2}.
	\end{align*}
	Now we prove the convergence of $\E[|e(s, \mathcal{S}^\infty((A^{(1)},V^{(1)}),c))-e(s, -\infty)|^2]$ as $c\rightarrow-\infty$. This is true by dominated convergence theorem since $e(s, \mathcal{S}^\infty((A^{(1)},V^{(1)}),c))-e(s, -\infty)\rightarrow0$ by Assumption \ref{assu:sampling_design}.
\end{proof}

\section{Proof of Theorem \ref{thm:bootstrap}}\label{sec:proof_bootstrap}

\subsection{Necessary definitions}

\paragraph{Random vectors.} We define the random vector 	
\begin{align*}
	W^{(1,b)}\equiv \left(S_1^{(b)}, \hat V^{(1)}, H^{(1)},\mathrm{Vec}((\hat{\bm\Sigma}^{(1)})^{1/2}), (H^{(1)})^{1/2}\right)
\end{align*}
where $H^{(1)}=(H^{(1)}(0),H^{(1)}(1))$ is defined as in \eqref{eq:H_N_def}, and 
\begin{align*}
	W^{(2,b)}\equiv \left(S_2^{(b)}, \hat V^{(2,b)}, \hat H^{(2,b)},\mathrm{Vec}((\hat{\bm\Sigma}^{(2,b)})^{1/2}), (\hat H^{(2,b)})^{1/2}\right)
\end{align*}	
where $\hat H^{(2,b)}=(\hat H^{(2,b)}(0),\hat H^{(2,b)}(1))$ and $\hat{V}^{(1)}\equiv (\hat{V}^{(1)}(0),\hat{V}^{(1)}(1))$. 

\paragraph{Random functions.} For $x\in\mathbb{R}^{12}$, recalling the definiton of function $h$ in \eqref{eq:def_h_c_function}, the weight function is defined as
\begin{align*}
	\hat{H}^{(2,b)}(s,x)\equiv 
	\begin{cases}
		e(s, h(x,0)) & \text{under Assumption \ref{assu:adaptive_weighting}},\\
		\min\{1-\bar l,\max\{\bar l,e(s, h(x, 0))\}\}& \text{under Assumption \ref{assu:constant_weighting}},
	\end{cases}
\end{align*}
and the variance function is defined as $\hat{V}^{(2,b)}(s,x)\equiv \hat{\E}[Y_{uN}^2(s)]-\hat{H}^{(2,b)}(s,x)(\hat{\E}[Y_{uN}(s)])^2$. By slightly abusing the notation, we define $\hat{H}^{(2,b)}(x)\equiv (\hat{H}^{(2,b)}(0,x),\hat{H}^{(2,b)}(1,x))$, the random variance vector function $\hat{V}^{(2,b)}(x)\equiv (\hat{V}^{(2,b)}(0,x),\hat{V}^{(2,b)}(1,x))$ and the covariance function 
\begin{align*}
	\hat{\mathrm{Cov}}^{(2,b)}(x)\equiv -\frac{(\hat{H}^{(2,b)}(0,x)\hat{H}^{(2,b)}(1,x))^{1/2}}{(\hat{V}^{(2,b)}(0,x)\hat{V}^{(2,b)}(1,x))^{1/2}}\hat{\E}[Y_{uN}(0)]\hat{\E}[Y_{uN}(1)]
\end{align*}
and the covariance matrix function $\hat{\bm \Sigma}^{(2,b)}(x)\equiv (\hat{\mathrm{Cov}}^{(2,b)}(x))_{2\times 2}$. Last, define the function 
\begin{align*}
	W^{(2,b)}(x)
	&
	\equiv \left(S_2^{(b)},\hat{V}^{(2,b)}(x),\hat{H}^{(2,b)}(x),\mathrm{Vec}((\hat{\bm \Sigma}^{(2,b)}(x))^{1/2}),(\hat H^{(2,b)}(x))^{1/2}\right)\\
	W^{(a,b)}(x)
	&
	\equiv \left(Z_2,\hat{V}^{(2,b)}(x),\hat{H}^{(2,b)}(x),\mathrm{Vec}((\hat{\bm \Sigma}^{(2,b)}(x))^{1/2}),(\hat H^{(2,b)}(x))^{1/2}\right),
\end{align*}
where $Z_2$ is defined as in \eqref{eq:mathbb_A_definition}.

\edit{
\subsection{Proof of Theorem \ref{thm:bootstrap}}
}

\begin{proof}[Proof of Theorem \ref{thm:bootstrap}]
	Recall $\mathbb{W}$ as defined in \eqref{eq:mathbb_A_definition}. We use Lemma \ref{lem:sufficient_condition_CMT} to prove the $o_p(1)$ part, with $W_N=(W^{(1,b)},W^{(2,b)})$ and $W=\mathbb{W}$. The continuous function $g$ is chosen to be as in \textbf{Step 3} in Section \ref{sec:proof_roadmap}. Thus it suffices to prove the following statement is true: for any $f$ satisfying condition~\eqref{eq:choice_of_f},
	\begin{align*}
		\E[f(W^{(1,b)},W^{(2,b)})|\mathcal{G}_N]-\E[f(\mathbb{W})]\convp 0.
	\end{align*}
	Under Assumption~\ref{assu:sampling_design}(1), the following statement can be proved under more delicate analyses:
	\begin{align}\label{eq:convergence-bootstrap-Lipschitz-function}
		\E[f(W^{(1,b)},W^{(2,b)})|\mathcal{G}_N]-\E[f(\mathbb{W})]=	O_p(N^{-1/2}).
	\end{align}
	Consider the following decomposition: 
	\begin{align*}
		&
		\E[f(W^{(1,b)},W^{(2,b)})|\mathcal{G}_N]-\E[f(\mathbb{W})]\\
		&
		=\E[f(W^{(1,b)},W^{(2,b)}(W^{(1,b)}))|\mathcal{G}_N]-\E[f(W_1,W_2)]\\
		&
		=\E[f(W_1,W^{(a,b)}(W_1))|\mathcal{G}_N]-\E[f(W_1,W_2)]\\
		&
		\qquad+\E[f(W^{(1,b)},W^{(a,b)}(W^{(1,b)}))|\mathcal{G}_N]-\E[f(W_1,W^{(a,b)}(W_1))|\mathcal{G}_N]\\
		&
		\qquad + \E[f(W^{(1,b)},W^{(2,b)}(W^{(1,b)}))|\mathcal{G}_N]-\E[f(W^{(1,b)},W^{(a,b)}(W^{(1,b)}))|\mathcal{G}_N]\\
		&
		\equiv M_1+M_2+M_3.
	\end{align*}
	It is easy to see that $W^{(2,b)}|W^{(1,b)},\mathcal{G}_N\overset{d}{=}W^{(a,b)}|W^{(1,b)},\mathcal{G}_N$ so we know the following statement is true almost surely:
	\begin{align*}
		\E[f(W^{(1,b)},W^{(2,b)}(W^{(1,b)}))|\mathcal{G}_N]-\E[f(W^{(1,b)},W^{(a,b)}(W^{(1,b)}))|\mathcal{G}_N]=0.
	\end{align*}
	Thus $M_3=0$ and we will prove $M_1,M_2=o_p(1)$ and derive their rates under Assumption~\ref{assu:sampling_design}(1). 
	\begin{enumerate}
		\item \textbf{Proof of $M_1=O_p(N^{-1/2})$.} Notice $(W_1,W_2)\indep \mathcal{G}_N$. Then we have
		\begin{align*}
			M_1=\left|\E[f(W_1,W^{(a,b)}(W_1))|\mathcal{G}_N]-\E[f(W_1,W_2)|\mathcal{G}_N]\right|.
		\end{align*}
		Then by the Lipschitz property and boundedness of $f$, we can bound $M_1\lesssim \E[\left\|W^{(a,b)}(W_1)-W_2\right\|_2]$. In other words, we need to bound, by Lemma \ref{lem:holder_continuity_Frobenius},
		\begin{align*}
			\E[\|\hat{V}^{(2,b)}(W_1)-V^{(2)}\|_2],\ \E[\|\hat{H}^{(2,b)}(W_1)-H^{(2)}\|_2],\ \E[\|\hat{\bm \Sigma}^{(2,b)}-\bm\Sigma^{(2)}\|_{\mathrm{F}}].
		\end{align*}
		Also notice $\|\hat{\bm \Sigma}^{(2,b)}-\bm\Sigma^{(2)}\|_{\mathrm{F}}=\sqrt{2}|\hat{\mathrm{Cov}}^{(2,b)}-\mathrm{Cov}^{(2)}|$. By observing that $\hat{H}^{(2,b)}(W_1)=H^{(2)}$, we know $M_1=O_p(N^{-1/2})$ by the rate of convergence $\WIPW(s)$ and $\WIPWS(s)$, proved in Lemma~\ref{lem:consistency_WIPW}.
		
		\item \textbf{Proof for $M_2$.} Define $W^{(c,b)}\equiv (S_1^{(b)}, V^{(1)}, H^{(1)},\mathrm{Vec}(({\bm\Sigma}^{(1)})^{1/2}),(H^{(1)})^{1/2})$. Since
		\begin{align*}
			(W^{(c,b)},W^{(a,b)}(W^{(c,b)}))|\mathcal{G}_N\overset{d}{=}(W_1,W^{(a,b)}(W_1))|\mathcal{G}_N,
		\end{align*}
		it suffices to work with
		\begin{align*}
			M_2=\left|\E[f(W^{(1,b)},W^{(a,b)}(W^{(1,b)}))|\mathcal{G}_N]-\E[f(W^{(c,b)},W^{(a,b)}(W^{(c,b)}))|\mathcal{G}_N]\right|.
		\end{align*}
		By the Lipschitz property and boundedness of $f$ and Lemma~\ref{lem:holder_continuity_Frobenius}, we can bound  
		\begin{align*}
			M_2
			&
			\lesssim \|W^{(1,b)}-W^{(c,b)}\|_2+\|\hat{V}^{(2,b)}(W^{(1,b)})-\hat{V}^{(2,b)}(W^{(c,b)})\|_2\\
			&
			\qquad+\|\hat{H}^{(2,b)}(W^{(1,b)})-\hat{H}^{(2,b)}(W^{(c,b)})\|_2+\|\hat{\bm \Sigma}^{(2,b)}(W^{(1,b)})-\hat{\bm \Sigma}^{(2,b)}(W^{(c,b)})\|_{\mathrm{F}}.
		\end{align*}
		By the consistency of $\hat{\mathbb E}[Y_{uN}(s)]$ and $\hat{\mathbb E}[Y^2_{uN}(s)]$ proved in Lemma \ref{lem:consistency_WIPW}, we know 
		\begin{align}\label{eq:upper-bound-M-2}
			M_2
			&
			\lesssim \|W^{(1,b)}-W^{(c,b)}\|_2+\frac{1}{N^{1/2}}+\|\hat{H}^{(2,b)}(W^{(1,b)})-\hat{H}^{(2,b)}(W^{(c,b)})\|_2.
		\end{align}
		We consider these terms subsequently. 
		\paragraph{Proof of $\|W^{(1,b)}-W^{(c,b)}\|_2=O_p(N^{-1/2})$:}
		It suffices and it is easy to show
		\begin{align}\label{eq:convergence_V_Sigma_bootstrap}
			\|\hat{V}^{(1)}-V^{(1)}\|_2=O_p(N^{-1/2})\quad\text{and}\quad\|\hat{\bm\Sigma}^{(1)}-\bm\Sigma^{(1)}\|_{\mathrm{F}}=O_p(N^{-1/2}).
		\end{align}
		
		\paragraph{Proof for $\|\hat{H}^{(2,b)}(W^{(1,b)})-\hat{H}^{(2,b)}(W^{(c,b)})\|_2$}

		By the Lipschitz property of $\min\{1-\bar l,\max\{\bar l,x\}\}$, we get the bound: for any $s\in\{0,1\}$,
		\begin{align*}
			|e(s, h(W^{(1,b)},0))-e(s, \mathcal{S}^\infty(((\bm \Sigma^{(1)})^{1/2} S_1^{(b)},V^{(1)}),0))|\lesssim|e(s, h(W^{(1,b)},0))-e(s, h(W^{(c,b)},0))|.
		\end{align*}
		Depending on the smoothness of $e(s,x)$, we divide the proof into two cases based on Assumption \ref{assu:sampling_design}. 
		\begin{itemize}
			\item \textbf{When $e(s,x)$ is Lipschitz continuous in $x$:} We will bound $|h(W^{(1,b)},0)-h(W^{(c,b)},0)|$. By the definition of $h$ in \eqref{eq:def_h_c_function}, it suffices to bound
			\begin{align*}
				\|(\hat{\bm \Sigma}^{(1)})^{1/2} S_1^{(b)}-(\bm \Sigma^{(1)})^{1/2} S_1^{(b)}\|_2=O_p(N^{-1/2})\quad\text{and}\quad\|\hat V^{(1)}-V^{(1)}\|_2=O_p(N^{-1/2}).
			\end{align*}
			This is obvious by result \eqref{eq:convergence_V_Sigma_bootstrap} and 
			\begin{align*}
				\|(\hat{\bm \Sigma}^{(1)})^{1/2} S_1^{(b)}-(\bm \Sigma^{(1)})^{1/2} S_1^{(b)}\|_2
				&
				\leq \|S_1^{(b)}\|_2\|(\hat{\bm \Sigma}^{(1)})^{1/2}-(\bm \Sigma^{(1)})^{1/2}\|_2\\
				&
				\leq \|S_1^{(b)}\|_2\|(\hat{\bm \Sigma}^{(1)})^{1/2}-(\bm \Sigma^{(1)})^{1/2}\|_{\mathrm{F}}\\
				&
				\leq \|S_1^{(b)}\|_2\sqrt{2}\|\hat{\bm \Sigma}^{(1)}-\bm \Sigma^{(1)}\|_{\mathrm{F}}.
			\end{align*}
			Therefore, by result~\ref{eq:convergence_V_Sigma_bootstrap}, we know 
			\begin{align}\label{eq:convergence_S}
				|h(W^{(1,b)},0)-h(W^{(c,b)},0)|=O_p(N^{-1/2}).
			\end{align}
			so that 
			\begin{align}\label{eq:convergence-H-2-b}
				\|\hat{H}^{(2,b)}(W^{(1,b)})-\hat{H}^{(2,b)}(W^{(c,b)})\|_2=O_p(N^{-1/2}).
			\end{align}
		
			\item \textbf{When $e(s,x)=\sum_{k=1}^K c_k\indicator(g(x)\in C_k)$:} It suffices to prove for any $k\in [K]$,
			\begin{align*}
				|\indicator(g(h(W^{(1,b)},0))\in C_k)-\indicator(g(h(W^{(c,b)},0))\in C_k)|\convp 0.
			\end{align*}
			Then with result \eqref{eq:convergence_S}, applying Lemma \ref{lem:continuous_mapping_lem}, we know the claim is true since $g(h(W^{(c,b)},0))$ is a continuous random variable and $\P[g(h(W^{(c,b)},0))\in\partial C_k]=0$.
		\end{itemize}
	\end{enumerate}
	\paragraph{Conclusion for the Lipschitz case.} Under Assumption~\ref{assu:sampling_design}(1), by bound~\eqref{eq:upper-bound-M-2},\eqref{eq:convergence_V_Sigma_bootstrap} and~\eqref{eq:convergence_S}, we know $M_2=O_p(N^{-1/2})$. Thus, we have proved claim~\eqref{eq:convergence-bootstrap-Lipschitz-function}.

\end{proof}

\section{Extenstion}\label{sec:extension}

\subsection{Extension to $m=1$}\label{sec:extension_m_1}

In some adaptive experimental designs, unpromising treatments are dropped in the second stage—a strategy known as the ``drop-the-loser" approach \citep{sampson2005drop,sill2009drop}. In such cases, the sampling probability for one of the treatment arms is set to exactly zero in the second stage. 
Recall Assumptions \ref{assu:constant_weighting} and \ref{assu:adaptive_weighting}; neither allows the sampling probability of any treatment to be exactly zero in the follow-up stage. In this section, we further extend the results in Theorem \ref{thm:weak_convergence_W_N} to incorporate the weighting $m=1$ in $h_{N}^{(t)}(s)=e_N^{m}(s,\mathcal{H}_{t-1})/N^{1/2}$ considered in the general estimator \eqref{eq:WIPW_estimator}. In fact, we can write the resulting test statistic as
\begin{align}\label{eq:WIPW_m_1}
	\WIPW(s)=\frac{\sum_{t=1}^2 N_t \indicator(A_{uN}^{(t)}=s)Y_{uN}^{(t)}}{\sum_{t=1}^2N_te_N(s,\mathcal{H}_t)}.
\end{align}
Consider the following assumption and theorem.

\begin{assumption}[Adaptive weighting with $m=1$]\label{assu:adaptive_weighting_m_1}
	Suppose adpative weighting ($m=1$) is used and clipping rate $l_N=0$ as in Assumption \ref{assu:sampling_design}.
\end{assumption}

\begin{theorem}[Adaptive weighting with $m=1$]\label{thm:weak_convergence_m_1}
	Suppose Assumption \ref{assu:moment_condition}-\ref{assu:sampling_design} and Assumption \ref{assu:adaptive_weighting_m_1} hold. Then, for any $s\in \{0,1\}$, we have $\WIPW(s)-\E[Y_{uN}(s)]=o_p(1)$. Furthermore, define 
	\begin{align*}
		M^{(t)}(s)\equiv q_t\left(\frac{H^{(t)}(s)}{\sum_{t=1}^2 q_t H^{(t)}(s)}\right)^2\quad\text{and}\quad\bar w^{(t)}(s)=\left(M^{(t)}(s)/(R^{(t)}(s))^2\right)^{1/2}.
	\end{align*}
	Then considering the test statistic~\eqref{eq:WIPW_m_1}, we have
	\begin{align*}
		\sqrt{N}(\WIPW(s)-\E[Y_{uN}(s)])
		&
		\convd \sum_{t=1}^2 A^{(t)}(0) \bar w^{(t)}(0)\\
		\sqrt{N}\{T_N-(\E[Y_{uN}(0)]-\E[Y_{uN}(1)])\}
		&
		\convd\sum_{t=1}^2 A^{(t)}(0) \bar w^{(t)}(0) -\sum_{t=1}^2 A^{(t)}(1) \bar w^{(t)}(1).
	\end{align*}
	Defining $w^{(t)}(s)=\bar w^{(t)}(s)/(\sum_{s=0}^1\sum_{t=1}^2 (\bar w^{(t)}(s))^2)^{1/2}$, then we have 
	\begin{align*}
		W_N-\frac{\sqrt{N}(\E[Y_{uN}(0)]-\E[Y_{uN}(1)])}{(N\hat{V}_N(0)+N\hat{V}_N(1))^{1/2}}\convd\sum_{t=1}^2 A^{(t)}(0) w^{(t)}(0) -\sum_{t=1}^2 A^{(t)}(1) w^{(t)}(1).		
	\end{align*}
\end{theorem}

\noindent Theorem \ref{thm:weak_convergence_m_1} demonstrates that $m=1$ weighting can accommodate the early-dropping experiments. We now comment on the proof of Theorem \ref{thm:weak_convergence_m_1}. 

\begin{remark}[Comment on the proof of Theorem~\ref{thm:weak_convergence_m_1}]
	The proof of Theorem~\ref{thm:weak_convergence_m_1} is similar to the proof of Theorem~\ref{thm:weak_convergence_W_N}, following the general proof roadmap sketched in Appendix~\ref{sec:proof_roadmap}. The difference happens in the \textbf{Step 2}, where we need to show the weak convergence of the random vector $(E_N^{(1)},E_{N}^{(2)})$, where
	\begin{align*}
		E_N^{(t)}=\Big(\frac{1}{N_t^{1/2}}\sum_{u=1}^{N_t} \indicator(A_{uN}^{(t)}=s)(Y_{uN}^{(t)}-H_N^{(t)}\E[Y_{uN}^{(t)}]),H_N^{(t)}\Big)\quad\text{and}\quad H_N^{(t)}=e_N(s,\mathcal{H}_{t-1}).
	\end{align*}
	The other proofs are similar.
\end{remark}

\subsection{Extension to test statistics with augmentation}\label{sec:extension_augmentation}

When data is generated independently as in non-adaptive experiments, efficiency of IPW estimator may be improved by augmenting the statistic with a consistent sample mean estimator. In particular, we consider the weighted augmented inverse probability weighted ($\WAIPW$) estimator, $\WAIPW(s)$, defined as
\begin{align}\label{eq:WAIPW}
	\sum_{t=1}^2 \frac{N_t h_{N}^{(t)}(s)}{\sum_{t=1}^2 N_t h_N^{(t)}(s)}\left(\frac{1}{N_t}\sum_{u=1}^{N_t}\frac{\indicator(A_{uN}^{(t)}=s)(Y_{uN}^{(t)}-\hat \E[Y_{uN}(s)])}{e_N(s,\mathcal{H}_{t-1})}+\hat \E[Y_{uN}(s)]\right).
\end{align} 
Furthermore, define $\WIPW_a(s)$ as 
\begin{align*}
	\sum_{t=1}^2 \frac{N_t h_{N}^{(t)}(s)}{\sum_{t=1}^2 N_t h_N^{(t)}(s)}\left(\frac{1}{N_t}\sum_{u=1}^{N_t}\frac{\indicator(A_{uN}^{(t)}=s)(Y_{uN}^{(t)}- \E[Y_{uN}(s)])}{e_N(s,\mathcal{H}_{t-1})}\right)+\E[Y_{uN}(s)].
\end{align*}
It is not hard to show that $\sqrt{N}\left(\WAIPW(s)-\WIPW_a(s)\right)$ converges to $0$ in probability as long as $\hat \E[Y_{uN}(s)]-\E[Y_{uN}(s)]\convp 0$. Then we can apply Theorem~\ref{thm:weak_convergence_W_N} and Theorem~\ref{thm:weak_convergence_m_1} with $Y_{uN}=Y_{uN}-\E[Y_{uN}(s)]$ to prove that $\sqrt{N}(\WIPW_a(s)-\E[Y_{uN}(s)])$ converges weakly. Similarly, we can show the weak limits of $T_N$ and $W_N$.

\paragraph{Asymptotic equivalence between $\WAIPW(s)$ and sample mean.}

We can show that when $m=1$, the $\WAIPW(s)$ is asymptotically equivalent to the sample mean estimator, $\mathrm{SM}(s)$, defined as
\begin{align*}
	\mathrm{SM}(s)\equiv\frac{\sum_{t=1}^2\sum_{u=1}^{N_t}\indicator(A_{uN}^{(t)}=s)Y_{uN}^{(t)}}{\sum_{t=1}^2\sum_{u=1}^{N_t}\indicator(A_{uN}^{(t)}=s)}.
\end{align*}
By convention, we define $0/0=1$. We formalize the equivalence claim in the following lemma.

\begin{lemma}[Asymptotic equivalence between sample mean and $\WAIPW(s)$]\label{lem:sample_mean_test_statistic}
	Suppose $m=1$ and $\hat\E[Y_{uN}(s)]-\E[Y_{uN}(s)]\convp 0$. Furthermore, suppose Assumption~\ref{assu:moment_condition} holds. Then we have $\sqrt{N}\left(\mathrm{SM}(s)-\WAIPW(s)\right)\convp 0$.
\end{lemma}
\noindent The proof of Lemma~\ref{lem:sample_mean_test_statistic} can be found in Appendix~\ref{sec:proof_sample_mean_equivalence}.

\subsection{Extension to selection with nuisance parameter}\label{sec:extension_nuisance}

In practice, the selection algorithm can depend on statistic beyond $S_{N}^{(1)}(0)-S_{N}^{(1)}(1)$. The example includes the case when the selection depends on the interim $p$-value, which involves the standard deviation estimate, going beyond the difference of the two statistics. The following theorem shows that our results can be extended to such case.

\begin{theorem}[Extension to selection algorithm with nuisance parameter]\label{thm:weak_convergence_W_N_nuisance}
	Suppose the follow-up stage sampling probability is given by
	\begin{align}\label{eq:adaptive_sampling_nuisance}
		\mathcal{S}_N(S_N^{(1)}(0)-S_N^{(1)}(1))=\min\{1-l_N,\max\{l_N, e(0, (S_N^{(1)}(0)-S_N^{(1)}(1))/\hat\sigma)\}\},
	\end{align}
	where $\hat\sigma\in(0,\infty)$. Then if there exists $\sigma>0$ such that $\hat\sigma\convp \sigma\in(0,\infty)$, then the conclusion of Theorem~\ref{thm:weak_convergence_W_N} still holds.
\end{theorem}
\paragraph{Proof sketch for Theorem~\ref{thm:weak_convergence_W_N_nuisance}.}
The proof is similar to the proof of Theorem~\ref{thm:weak_convergence_W_N}. In particular, it follows the general proof roadmap sketched in Appendix~\ref{sec:proof_roadmap}. We need to choose the appropriate joint vector $E_N^{(t)}$ to accommodate the nuisance parameter $\hat\sigma$. In particular, keeping $E_N^{(1)}$ as defined in~\eqref{eq:E_N_t_def}, we can define the new $E_N^{(2)}$ as
\begin{align*}
	E_N^{(2)}=\left((\bm \Sigma_N^{(2)})^{-1/2}\Lambda_{N}^{(2)}, V_N^{(2)},H_N^{(2)},\mathrm{Vec}((\bm \Sigma_N^{(2)})^{1/2})\right)
\end{align*}
where $H_N^{(2)}$ is defined as in~\eqref{eq:adaptive_sampling_nuisance}. The other proofs are similar.

\subsection{Extension to adaptive experiments with stopping time}\label{sec:extension_stopping_time}

We outline how our main results can be extended to adaptive experiments involving a stopping time. Specifically, we consider a stopping time $\tau$ that depends on the quantity $S_N^{(1)}(0) - S_N^{(1)}(1)$. To describe the stopping criterion, we define the event $\mathcal{E} \equiv \{ D(S_N^{(1)}(0) - S_N^{(1)}(1)) \in [0,\beta] \subset \mathbb{R} \}$, where $D$ is a decision rule and $\beta >0$ lies on the decision boundary. The stopping time is then defined as $\tau \equiv \indicator(\mathcal{E})$, where $\tau = 1$ indicates continuation of the experiment and $\tau = 0$ indicates early termination. This setup captures scenarios in which strong preliminary evidence warrants stopping the experiment at the pilot stage. For example, if $D(x) = x$, then $\tau = 0$ when the evidence in favor of treatment 0 over treatment 1 is sufficiently strong, if $\beta$ is large. Then the test statistic can be written as 
\small
\begin{align}\label{eq:stopping_time_test_statistic}
	\WIPW(s)=\frac{N_1 h_{N}^{(1)}(s)}{N_1 h_{N}^{(1)}(s)+N_2 h_{N}^{(2)}(s)\indicator(\mathcal{E})} \hat{\Lambda}_{N}^{(1)}(s)+\frac{N_2 h_{N}^{(2)}(s)\indicator(\mathcal{E})}{N_1 h_{N}^{(1)}(s)+N_2 h_{N}^{(2)}(s)\indicator(\mathcal{E})} \hat{\Lambda}_{N}^{(2)}(s).
\end{align}
\normalsize

\paragraph{Sketched derivation of weak limit.}

Weak limit $\WIPW(s)$ defined in~\eqref{eq:stopping_time_test_statistic} can be derived following the similar proof roadmap sketched in Appendix~\ref{sec:proof_roadmap}. In particular, we will need to derive the joint weak convergence of the random vector including $\hat \Lambda_N^{(1)}(s),\hat \Lambda_N^{(2)}(s),h_{N}^{(1)}(s), h_{N}^{(2)}(s)\indicator(\mathcal{E})$.

\subsection{Proof of Lemma~\ref{lem:sample_mean_test_statistic}}\label{sec:proof_sample_mean_equivalence}

\begin{proof}[Proof of Lemma~\ref{lem:sample_mean_test_statistic}]
	When $m=1$, we can write the $\WAIPW(s)$ as
	\begin{align*}
		\WAIPW(s)=\frac{\sum_{t=1}^2\sum_{u=1}^{N_t}\indicator(A_{uN}^{(t)}=s)(Y_{uN}^{(t)}-\hat\E[Y_{uN}(s)])}{\sum_{t=1}^2 N_te_N(s,\mathcal{H}_{t-1})}+\hat\E[Y_{uN}(s)].
	\end{align*}
	For the ease of notation, define 
	\begin{align*}
		\mathcal{I}_{N}(s)\equiv \sum_{t=1}^2\sum_{u=1}^{N_t}\indicator(A_{uN}^{(t)}=s)\quad\text{and}\quad \mathcal{Y}_N(s)\equiv \sum_{t=1}^2\sum_{u=1}^{N_t}\indicator(A_{uN}^{(t)}=s)Y_{uN}^{(t)}.
	\end{align*}
	Also, define 
	\begin{align*}
		\mathcal{R}_N(s)=\mathcal{Y}_N(s)-\sum_{t=1}^2\sum_{u=1}^{N_t}\indicator(A_{uN}^{(t)}=s)\E[Y_{uN}^{(t)}].
	\end{align*}
	Then we can write the difference as
	\begin{align*}
		\mathrm{SM}(s)-\WAIPW(s)
		&
		=\left(\frac{\mathcal{Y}_N(s)}{\mathcal{I}_{N}(s)}-\hat\E[Y_{uN}(s)]\right)\left(1-\frac{\mathcal{I}_{N}(s)}{\sum_{t=1}^2 N_te_N(s,\mathcal{H}_{t-1})}\right)\\
		&
		\equiv A_N(s)\times B_N(s).
	\end{align*}
	It suffices to prove $A_N(s)=o_{p}(1)$ and $B_N(s)=O_{p}(1/\sqrt{N})$. 
	\paragraph{Proof of $A_N(s)=o_{p}(1)$.} By the definition of $\mathcal{I}_{N}(s)$, we can write
	\begin{align*}
		A_N(s)=\frac{\mathcal{R}_N(s)}{\mathcal{I}_N(s)}+\E[Y_{uN}(s)]-\hat\E[Y_{uN}(s)]\equiv A_{1N}(s)+\E[Y_{uN}(s)]-\hat\E[Y_{uN}(s)].
	\end{align*}
	Since by the assumption, $\E[Y_{uN}(s)]-\hat\E[Y_{uN}(s)]=o_p(1)$, it suffices to prove that $\mathrm{Var}[A_{1N}(s)]=o(1)$. In fact, we can show something stronger: $\mathrm{Var}[\sqrt{N}A_{1N}(s)]=O(1)$. Since the denominator $\mathcal{I}_N=\sum_{t=1}^2\sum_{u=1}^{N_t}\indicator(A_{uN}^{(t)}=s)$ is likely to be $0$, we divide the proof into three cases. Define $\mathcal{I}^{(t)}_N(s)\equiv \sum_{u=1}^{N_t}\indicator(A_{uN}^{(t)}=s)$.
	\begin{enumerate}
		\item \textbf{When $\mathcal{I}_N(s)=0$.} In this case, we have $A_{1N}(s)=1$. However, this is event is exponentially unlikely since $\mathcal{I}_N(s)\geq \mathcal{I}^{(1)}_N(s)=\sum_{u=1}^{N_1}\indicator(A_{uN}^{(1)}=s)$ and $\E[\mathcal{I}^{(1)}_N(s)]=N_1e(s)$. Therefore, $\mathrm{Var}[\sqrt{N}A_{1N}(s)\indicator(\mathcal{I}_{N}(s)=0)]\rightarrow0$.
		\item \textbf{When $\mathcal{I}_N(s)>0$ but $\mathcal{I}_{N}^{(1)}(s)=0$.} In this case, we have $|A_{1N}(s)|\leq |Y_{uN}^{(2)}(s)-\E[Y_{uN}(s)]|$. Then we have 
		\begin{align*}
			\mathrm{Var}[\sqrt{N}A_{1N}(s)\indicator(\mathcal{I}_{N}^{(1)}(s)=0)]
			&
			\leq \E[NA_{1N}^2(s)\indicator(\mathcal{I}_{N}^{(1)}(s)=0)]\\
			&
			= N\mathrm{Var}[Y_{uN}^{(2)}(s)]\P[\mathcal{I}_{N}^{(1)}(s)=0].
		\end{align*} 
		Since $\P[\mathcal{I}_{N}^{(1)}(s)=0]\rightarrow 0$ exponentially, we have $\mathrm{Var}[\sqrt{N}A_{1N}(s)\indicator(\mathcal{I}_{N}^{(1)}(s)=0)]\rightarrow 0$.
		\item \textbf{When $\mathcal{I}_{N}^{(1)}(s)>0$.} We compute 
		\small
		\begin{align*}
			\mathrm{Var}[A_{1N}(s)\indicator(\mathcal{I}_{N}^{(1)}(s)>0)]\leq \E\left[\frac{\left(\mathcal{R}_N(s)\right)^2}{(\mathcal{I}_N(s))^2}\indicator(\mathcal{I}_{N}^{(1)}(s)>0)\right].
		\end{align*}
		\normalsize
		Since $\mathcal{I}_N(s)\geq \mathcal{I}_N^{(1)}(s)= \sum_{u=1}^{N_1}\indicator(A_{uN}^{(1)}=s)$, we know 
		\begin{align*}
			\E\left[\frac{\left(\mathcal{R}_N(s)\right)^2}{(\mathcal{I}_N(s))^2}\indicator(\mathcal{I}_{N}^{(1)}(s)>0)\right]
			&
			\leq \E\left[\frac{\left(\mathcal{R}_N(s)\right)^2}{(\mathcal{I}_N^{(1)}(s))^2}\indicator(\mathcal{I}_{N}^{(1)}(s)>0)\right]\\
			&
			= \E\left[\frac{\E\left[\left(\mathcal{R}_N(s)\right)^2|\mathcal{H}_1\right]}{(\sum_{u=1}^{N_1}\indicator(A_{uN}^{(1)}=s))^2}\indicator(\mathcal{I}_{N}^{(1)}(s)>0)\right].
		\end{align*}
		Further, we can decompose 
		\small
		\begin{align*}
			\E\left[\left(\mathcal{R}_N(s)\right)^2|\mathcal{H}_1\right]
			&
			=\sum_{u=1}^{N_1}\indicator(A_{uN}^{(1)}=s)(Y_{uN}^{(1)}-\E[Y_{uN}(s)])^2 + N_2e_N(s,\mathcal{H}_1)\mathrm{Var}[Y_{uN}(s)]\\
			&
			\leq \sum_{u=1}^{N_1}(Y_{uN}^{(1)}-\E[Y_{uN}(s)])^2 + N_2e_N(s,\mathcal{H}_1)\mathrm{Var}[Y_{uN}(s)].
		\end{align*}
		\normalsize
		Then define
		\begin{align*}
			A_{2N}(s)\equiv \E\left[\frac{\sum_{u=1}^{N_1}(Y_{uN}^{(1)}-\E[Y_{uN}(s)])^2}{(\sum_{u=1}^{N_1}\indicator(A_{uN}^{(1)}=s))^2}\indicator(\mathcal{I}_{N}^{(1)}(s)>0)\right]
		\end{align*}
		and 
		\begin{align*}
			A_{3N}(s)\equiv \E\left[\frac{N_2\indicator(\mathcal{I}_{N}^{(1)}(s)>0)}{(\sum_{u=1}^{N_1}\indicator(A_{uN}^{(1)}=s))^2}\right]\mathrm{Var}[Y_{uN}(s)].
		\end{align*}
		It suffices to prove that $A_{2N}(s)=O(1/N)$ and $A_{3N}(s)=O(1/N)$. We first prove the claim for $A_{2N}(s)$. By Cauchy-Schwarz inequality, we have
		\begin{align*}
			(A_{2N}(s))^2
			&
			\leq \E\left[\left(\frac{\sum_{u=1}^{N_1}(Y_{uN}^{(1)}-\E[Y_{uN}(s)])^2}{N_1}\right)^2\right]\E\left[\left(\frac{N_1\indicator(\mathcal{I}_{N}^{(1)}(s)>0)}{(\sum_{u=1}^{N_1}\indicator(A_{uN}^{(1)}=s))^2}\right)^2\right]\\
			&
			\leq \E[(Y_{uN}^{(1)}-\E[Y_{uN}(s)])^4]\E\left[\left(\frac{N_1\indicator(\mathcal{I}_{N}^{(1)}(s)>0)}{(\sum_{u=1}^{N_1}\indicator(A_{uN}^{(1)}=s))^2}\right)^2\right]\\
			&
			\leq \E[(Y_{uN}^{(1)}-\E[Y_{uN}(s)])^4]\E\left[\left(\frac{4N_1}{(1+\sum_{u=1}^{N_1}\indicator(A_{uN}^{(1)}=s))^2}\right)^2\right].
		\end{align*}
		For the first term, by Assumption~\ref{assu:moment_condition}, we have $\E[(Y_{uN}^{(1)}-\E[Y_{uN}(s)])^4]=O(1)$. For the second term, we use the following lemma to conclude the proof. 
		\begin{lemma}[\citet{cribari2000note}]\label{lem:inverse_binomial_moment}
			Suppose $X_1,\ldots,X_N$ are i.i.d. Bernoulli random variables with $\E[X_i]=p$. Then we have 
			\begin{align*}
				\E\left[\frac{N^k}{(1+\sum_{i=1}^N X_i)^k}\right]=O\left(1/p^k\right).
			\end{align*}
		\end{lemma}
		Now we can apply Lemma~\ref{lem:inverse_binomial_moment} with $k=4$ so that we have $A_{2N}(s)=O(1/N_1)=O(1/N)$. Similarly, $A_{3N}(s)=O(1/N)$. Thus $\mathrm{Var}[\sqrt{N}A_{1N}(s)\indicator(\mathcal{I}_{N}^{(1)}(s)>0)]\rightarrow 0$.
	\end{enumerate}
	This concludes the proof of $A_N(s)=o_{p}(1)$. 
	\paragraph{Proof of $B_N(s)=O_{p}(1/\sqrt{N})$.} 
	We will prove that $\mathrm{Var}[\sqrt{N}B_N(s)]=O(1)$. The proof follows the similar argument as the proof of $\mathrm{Var}[\sqrt{N}A_{1N}(s)]=O(1)$. We omit it.
\end{proof}

\subsection{Power comparison: $m=1/2$ versus $m=1$}\label{sec:power-comparison}

We have observed from Figure \ref{fig:simulation-rejection-plot-thompson} that the adaptive weighting with $m=1/2$ is more powerful than the constant weighting ($m=0$). The natural question is then to ask which adaptive weighting scheme, $m=1$ or $m=1/2$, is better in terms of power. To investigate this question, we conduct a set of simulation with both unnormalized and normalized tests and the simulation invovles $4$ common distributions: Gaussian, Bernoulli, Poisson and Student distributions. We consider the Thompson sampling \eqref{eq:modified-cliped-TS} with $l_N=0.2$ and consider the sampel size $N=20,000$ and $N_1=N_2=10,000$. We set the significance level to be $0.05$. The distribution information can be summarized as below:
\begin{itemize}
	\item \textbf{Gaussian:} $Y_{u}(0)\sim N(\theta,1),Y_{u}(1)\sim N(0,0.25)$.
	\item \textbf{Bernoulli:} $Y_{u}(0)\sim\mathrm{Bern}(0.5+\theta),Y_{uN}(1)\sim\mathrm{Bern}(0.5)$.
	\item \textbf{Poisson:} $Y_{u}(0)\sim \mathrm{Pois}(1+\theta),Y_{u}(1)\sim \mathrm{Pois}(1)$.
	\item \textbf{Student:} $Y_{u}(0)\sim \theta+\mathrm{t}(4),Y_{u}(1)\sim \mathrm{t}(10)$ where $4$ and $10$ are degrees of freedom in the student distributions.
\end{itemize}
In particular, we choose $\theta\in\{0, 0.01, 0.02, 0.03,0.04\}$. The results are presented in Figure \ref{fig:power_comparison}.

\begin{figure}[!ht]
	\centering
	\includegraphics[width=0.99\textwidth]{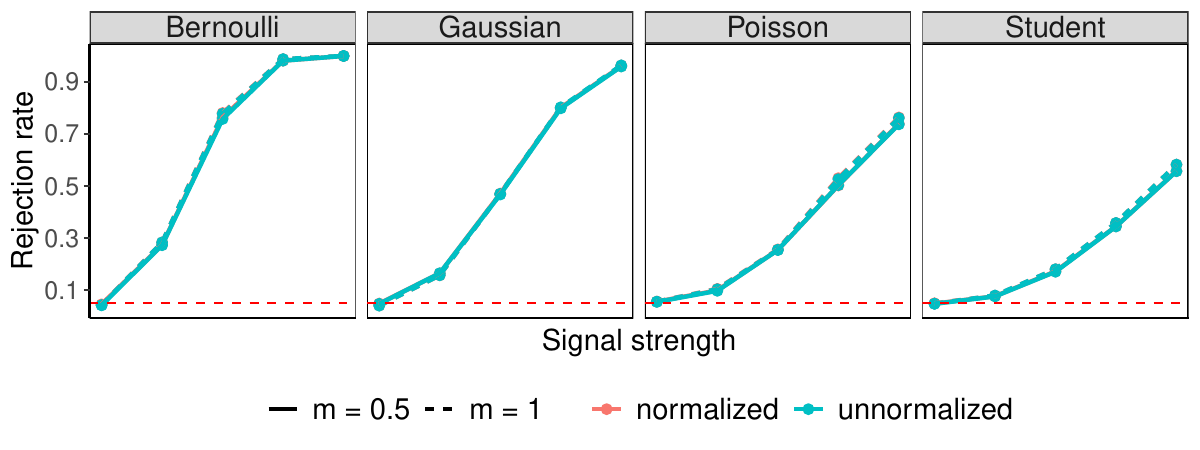}
	\caption{Rejection plots for adaptive weighting with $m=1$ and $m=1/2$ on the centered data $(\tilde Y_{uN}(0),\tilde Y_{uN}(1))$.}
	\label{fig:power_comparison}
\end{figure}

\paragraph{Implementation details.}

We center the outcome $Y_{uN}(s)$ by generating $\tilde Y_{uN}(s)=Y_{uN}(s)-\E[Y_{uN}(1)]$ and compare the test power for the average treatment effect $\E[\tilde Y_{uN}(0)]-\E[\tilde Y_{uN}(1)]$. The motivation for this operation is we do not want the test to be affected by the absolute signal strength. Asymptotically, such operation is equivalent to use $\WAIPW(s)$, defind as in~\eqref{eq:WAIPW}, when testing with the original data $(Y_{uN}(0),Y_{uN}(1))$. This is because the magnitude of $\theta$ is very small. Also, we want to point out that when $\E[\tilde Y_{uN}(s)]\sim 1/\sqrt{N}$ and $\E[\tilde Y_{uN}(1)]=0$ statistic $\WIPW(s)$ with $m=1$ is asymptotically equivalent to the sample mean, as proved in Lemma~\ref{lem:sample_mean_test_statistic}. In other words, the asymptotic power function for the sample mean test statistic is the same as the power function for the $m=1$ weighting when testing is performed on the transformed data $(\tilde Y_{uN}(0), \tilde Y_{uN}(1))$. 

\paragraph{Interpretation of the results.}

For all the other setups, it seems both $m=1$ and $m=1/2$ have very similar power performance. This means we would expect the sample mean test statistic should also have very comparable power performance with the weighting $m=1/2$. It is generally unclear if the test statistics considered in this paper are optimal or not under these distributions. As an exception, the sample mean test stiatistic has been shown to be near-optimal in some Gaussian setup, shown in Section 4.4 of \citet{Hirano2023}.

\end{document}